\documentclass{article}
\usepackage{authblk}
\usepackage{ytableau}
\usepackage{geometry}

\usepackage{tikz}
\usetikzlibrary{arrows.meta,positioning,calc,fit,decorations.pathreplacing}
\usepackage{a4wide}
\usepackage[english]{babel}
\usepackage[T1]{fontenc}
\usepackage[latin1]{inputenc}
\usepackage{dsfont}
\usepackage{mathrsfs}
\usepackage{amsmath}
\usepackage{bbm}
\usepackage{amssymb}
\usepackage{verbatim}
\usepackage{amsthm}
\usepackage{graphicx}
\usepackage{xcolor}
\usepackage{fancybox}
\usepackage[square,numbers]{natbib}
\usepackage{enumitem}
\usepackage{eurosym}
\usepackage{tabularx,booktabs,array}
\usepackage[colorlinks=true,linkcolor=blue!60!black,citecolor=blue!60!black,urlcolor=blue!60!black]{hyperref}
\usepackage[nameinlink,capitalize,noabbrev]{cleveref}
\allowdisplaybreaks
\newtheorem{theorem}{Theorem}[section]
\newtheorem{prop}[theorem]{Proposition}
\newtheorem{lemma}[theorem]{Lemma}
\newtheorem{assumption}[theorem]{Assumption}
\newtheorem{corollary}[theorem]{Corollary}
\newtheorem{definition}{Definition}
\newtheorem{remark}{Remark}

\newtheoremstyle{named}{}{}{\bfshape}{}{\bfseries}#1{\thmnote{#3 }}
\theoremstyle{named}

\newtheoremstyle{named1}{}{}{\bfshape}{}{\bfseries}#1 {\thmnote{#3 }}
\theoremstyle{named1}

\newcommand{\R}{\mathbb{R}}
\newcommand{\N}{\mathbb{N}}
\newcommand{\E}{\mathbb{E}}
\newcommand{\Pp}{\mathbb{P}}
\newcommand{\Tr}{\operatorname{Tr}}
\newcommand{\Var}{\operatorname{Var}}
\newcommand{\Ree}{\operatorname{Re}}
\newcommand{\Img}{\operatorname{Im}}
\newcommand{\diag}{\operatorname{diag}}

\newcommand{\WSS}{\mathrm{WSS}}
\newcommand{\Mult}{\mathrm{Mult}}

\newcommand{\btheta}{\boldsymbol{\theta}}
\newcommand{\bvartheta}{\boldsymbol{\vartheta}}

\newcommand{\ba}{\boldsymbol{a}}
\newcommand{\bone}{\boldsymbol{1}}
\newcommand{\cE}{\mathcal{E}}
\newcommand{\cI}{\mathcal{I}}

\newcommand{\bm}{\mathbf}

\newcommand{\C}{\mathbb C}

\newcommand{\bbC}{\mathbb C}
\newcommand{\bbR}{\mathbb R}

\newcommand{\ot}{\otimes}

\newcommand{\cF}{\mathcal F}
\newcommand{\cQ}{\mathcal Q}
\newcommand{\cM}{\mathfrak M}

\newcommand{\cR}{\mathcal R}
\newcommand{\cA}{\mathcal A}

\newcommand{\cH}{\mathcal H}
\newcommand{\cK}{\mathcal K}
\newcommand{\cU}{\mathcal U}
\newcommand{\cV}{\mathcal V}
\newcommand{\cS}{\mathcal S}

\newcommand{\BayesRisk}{\mathcal R_{\Pi}}

\newcommand{\Young}{\mathbb Y_{n,d}}
\newcommand{\Haar}{\mu_{\mathrm H}}
\newcommand{\spec}{\operatorname{Spec}}

\newcommand{\cT}{\mathcal{T}}

\usepackage{stmaryrd}
\usepackage{cases}

\newcommand{\argmin}{\mathop{\rm arg~min}\limits}

\makeatother

\theoremstyle{plain}

\theoremstyle{remark}

\providecommand{\notename}{Note}
\providecommand{\theoremname}{Theorem}

\title{Classical Sufficiency in Quantum Statistical Experiments}
\author{
Samriddha Lahiry\\
Department of Statistics and Data Science\\
National University of Singapore\\
\href{mailto:slahiry@nus.edu.sg}{\texttt{slahiry@nus.edu.sg}}
}

\date{}

\begin{document}

\maketitle

\begin{abstract}
A quantum measurement transforms a quantum statistical experiment into a classical one, but different measurements generally yield experiments containing different amounts of statistical information. We introduce quantum-to-classical sufficiency: relative to a prescribed class of admissible measurements, a measurement is sufficient if its induced classical experiment Blackwell-dominates the experiments induced by every other measurement in the class. This framework separates two components of quantum inference: reducing the admissible measurement class through symmetry or decision-theoretic arguments, and identifying a canonical measurement within the reduced class through an experiment-level sufficiency condition.
We establish a factorization criterion under which the parameter enters the group-averaged experiment only through the weights of an orthogonal block decomposition. Measurement of the block label is then sufficient relative to all invariant measurements.

Applied to tensor-product models under unitary conjugation, this result identifies weak Schur sampling as the sufficient measurement for spectral inference. Combined with risk-preserving symmetrization, it yields exact finite-sample reductions of invariant Bayes and minimax problems over all measurements to classical decision problems based on Young diagrams. We use this reduction to derive sharp first-order asymptotic Bayes and minimax risks for estimating smooth spectral functionals.

We establish an analogous reduction for a noncompact group action in a continuous-variable setting. Specifically, we consider displaced thermal states and the estimation of functionals of the thermal parameter in the presence of an unknown displacement. Within the class of displacement-invariant measurements, total residual photon number measurement emerges as the sufficient measurement. Combined with Hunt-Stein reduction, this yields exact finite-sample reductions of Bayes-minimax and minimax problems over all measurements to classical decision problems based on the total residual photon count, which follows a negative binomial distribution.
\end{abstract}

\tableofcontents

\newpage

\section{Introduction}

A distinctive feature of statistical inference for quantum systems is that the
observation itself is part of the statistical procedure. In a classical
experiment, the sampling distribution is specified before a decision rule is
chosen. In a quantum experiment
$\mathcal Q=\{\rho_\theta:\theta\in\Theta\}$, by contrast, the state
$\rho_\theta$ does not determine a unique classical observation: a
measurement must first be selected. If $M$ is a positive operator-valued
measure (POVM), measuring $\rho_\theta$ with $M$ produces a probability
distribution $P_\theta^M$ and hence the classical statistical experiment

$$
    \mathcal E_M
    :=
    \{P_\theta^M:\theta\in\Theta\}.
$$

A quantum decision procedure therefore consists of two components: a
measurement and a classical decision rule applied to its outcome. Optimal
inference may require joint optimization over both components, a basic feature
of quantum statistical decision theory; see, for example,
\cite{Helstrom,Holevo2011}.

The purpose of this paper is to identify situations in which these two
optimizations can be separated. We formulate this question using the
comparison theory of statistical experiments developed by Blackwell and
Le Cam
\cite{Blackwell1951,Blackwell1953,LeCam1964,LeCam1996,LeCam1986,Torgersen1991}.
Recall that a classical experiment $\mathcal E$ Blackwell-dominates another
experiment $\mathcal F$ (in compact notation $\mathcal E
    \succeq_{\mathrm B}
    \mathcal F$) if the observations from $\mathcal F$ can be
generated from those of $\mathcal E$ by a parameter-independent Markov
kernel. Under the classical randomization theorem, this means that every
decision procedure based on $\mathcal F$ can be reproduced from
$\mathcal E$ with the same risk function.

Motivated by this comparison, we formulate a notion of sufficiency in the class of experiments generated by different POVMs. Let $\mathfrak M$ be a prescribed class of
admissible measurements. We call $M^\star\in\mathfrak M$
\emph{quantum-to-classically sufficient} for $\mathcal Q$ relative to
$\mathfrak M$ if

$$
    \mathcal E_{M^\star}
    \succeq_{\mathrm B}
    \mathcal E_M,
    \qquad
    M\in\mathfrak M.
$$

Thus $\mathcal E_{M^\star}$ is a greatest element, up to Blackwell
equivalence, among the classical experiments induced by measurements in
$\mathfrak M$. Equivalently, the outcome distribution of every admissible
measurement can be generated from that of $M^\star$ by a
parameter-independent classical randomization. This notion is stronger than
optimality of $M^\star$ for a particular decision theoretic problem:
once $M^\star$ has been measured, every classical decision procedure
available from any measurement in $\mathfrak M$ can be reproduced by
classical post-processing.

The restriction to a prescribed measurement class is essential. In many
quantum statistical problems, a symmetry argument first
allows optimization over arbitrary quantum procedures to be restricted,
without loss for the decision problem under consideration, to an invariant measurement class. Quantum-to-classical sufficiency addresses
a second and logically distinct question: whether one measurement within
that class dominates all the others at the level of the induced statistical
experiment. The resulting two-stage reduction may be summarized as

$$
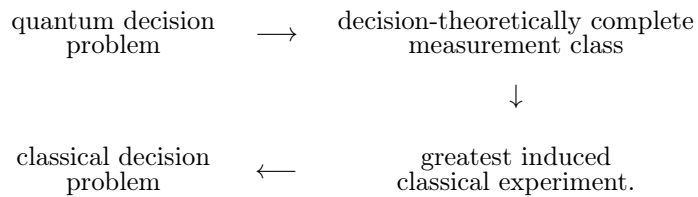
\begin{array}{ccc}
\begin{array}{c}
\text{quantum decision}\\[-1mm]
\text{problem}
\end{array}
&
\longrightarrow
&
\begin{array}{c}
\text{decision-theoretically complete}\\[-1mm]
\text{measurement class}
\end{array}
\\[4mm]
&&
\downarrow
\\[4mm]
\begin{array}{c}
\text{classical decision}\\[-1mm]
\text{problem}
\end{array}
&
\longleftarrow
&
\begin{array}{c}
\text{greatest induced}\\[-1mm]
\text{classical experiment}.
\end{array}
\end{array}
$$

The first step can depend on the parameter of interest, the loss, the prior,
and the treatment of nuisance parameters. The second is an experiment-level
statement relative to the reduced measurement class and is independent of
the subsequent classical decision rule.

\subsection{Main contributions}

Our first contribution is a structural criterion for the second reduction.
Suppose that, in an invariant quantum statistical model, symmetrization
yields a family of states $\{\bar\rho_\eta\}$ that reproduces the outcome
distributions of all invariant measurements and admits an orthogonal block
decomposition of the form
\[
    \bar\rho_\eta
    =
    \bigoplus_z p_\eta(z)\tau_z.
\]
Here $\eta$ denotes the invariant parameter, which enters only through the
block probabilities $p_\eta(z)$, while the conditional block states
$\tau_z$ are independent of $\eta$.
 We show that measurement of the block label is then
quantum-to-classically sufficient relative to all invariant POVMs. The
reason is the same as in the classical factorization principle: conditional
on the observed block, all remaining randomness is parameter-free, so every
invariant measurement is obtained by a parameter-independent classical
randomization of the block label.

We also investigate the converse. Sufficiency of the block-label measurement
always implies a corresponding factorization of expectations of certain invariant observables. In finite dimension, under a block-coordinate identifiability
condition, this observable-side factorization can be upgraded to the
parameter-free state decomposition above. Thus, subject to the stated
identifiability condition, the existence of parameter-free conditional block
states characterizes block-label sufficiency. This gives a quantum analogue,
relative to an invariant measurement class, of the classical
characterization of sufficiency through parameter-free conditional
distributions.

Our first application is spectral inference from $n$ identical copies
of a fixed-dimensional quantum state. Under unitary conjugation, the
eigenvectors are nuisance parameters, while the spectrum is invariant.
Schur-Weyl duality supplies a block decomposition indexed by Young
diagrams. We show that the corresponding block-label measurement,
known as \emph{weak Schur sampling (WSS)}, is
quantum-to-classically sufficient relative to unitarily invariant
POVMs. For a unitarily invariant prior and a spectrum-only loss, this
sufficiency result, combined with symmetrization, reduces the Bayes
problem over all measurements exactly to a classical
Bayes problem based on the observed Young diagram. A corresponding
argument gives the exact finite-sample minimax reduction for general
spectrum-only losses.

We then exploit this exact reduction to study the asymptotic risk of smooth
spectral functionals under the law of the weak Schur experiment. For absolutely continuous priors satisfying mild
regularity conditions, we obtain the sharp first-order Bayes risk, with
leading term given by a corresponding multinomial information bound.
Combining this calculation with the exact finite-sample reduction gives the
same first-order expansion for the optimal Bayes risk over all collective
POVMs. We also establish the corresponding minimax result uniformly over
compact subsets of the strictly positive simple-spectrum region and show that the
natural plug-in estimator based on the normalized Young diagram is
asymptotically minimax.
The proof combines an exact randomization from the
multinomial experiment to WSS, a quantitative comparison between the
Schur-Weyl and multinomial distributions on regular spectral sets, and
classical posterior and risk asymptotics. 

Our second application concerns a continuous-variable model with a
noncompact nuisance symmetry. We consider displaced thermal states, where
the mean photon number, or a smooth functional of it, is the parameter of
interest and a common displacement is an unknown nuisance parameter. A unitary transformation concentrates the displacement into a single
collective mode and leaves the remaining relative modes in centered thermal
states. After this transformation, every displacement-invariant measurement acts
trivially on the collective mode, so inference within this class reduces
to the relative modes.

The centered relative-mode experiment has a parameter-free conditional block
decomposition indexed by total residual photon number. Consequently, measurement of
the \emph{total residual photon number} is quantum-to-classically sufficient
relative to all collective-displacement-invariant POVMs. The resulting
classical observation has a negative-binomial distribution. Thus every
decision procedure based on an invariant quantum measurement can
be reproduced by a parameter-independent classical randomization of a single
photon-count observation.

To connect this experiment-level statement with the original nuisance-robust
decision problem, we establish the noncompact symmetrization result needed
for general action-valued procedures.
Averaging over asymptotically invariant
probability measures on the displacement group yields an invariant procedure
without increasing the maximal risk. Combining this reduction with
quantum-to-classical sufficiency gives an exact finite-sample equivalence
between the quantum minimax problem and an ordinary negative-binomial
decision problem. Since the displacement group admits no proper
translation-invariant probability measure, we consider a prior on the thermal parameter while treating the
displacement in the minimax sense; this yields an exact Bayes-minimax
reduction as well. Standard analysis of the resulting one-dimensional
exponential family then gives the sharp first-order Bayes-minimax and
minimax risks for smooth thermal functionals.

The two applications have different physical and mathematical origins but
exhibit the same statistical mechanism. In spectral inference the symmetry
group is compact, the sufficient classical observation is a Young diagram,
and the asymptotic analysis of its distribution is nontrivial. In the
displaced-thermal problem the nuisance group is noncompact, concentration
isolates the nuisance mode, and the sufficient classical experiment is the negative-binomial family. In both cases symmetry first identifies
a decision-theoretically complete class of measurements, after which a
parameter-free conditional factorization identifies a single measurement
that dominates the entire class. The subsequent statistical optimization is
then classical.

\subsection{Relation to existing notions and literature}
\label{subsec:related-literature}

The comparison used in this paper is classical after a measurement has been
performed. Blackwell and Le Cam compare classical statistical experiments
through parameter-independent randomizations and their decision-theoretic
consequences
\cite{Blackwell1951,Blackwell1953,LeCam1964,LeCam1996,LeCam1986,Torgersen1991}.
Quantum comparison theory instead asks whether one quantum statistical
experiment can be transformed into another by a quantum channel, statistical
morphism, or related randomization
\cite{Buscemi2012,Matsumoto2012,Jencova2016}. These are different comparison
problems: our basic objects are the \emph{classical experiments induced by
different measurements on a fixed quantum model}.

Our framework is similarly distinct from established notions of quantum
sufficiency. In the theory initiated by Petz, sufficiency of a quantum
channel or coarse-graining concerns preservation and recovery of the
underlying quantum statistical experiment and is formulated in terms of
sufficient subalgebras
\cite{Petz1986,JencovaPetz2006,JencovaPetzSurvey2006}.
By contrast, quantum-to-classical
sufficiency does not require recovery of the quantum states. It asks whether
one admissible measurement produces enough \emph{classical} information to
simulate the classical experiment generated by every other measurement in a
specified class.

A related instrument-based approach to sufficiency was developed in
\cite[Section~5]{MR2017871}. An instrument is called
\emph{exhaustive} when its posterior quantum state, conditional on the
observed outcome, is parameter-independent, so that no further information
about the parameter remains in the post-measurement system. This is closely
related to the parameter-free conditional-state structure arising from our
block-factorization principle. Our notion is stronger in a different
direction: exhaustivity of one instrument does not imply that its classical
outcome can reproduce the experiments generated by other admissible
measurements, whereas quantum-to-classical sufficiency requires Blackwell
dominance over the entire prescribed measurement class. The authors also
define quantum sufficiency for a coarsening of a fixed instrument, requiring
the retained statistic to be classically sufficient and the corresponding
posterior quantum models to preserve the same inferential content. Thus their
comparison concerns lossless coarsening within a single instrument, whereas
ours compares distinct POVMs on a fixed model in order to identify a greatest
induced classical experiment.

Another closely related concept arises in the theory of
post-processing comparison and minimal sufficiency for POVMs and quantum
statistical experiments \cite{Kuramochi2015,Kuramochi2017}. That theory
seeks a least-redundant representative of the post-processing-equivalence
class of a fixed POVM or statistical experiment. Our order-theoretic problem
points in the opposite direction: relative to a fixed quantum model and a
prescribed class of measurements, we seek a greatest \emph{induced
classical experiment}. Moreover, our comparison is model-relative, whereas
operator-level post-processing comparison of POVMs is uniform over all input
states. We make these distinctions precise in Section~3.

Symmetry reductions form the other main ingredient of the paper. Invariant
and covariant measurements have long played a central role in quantum
statistical decision theory; see, for example, \cite{Holevo2011}. General
minimax measurements and symmetry methods were studied by Bogomolov
\cite{bogomolov1982minimax}, while Kumagai and Hayashi gave a noncompact
quantum Hunt-Stein analysis for Gaussian hypothesis-testing problems with
nuisance parameters and used it, together with further structural
reductions, to construct quantum analogues of classical $\chi^2$, $t$, and
$F$ tests \cite{KumagaiHayashi2013}. In the present paper the role of
symmetry is deliberately separated from sufficiency: symmetry identifies a
decision-theoretically sufficient \emph{class} of measurements, whereas
quantum-to-classical sufficiency asks whether a single measurement dominates
all experiments induced within that class.

For spectral inference, the representation-theoretic approach to
quantum spectrum estimation and the statistical properties of WSS
have been studied extensively
\cite{alicki,keyl_werner,harrow,odonnell2016efficient,odonellspectrum}.
Previous lossless reductions to WSS are formulated in the context of
\emph{testing}: WSS followed by classical post-processing suffices for
the quantum collision problems considered in \cite{harrow} and, more
generally, for testing unitarily invariant properties
\cite{MontanaroDeWolf2016,odonellspectrum}. We formulate the
underlying reduction at the level of statistical experiments, as a
model-relative Blackwell comparison, and place it within the structural
sufficiency theory developed here. This provides the basis for the exact
decision-theoretic reductions and subsequent risk analysis.

The displaced-thermal application has an analogous relationship to the
Gaussian testing literature. Kumagai and Hayashi use displacement symmetry,
concentration transformations, and number-type measurements to solve
specific invariant Gaussian testing problems
\cite{KumagaiHayashi2013}. We formulate the corresponding structural
reduction at the level of statistical experiments: after restriction to
displacement-invariant POVMs, the total residual photon-number
experiment Blackwell-dominates every experiment induced by a measurement in
that class. The resulting reduction is therefore independent of the
particular test or loss and applies directly to the nuisance-robust minimax
and Bayes-minimax decision problems considered here.

Finally, comparison of experiments also plays a central role in quantum
local asymptotic normality and quantum asymptotic equivalence
\cite{GutaJencova2007,Kahn&Guta,YFG13,Fujiwara2020}, and has been used
to construct estimators attaining asymptotically optimal risks
\cite{Fujiwara2023,lowLAN,minimaxpure}. These results compare sequences
of quantum experiments asymptotically, typically by replacing a many-copy
model with a limiting quantum Gaussian experiment. The reductions considered
here are of a different nature: the quantum-to-classical comparison is exact
at each finite sample size, and asymptotic arguments enter only after the
relevant classical experiment has been identified.

\subsection*{Organization of the paper}

The paper begins in Section \ref{sec:prelim} with the quantum statistical notation and background needed for the subsequent developments. In Section \ref{sec:exp_n_suff}, we introduce the model-relative Blackwell order and quantum-to-classical sufficiency, together with their decision-theoretic interpretation and relation to existing notions of sufficiency. The structural theory is developed in Section \ref{sec:symmetry} through a block-factorization criterion and a corresponding converse. We then turn to the two main applications: Section \ref{sec:wss} treats spectral inference and weak Schur sampling, while Section \ref{sec:thermal} considers displaced thermal states and total residual photon-number measurement. In both cases, the exact classical reduction is followed by an analysis of the corresponding asymptotic Bayes and minimax risks. We discuss the extensions in Section \ref{sec:discussion}, while the proofs and technical auxiliary results are collected in the appendices.

\subsection{Notation}
 The vectors of a Hilbert
space $\mathcal{H}$ (assumed separable) are written as ``ket'' $|v\rangle$, $v^*$ (a vector in the dual space $\mathcal{H}^*$) as ``bra'' $\langle v|$ and the
inner product of two vectors as the ``bra-ket'' $\langle u|v\rangle\in
\mathbb{C}$ which is linear with respect to the right entry and anti-linear
with respect to the left entry. Similarly, $M:=|u\rangle\langle v| $ is the
rank one operator acting as $M  : |w\rangle \mapsto M|w \rangle = \langle
v |w\rangle |u\rangle$.  For an operator $A$ the expression $\langle u|Av\rangle$ will sometimes be denoted as $\langle u|A|v\rangle$. The space of
bounded linear operators on $\mathcal{H}$ is denoted by $\mathcal{L}(\mathcal{H})$. We write $\mathcal T_1(\mathcal H)\subset\mathcal L(\mathcal H)$ for
the trace-class operators, equipped with the trace norm
\[
\|A\|_1:=\Tr\!\left[(A^\dagger A)^{1/2}\right].
\]

For any Hilbert space, the usual norm will be denoted by $||.||$ and the identity operator on that space by $\mathbf{1}$ where the particular space will be understood from the context. For probability measures $\mu$ and $\nu$, we use
$ \|\mu-\nu\|_{\mathrm{TV}}
    :=
    \sup_B |\mu(B)-\nu(B)|,
$
and for probability mass functions $p$ and $q$ on a countable space,
$\|p-q\|_1
    :=
    \sum_x |p(x)-q(x)|
    =
    2\|p-q\|_{\mathrm{TV}}.$ By $a\vee b$ and $a\wedge b$ we will denote $\max(a,b)$ and $\min(a,b)$ respectively and $a_+$ will be used to denote $a \vee 0$. By $\lfloor a\rfloor$ and $\lceil a \rceil$, we will denote the largest integer less than or equal to $a$ and the smallest integer greater than or equal to $a$ respectively. We will use the notation $a_n\asymp b_n$ whenever $c<\liminf_n (a_n/b_n) \leq \limsup_n (a_n/b_n) <C$ for some constants $c,C>0$. Throughout the paper, $c$ and $C$ will denote arbitrary constants.

\section{Preliminaries}\label{sec:prelim}

\subsection{States, Measurements and Observables}
A \textit{state of a quantum system} is described by a self-adjoint operator
$\rho$ on a complex Hilbert space {$\mathcal{H}$, which is positive }%
$(\rho\geq0$) and normalized to $\mathrm{Tr}\left(  \rho\right)  =1$ (a
density operator). A state is called \textit{pure} if it is of the form $\rho=|\psi\rangle\langle\psi|$, otherwise it is called a \textit{mixed state}. We denote the set of states by $\mathcal{S}(\mathcal{H})$. It can be shown that $\mathcal{S}(\mathcal{H})\subset \mathcal{T}_1(\mathcal{H})$. 

Data on a quantum system are obtained from
\textit{observables} which are self-adjoint operators in the Hilbert space $\mathcal{H}$. If  $S$ is a self-adjoint operator in {$\mathcal{H}$
with spectral decomposition }$S=\sum_{j}\lambda_{j}\Pi_{j}$ where $\Pi_{j}$
are projectors, then a measurement generates a \textit{discrete random variable} $X_{S}$ taking
values in the set of eigenvalues $\left\{  \lambda_{1},\lambda_{2}%
,\ldots\right\}  $ with probabilities $p_{j}=\mathrm{Tr}(\rho\cdot\Pi_{j})$. Whenever the first absolute moment is finite, the expectation of
\(X_S\) under the state \(\rho\) is given by the
\textit{Born-von Neumann postulate:}%
\begin{equation}
E_{\rho}X_{S}=\sum_{j}\lambda_{j}\mathrm{Tr}\left[  \rho\Pi_{j}\right]
=\mathrm{Tr}\left[  \rho S\right]  . \label{trace-rule}%
\end{equation}

More generally, a measurement with outcomes in a measurable space
$(\Omega,\mathfrak B)$ is described by a positive operator-valued
measure (POVM).
\begin{definition}\label{def.POVM}
A positive operator valued measure (POVM) is a map $M:\mathfrak{B}\to \mathcal{L}(%
\mathcal{H})$ having the following properties
\begin{itemize}
\item[1)] positivity: $M(B) \geq 0$ for all events $B\in\mathfrak{B}$ (hence M(B) is self-adjoint)
\item[2)] $\sigma$-additivity: $M(\cup_i B_i) = \sum_i {M}(B_i)$ for
any countable set of mutually disjoint events $B_i$ (here the convergence is in the weak operator topology of $\mathcal{L}(\mathcal{H})$)
\item[3)] normalization: $M(\Omega) = \mathbf{1}$.
\end{itemize}
\end{definition}

The simplest example of a POVM is given by an \emph{effect} $E$, which is a bounded operator satisfying $0\leq E\leq \mathbf{1}$. Indeed $\{E,\mathbf{1}-E\}$ forms a two-outcome POVM. If the operators $M(B)$ are also orthogonal projections, i.e. $M(A)^2=M(A)$ and $M(B)M(A)=0$ when $A\cap B=\emptyset$, then it is called a \textit{simple measurement}. The collection of projectors $\{\Pi_j\}$ in the spectral decomposition $S=\sum_j \lambda_j\Pi_j$ is an example of a simple measurement. The outcome of the measurement has probability
distribution
\begin{equation}
P_\rho(B) = \mathrm{Tr}(\rho M(B)), \qquad B\in \mathfrak{B}.
\label{prob_measurement}
\end{equation}

The spectral theorem associates with every self-adjoint operator $S$,
with dense domain $D(S)\subset\mathcal H$, a projection-valued spectral
measure $M$ such that
\[
    S=\int_{\sigma(S)}x\,M(dx),
\]
in the sense of the spectral calculus. Here $\sigma(S)$ is the spectrum of $S$ and $M$ is a POVM, also called spectral measure associated with the operator $S$. When $S$ is an observable with a continuous spectrum, it generates a \textit{continuous random variable} $X_S$ with probabilities given by (\ref{prob_measurement}). 
Whenever the first absolute moment is finite,
\[
    E_\rho[X_S]
    =
    \int_{\sigma(S)} x\,\Tr\{\rho M(dx)\},
\]
which we also denote by $\Tr(\rho S)$.

\subsection{Quantum Gaussian states}
\label{subsec:gaussian-states}

Next we discuss quantum Gaussian states, which feature in the sufficiency results given in Section \ref{sec:thermal}. To describe one-mode quantum Gaussian states, consider the position
and momentum observables $Q$ and $P$ on
$\mathcal H=L^2(\mathbb R)$. On a suitable common domain $D$,
their actions are
\[
    (Qf)(x)=xf(x),
    \qquad
    (Pf)(x)=-i\frac{df}{dx}(x),
    \qquad f\in D.
\]
These operators satisfy the Heisenberg commutation relation
$$[Q,P]=i\textbf{1}.$$
It can be shown that $Z_{u}:=u_{1}Q+u_{2}P$, $u\in\mathbb{R}^{2}$ are
observables (called the \textit{canonical observables}). In this context we define the \textit{quantum characteristic function} as $\tilde{W}_{\rho}(u_1,u_2)=\mathrm{Tr}(\rho\exp\left(  iZ_{u}\right))$.
 If the following relation holds
\[
E_{\rho}\exp\left(  iZ_{u}\right)  =\mathrm{Tr}(\rho\exp\left(  iZ_{u}\right))=\exp\left(  iu^{T}\mu-\frac{1}%
{2}u^{T}\Sigma u\right)  \text{, }u\in\mathbb{R}^{2},
\]
 then $\rho$ is called a Gaussian state with mean $\mu$ and covariance matrix $\Sigma$.
 For such quantum Gaussian states in $L^{2}\left(  \mathbb{R}\right)  $ we
adopt a compact notation, resembling the one for the
$2$-variate normal law:
\begin{equation}
\rho=\mathbb{N}_{2}\left(  \mu,\Sigma\right)  . \label{notation-q-Gauss}%
\end{equation}
Here $\Sigma$ is a $2\times 2$ real matrix such that
$$\Sigma\geq \pm \frac{i}{2} \left(\begin{array}{cc}
    0 &  -1\\
     1& 0
\end{array}\right).$$
To define the simplest
Gaussian state, let $\psi_{0}=\sqrt{\varphi_{1/2}}$ be the square root of the
density function of the normal $N\left(  0,1/2\right)  $ distribution and
consider the operator $\rho_0$ acting by $\rho_0 f=\psi_{0}\left\langle \psi
_{0},f\right\rangle $, $\;f\in L^{2}\left(  \mathbb{R}\right)  $. Since
$\psi_{0}$ is a unit vector in $L^{2}\left(  \mathbb{R}\right)  $, the
operator $\rho_0$ (henceforth called the vacuum state) is a projection (written $\rho_0=\left\vert \psi_{0}%
\right\rangle \left\langle \psi_{0}\right\vert $ in Dirac notation) and it can be
shown that $\rho_0=\mathbb{N}_{2}\left(  0,I_2/2\right)$ in the notation described above.

An important class
is the collection of \textit{coherent states} $\mathbb{N}_{2}\left(  \mu,I_2/2\right)  $;
these are pure states which can be interpreted as a vacuum shifted by $\mu
\in\mathbb{R}^{2}$ (similar to the Gaussian shift model in classical statistics).
Consider the operators $a^{\dagger}=(Q-iP)/\sqrt{2}$ (the \textit{creation operator}), $a=(Q+iP)/\sqrt{2}$ (the \textit{annihilation operator}) and $\hat N=a^{\dagger}a$ (the number operator). It is well known that the \textit{Hermite basis} $\{|0\rangle,|1\rangle\,\ldots\}$ forms an eigenbasis of the number operator,  i.e. $\hat N|k\rangle=k|k\rangle$. For any $z \in \mathbb{C}$ define the displacement operator as
$$D(z)=\exp(za^{\dagger}-\bar{z}a)$$ 
It satisfies the Weyl relation
\begin{equation}
    D(z)D(w)
    =
    e^{i\operatorname{Im}(z\overline w)}
    D(z+w).
    \label{eq:weyl-relation}
\end{equation}
The phase in \eqref{eq:weyl-relation} disappears under conjugation, and hence
the additive group
$(\bbC,+)$
acts on states through
\[
    \rho\longmapsto D(z)\rho D(z)^\dagger.
\]

Another important class of Gaussian states is given by thermal states and
their displaced versions. For $N>0$, define
\begin{equation}
    \phi_N
    :=
    \frac{1}{N+1}
    \sum_{k=0}^{\infty}
    \left(
        \frac{N}{N+1}
    \right)^k
    |k\rangle\langle k|.
    \label{eq:thermal-state}
\end{equation}

The displaced thermal family is
\begin{equation}
    \rho_{z,N}
    :=
    D(z)\phi_ND(z)^\dagger,
    \qquad
    z\in\bbC,\quad N>0.
    \label{eq:displaced-thermal}
\end{equation}

We use $N$ as the statistical parameter of the thermal family.  The inverse temperature $\beta$ and
the mean photon number $N$ of the thermal states \eqref{eq:thermal-state} are related by
\[
    N=\frac{1}{e^\beta-1},
    \qquad
    \beta=\log\!\left(\frac{N+1}{N}\right).
\]
Thus $N\mapsto\beta$ is one-to-one on $(0,\infty)$, and any functional
of the temperature or inverse temperature may equivalently be regarded
as a functional of $N$.  For example, if $T=\beta^{-1}$ and
$\xi(T)$ is a temperature functional, then
\[
    \psi(N)
    :=
    \xi\!\left(
        \frac{1}{\log((N+1)/N)}
    \right).
\]
More generally, we shall refer to any functional $\psi(N)$ of the mean
photon number as a thermal functional.

One can show that the quantum characteristic function $\Tr(\rho_{z,N}\exp(iu_1Q+iu_2P))$ of the shifted thermal state is given by
\[
\Tr(\rho_{z,N}\exp(iu_1Q+iu_2P))=\exp(i(u_1\sqrt{2}\Ree(z)+u_2\sqrt{2}\Img(z))-\frac{(N+1/2)}{2}(u_1^2+u_2^2)).
\]
The family is covariant under the displacement group:
\begin{equation}
    D(w)\rho_{z,N}D(w)^\dagger
    =
    \rho_{z+w,N}
    \label{eq:displacement-covariance}.
\end{equation}

\section{Quantum statistical experiments and induced classical experiments}
\label{sec:exp_n_suff}

In this section we fix the statistical framework used throughout the paper. The comparison theory employed below is classical after a measurement has been performed: a measurement maps the original quantum experiment to an ordinary family of probability measures. We therefore distinguish carefully between the quantum model, the measurement, and the classical experiment induced by that measurement.

\subsection{Quantum statistical experiments}

Let $\cH$ be a separable Hilbert space and let $\mathcal S(\cH)$ denote the set of density operators on $\cH$. 

\begin{definition}[Quantum statistical experiment]
A quantum statistical experiment is a triple
\[
    \cQ=(\cH,\Theta,\{\rho_\theta:\theta\in\Theta\}),
\]
where $\Theta$ is the parameter space and $\rho_\theta\in\mathcal S(\cH)$ for each
$\theta\in\Theta$.
\end{definition}

The $n$-copy experiment associated with a one-copy family $\{\rho_\theta\}$ is
\[
    \cQ^{(n)}
    =\left((\cH)^{\otimes n},\Theta,
    \{\rho_\theta^{\otimes n}:\theta\in\Theta\}\right).
\]

Recall that measuring $\rho_\theta$ with $M$ produces the probability measure
\[
    P_\theta^M(B):=\Tr[\rho_\theta M(B)],\qquad B\in\mathfrak B.
\]
The corresponding induced classical experiment is
\[
    \cE_M:=\bigl(\Omega,\mathfrak B,\{P_\theta^M:\theta\in\Theta\}\bigr).
\]

\subsection{Comparison of measurements on a fixed model}
\label{subsec:model-relative-comparison}

Let $M$ and $N$ be POVMs with standard Borel outcome spaces
$(\mathsf X,\mathcal X)$ and $(\mathsf Y,\mathcal Y)$, respectively.

\begin{definition}[Model-relative Blackwell order]
\label{def:model-relative-order}
We write $M\succeq_{\cQ}N$ if there exists a Markov kernel
$K_N:\mathsf X\times\mathcal Y\to[0,1]$ such that
\[
    P_\theta^N(B)
    =
    \int_{\mathsf X}K_N(x,B)\,P_\theta^M(dx),
    \qquad
    \theta\in\Theta,\quad B\in\mathcal Y.
\]
Equivalently, $P_\theta^N=K_NP_\theta^M$ for every $\theta\in\Theta$.
\end{definition}

The order in \cref{def:model-relative-order} depends on the fixed quantum
model $\cQ=\{\rho_\theta:\theta\in\Theta\}$. It requires the
post-processing identity only for the outcome distributions generated by
states in $\cQ$. It is therefore weaker than the usual post-processing order
of POVMs, which requires the operator identity
\[
    N(B)
    =
    \int_{\mathsf X}K_N(x,B)\,M(dx),
    \qquad B\in\mathcal Y,
\]
and hence reproduces the outcome law of $N$ for every state on $\cH$.

\begin{definition}[Quantum-to-classical sufficiency]
\label{def:q2c-sufficiency}
Let $\mathfrak M$ be a prescribed class of POVMs on $\cH$. A measurement
$M^\star\in\mathfrak M$ is \emph{quantum-to-classically sufficient} for
$\cQ$ relative to $\mathfrak M$ if
\[
    M^\star\succeq_{\cQ}N
    \qquad\text{for every }N\in\mathfrak M.
\]
Equivalently, $\cE_{M^\star}$ is a greatest element, up to Blackwell
equivalence, in the family of induced classical experiments
$
    \{\cE_N:N\in\mathfrak M\}.
$
\end{definition}

The terminology in \cref{def:q2c-sufficiency} describes a greatest induced
classical experiment. The pairwise comparison itself is the ordinary
Blackwell comparison of the classical experiments $\cE_M$ and $\cE_N$.

\begin{remark}[Commuting quantum models]
\label{rem:commuting-model}
Suppose that $\cH$ is finite dimensional and the states
$\{\rho_\theta:\theta\in\Theta\}$ commute. They are then simultaneously
diagonalizable, so for a common orthonormal eigenbasis
$\{e_i\}_{i=1}^r$ we may write
\[
    \rho_\theta
    =
    \sum_{i=1}^r p_i(\theta)E_i,
    \qquad
    E_i:=|e_i\rangle\langle e_i|.
\]
Let $E^\star=\{E_i\}_{i=1}^r$ be the corresponding projective
measurement. Note that the outcome distribution for $E^\star$ is given by
$
    P_\theta^{E^\star}(\{i\})=p_i(\theta)
$. For any POVM $N$ on $(\mathsf Y,\mathcal Y)$, define
\[
    K_N(i,B)
    :=
    \Tr\!\left(E_iN(B)\right),
    \qquad
    B\in\mathcal Y.
\]
Since $E_i$ is rank one, $K_N(i,\cdot)$ is a probability measure, and
\[
\begin{aligned}
    P_\theta^N(B)
    &=
    \Tr\!\left(\rho_\theta N(B)\right)\\
    &=
    \sum_{i=1}^r
    p_i(\theta)\Tr\!\left(E_iN(B)\right)\\
    &=
    \sum_{i=1}^r
    K_N(i,B)P_\theta^{E^\star}(\{i\}).
\end{aligned}
\]
Hence
\[
    P_\theta^N=K_NP_\theta^{E^\star},
    \qquad \theta\in\Theta,
\]
so $E^\star$ is quantum-to-classically sufficient relative to the
class of all POVMs.
\end{remark}

\begin{remark}[Uniqueness up to Blackwell equivalence]
\label{rem:q2c-uniqueness}
If $M_1^\star$ and $M_2^\star$ are both quantum-to-classically sufficient
relative to the same class $\mathfrak M$, then
\[
    M_1^\star\succeq_{\cQ}M_2^\star
    \qquad\text{and}\qquad
    M_2^\star\succeq_{\cQ}M_1^\star.
\]
Hence $\cE_{M_1^\star}$ and $\cE_{M_2^\star}$ are Blackwell equivalent.
Accordingly, a greatest induced experiment is unique only up to classical
randomization equivalence.
\end{remark}

Blackwell dominance has an immediate decision-theoretic consequence: every
classical action distribution obtainable from the less informative experiment
can be reproduced from the more informative one. We record this consequence
for measurement-induced experiments.

\subsection{Quantum decision procedures}
\label{subsec:quantum-decision-procedures}

Let $(\mathsf A,\mathcal A)$ be a standard Borel action space.

\begin{definition}[Quantum decision procedure]
A quantum decision procedure consists of a POVM $M$ with outcome space
$(\mathsf Y,\mathcal Y)$ and a Markov kernel
$\delta:\mathsf Y\times\mathcal A\to[0,1]$.
\end{definition}

Equivalently, the measurement and the classical decision rule can be
combined into an \emph{action-valued POVM} $D$ defined by
\[
    D(A)
    :=
    \int_{\mathsf Y}\delta(y,A)\,M(dy),
    \qquad A\in\mathcal A.
\]

Let $L:\Theta\times\mathsf A\to[0,\infty]$ be a measurable loss
function. The risk of the procedure $(M,\delta)$ at
$\theta$ is
\[
    R_\theta(M,\delta)
    :=
    \int_{\mathsf Y}\int_{\mathsf A}
    L(\theta,a)\,\delta(y,da)\,
    \Tr\!\left(\rho_\theta M(dy)\right).
\]
Equivalently, for the corresponding action-valued POVM $D$,
\[
    R_\theta(D)
    =
    \int_{\mathsf A}
    L(\theta,a)\Tr\!\left(\rho_\theta D(da)\right).
\]

Let $\Pi$ be a prior probability measure on $\Theta$. The Bayes risk and optimal Bayes risk are
\[
    \BayesRisk(D)
    :=
    \int_\Theta R_\theta(D)\,\Pi(d\theta),
    \qquad
    r(\Pi)
    :=
    \inf_D\BayesRisk(D),
\]
where the infimum is over all action-valued POVMs on
$(\mathsf A,\mathcal A)$.

\begin{prop}[Risk comparison]
\label{prop:risk-transfer}
Suppose that $M\succeq_{\cQ}N$. For every randomized decision rule
$\delta_N$ based on $N$, there exists a randomized decision rule $\delta_M$
based on $M$ such that the induced action distributions coincide for every
$\theta\in\Theta$. Consequently,
\[
    R_\theta(M,\delta_M)
    =
    R_\theta(N,\delta_N),
    \qquad \theta\in\Theta,
\]
for every nonnegative measurable loss for which the risks are well-defined.
\end{prop}

\begin{proof}
Let $K_N$ be a Markov kernel satisfying
\[
    P_\theta^N=K_NP_\theta^M,
    \qquad \theta\in\Theta,
\]
and define
\[
    \delta_M(x,C)
    :=
    \int_{\mathsf Y}\delta_N(y,C)\,K_N(x,dy),
    \qquad
    x\in\mathsf X,\quad C\in\mathcal A.
\]
Then, for every $C\in\mathcal A$,
\[
\begin{aligned}
    \int_{\mathsf X}\delta_M(x,C)\,P_\theta^M(dx)
    &=
    \int_{\mathsf X}\int_{\mathsf Y}
    \delta_N(y,C)K_N(x,dy)P_\theta^M(dx)\\
    &=
    \int_{\mathsf Y}\delta_N(y,C)\,P_\theta^N(dy).
\end{aligned}
\]
Thus the procedures $(M,\delta_M)$ and $(N,\delta_N)$ induce the same
distribution on the action space under every $\rho_\theta$. Integrating
$L(\theta,\cdot)$ against this common action distribution proves the risk
identity.
\end{proof}

\begin{remark}
\label{rem:randomization-converse}
Under the hypotheses of the classical randomization theorem, inclusion of
the attainable risk sets for all bounded classical decision problems also
implies the existence of a randomization kernel (cf. \cite{LeCam1996}). We use the kernel
formulation as the primary definition because it gives an exact,
decision-independent comparison of the induced experiments.
\end{remark}

\subsection{Comparison with other notions of sufficiency}
\label{subsec:comparison-sufficiency}

We compare Definition~\ref{def:q2c-sufficiency} with classical
sufficiency, quantum-channel sufficiency, and minimal sufficiency
under post-processing of POVMs.

\subsubsection{Classical sufficiency}

Let $\cE=\{P_\theta:\theta\in\Theta\}$ be a classical experiment on
$(\mathsf X,\mathcal X)$, and let
\[
    T:(\mathsf X,\mathcal X)\rightarrow(\mathsf T,\mathcal T)
\]
be a statistic. The induced experiment is
\[
    \cE^T=\{P_\theta^T:\theta\in\Theta\},
    \qquad
    P_\theta^T=P_\theta\circ T^{-1}.
\]

\begin{definition}[Classical sufficiency]
\label{def:classical-sufficiency}
The statistic $T$ is sufficient for $\cE$ if there exists a Markov kernel
$K$ from $(\mathsf T,\mathcal T)$ to $(\mathsf X,\mathcal X)$ such that
\[
    P_\theta(B)
    =
    \int_{\mathsf T}K(t,B)\,P_\theta^T(dt),
    \qquad
    \theta\in\Theta,\quad B\in\mathcal X.
\]
\end{definition}

Since $T$ itself maps $\cE$ to $\cE^T$, the definition is equivalent to
\[
    \cE\simeq_B\cE^T.
\]
It is also equivalent, under the usual regularity conditions, to requiring
that the conditional law of the original observation given $T$ be
independent of $\theta$.

Classical and quantum-to-classical sufficiency use the same Blackwell
comparison but apply it to different families. Classical sufficiency begins
with one distinguished classical experiment and asks whether a statistic
compresses its observation without loss. Quantum-to-classical sufficiency
begins before a measurement has been selected and compares
different experiments induced by different $M\in\mathfrak M$. It therefore
establishes that one measurement is complete for an admissible measurement
class, not merely that a statistic is sufficient for the outcome of a
previously fixed measurement. Once $M^\star$ has been identified, however,
an ordinary sufficient statistic for $\cE_{M^\star}$ may be used to remove
redundancy from its classical outcome without changing its Blackwell
dominance over the other induced experiments.

\subsubsection{Quantum-channel sufficiency}

A quantum channel
\[
    \Phi:\mathcal T_1(\cH)\rightarrow\mathcal T_1(\cK)
\]
maps $\cQ$ to the quantum experiment
\[
    \Phi(\cQ)=\{\Phi(\rho_\theta):\theta\in\Theta\}.
\]

\begin{definition}[Quantum randomization order and channel sufficiency]
\label{def:quantum-channel-sufficiency}
For quantum experiments
\[
    \cQ=\{\rho_\theta:\theta\in\Theta\}
    \quad\text{and}\quad
    \cR=\{\sigma_\theta:\theta\in\Theta\},
\]
we say that $\cQ$ quantum-dominates $\cR$ if there exists a quantum
channel $\Phi$ satisfying
\[
    \Phi(\rho_\theta)=\sigma_\theta,
    \qquad \theta\in\Theta.
\]
The channel $\Phi$ is sufficient for $\cQ$ if there is a recovery channel
\[
    \Psi:\mathcal T_1(\cK)\rightarrow\mathcal T_1(\cH)
\]
such that
\[
    \Psi\circ\Phi(\rho_\theta)=\rho_\theta,
    \qquad \theta\in\Theta.
\]
\end{definition}

Quantum-channel sufficiency preserves the quantum experiment itself. It
therefore preserves quantum decision problems, including procedures with
quantum outputs, and includes Petz sufficiency and equality cases of quantum
data-processing inequalities. By contrast, quantum-to-classical
sufficiency requires no recovery of $\rho_\theta$. It preserves only the
classical experiments obtainable from measurements in $\mathfrak M$.

The distinction remains when a POVM $M$ is viewed as a
quantum-to-classical channel
\[
    \mathsf M:\rho\longmapsto P_\rho^M.
\]
Sufficiency of $\mathsf M$ in the quantum-channel sense would require the
classical output to support a recovery channel reconstructing every
$\rho_\theta$. This condition generally fails for noncommuting models and
is strictly stronger than sufficiency in
\cref{def:q2c-sufficiency}. Indeed, if such a recovery channel exists, one
may recover $\rho_\theta$ and then implement any POVM; hence $M$ is
quantum-to-classically sufficient relative to the class of all POVMs. The
converse fails for restricted measurement classes. Weak Schur sampling, for
example, dominates the invariant POVMs relevant to spectral inference in
\cref{sec:wss}, but discards eigenvector information and cannot recover the
full tensor-power state.

\subsubsection{Minimal sufficient POVMs}

Minimal sufficiency under post-processing equivalence is particularly close
to our formulation because it also compares POVMs through classical
post-processing \cite{Kuramochi2015,Kuramochi2017}. For standard Borel
POVMs, write $A\preceq B$ when there is a Markov kernel $K$ satisfying the
operator identity
\[
    A(E)=\int K(y,E)\,B(dy),
\]
and write $A\simeq B$ when $A\preceq B$ and $B\preceq A$. Because the
relation is imposed at the operator level, it reproduces the outcome laws
uniformly over all input states.

A POVM $A$ is minimal sufficient if, for every $B\simeq A$, there
exists a measurable map $f$ from the outcome space of $B$ to that of
$A$ such that
\[
A(E)=B(f^{-1}(E))
\]
for every measurable outcome event $E$ of $A$.
It is therefore a least-redundant representative of the fixed
post-processing-equivalence class of $A$.

Despite the common post-processing language, this notion differs from
quantum-to-classical sufficiency in three respects. First, minimal POVM
sufficiency is not dependent on the states, whereas the order $\succeq_{\cQ}$ is relative
to the prescribed model and may identify measurements that differ on
states outside $\cQ$. Second, minimal sufficiency compares only POVMs in
one post-processing-equivalence class, whereas
\cref{def:q2c-sufficiency} compares all measurements in $\mathfrak M$,
which need not be mutually equivalent. Third, their order-theoretic
directions are opposite: minimal sufficiency selects a least-redundant
representative within one equivalence class, while quantum-to-classical
sufficiency seeks a greatest induced experiment across an admissible
measurement class.

Neither property implies the other. Appending an independent random label
to a quantum-to-classically sufficient measurement leaves its induced
experiment Blackwell equivalent to the original and hence still greatest,
but introduces removable outcome redundancy. Conversely, a minimal
sufficient POVM need not dominate measurements outside its own
post-processing-equivalence class; for example, a nonredundant qubit
projective measurement cannot reproduce an incompatible projective
measurement on the full qubit model by classical post-processing.

These distinctions also locate quantum-to-classical sufficiency relative
to decision-specific optimality. By \cref{prop:risk-transfer}, a sufficient
$M^\star$ transfers every classical decision procedure based on
$\mathfrak M$, with its entire risk function, to a procedure based on
$M^\star$. Equality of optimal Bayes or minimax risks for one prior and loss
is weaker: it concerns only that decision problem and does not provide the
parameter-independent kernels required to simulate every admissible
measurement experiment.

\section{Symmetry in quantum statistical models: Invariant models and sufficiency}
\label{sec:symmetry}

\subsection{Invariant problems and orbit parameters}
\label{sec:invariant-problems}

Let \(G\) be a compact Hausdorff group with normalized Haar measure \(dg\), acting measurably on the parameter space \(\Theta\). Let \(g\mapsto U_g\) be a strongly continuous unitary representation of \(G\) on the separable Hilbert space \(\cH\). We assume that the quantum statistical model \(\cQ=\{\rho_\theta:\theta\in\Theta\}\) is equivariant, in the sense that

$$
    \rho_{g\theta}=U_g\rho_\theta U_g^\dagger,
    \qquad g\in G,\quad \theta\in\Theta.
$$

We consider statistical problems in which the parameter of interest is invariant under the group action, so that the group orbit represents a nuisance direction. Write \(\theta\sim_G\theta'\) if \(\theta'=g\theta\) for some \(g\in G\), let \(\Theta/G\) denote the corresponding orbit space, and let \(\pi:\Theta\to\Theta/G\) be the quotient map. Thus the invariant parameter may be identified with the orbit \(\eta=\pi(\theta)\). When a measurable structure on \(\Theta/G\) is required, we equip it with the quotient \(\sigma\)-field induced by \(\pi\).

A POVM \(M\) is called \(G\)-invariant if \(U_g^\dagger M(B)U_g=M(B)\) for every \(g\in G\) and every measurable outcome event \(B\). For such a measurement, equivariance of the model immediately gives \(P_{g\theta}^M=P_\theta^M\). Hence the classical experiment induced by \(M\) depends on \(\theta\) only through its orbit \(\eta=\pi(\theta)\).

For a trace-class operator \(A\), define its group average by

$$
    \cT_G(A):=\int_G U_gAU_g^\dagger\,dg,
$$
where the integral is understood in the Bochner sense with respect to the trace norm. Strong continuity of the representation implies that \(g\mapsto U_gAU_g^\dagger\) is trace-norm continuous, so the integral is well defined. For every \(G\)-invariant POVM \(M\),

$$
    \Tr\!\left[\rho_\theta M(B)\right]
    =
    \Tr\!\left[\cT_G(\rho_\theta)M(B)\right].
$$

Moreover, Haar invariance and equivariance imply that \(\cT_G(\rho_\theta)\) is constant on group orbits. We may therefore define
$$
    \bar\rho_\eta
    :=
    \cT_G(\rho_\theta),
    \qquad
    \pi(\theta)=\eta,
$$
without dependence on the chosen representative \(\theta\). Consequently, relative to the class of \(G\)-invariant POVMs, the original experiment \(\cQ\) and the orbit-averaged experiment \(\overline{\cQ}:=\{\bar\rho_\eta:\eta\in\Theta/G\}\) induce exactly the same family of classical experiments.

We now characterize conditions under which this orbit-averaged quantum experiment itself admits an exact reduction to a single classical experiment.

\subsection{Block structure and quantum-to-classical sufficiency}
\label{sec:invariant-factorization}

Let
\[
    \cA_G
    :=
    \{A\in\mathcal L(\cH):[A,U_g]=0
      \text{ for every }g\in G\}
\]
be the invariant algebra. Suppose that $\mathsf Z$ is a finite set and 
\[
    \cH
    =
    \bigoplus_{z\in\mathsf Z}\cH_z
\]
is an orthogonal decomposition with projections
$\{\Pi_z:z\in\mathsf Z\}$ such that the $\Pi_z$ are central in $\cA_G$.
Equivalently,
\[
    \cA_G
    =
    \bigoplus_{z\in\mathsf Z}\cA_z,
    \qquad
    \cA_z
    :=
    \Pi_z\cA_G\Pi_z.
\]
In particular, every $G$-invariant POVM is block diagonal with respect to this
decomposition.

For the orbit experiment define
\[
    p_\eta(z)
    :=
    \Tr(\bar\rho_\eta\Pi_z).
\]
Whenever $p_\eta(z)>0$, write
\[
    \tau_{\eta,z}
    :=
    \frac{\Pi_z\bar\rho_\eta\Pi_z}{p_\eta(z)}.
\]
Thus every invariant state admits the block representation
\[
    \bar\rho_\eta
    =
    \bigoplus_{z\in\mathsf Z}
    p_\eta(z)\tau_{\eta,z}.
\]
If $p_\eta(z)=0$, positivity implies
$\Pi_z\bar\rho_\eta\Pi_z=0$, and $\tau_{\eta,z}$ may be chosen
arbitrarily as a density operator on the nonzero block
$\mathcal H_z$.  

Let
\[
    M^\star(\{z\})=\Pi_z,
    \qquad
    z\in\mathsf Z,
\]
denote the projective measurement of the block label. The next theorem shows that $M^\star$ is quantum-to-classically
sufficient when the conditional block states can be chosen
independently of $\eta$.

\subsubsection{A sufficient condition}

The following theorem is the quantum analogue of the classical characterization
of sufficiency through parameter-free conditional distributions.

\begin{theorem}[Invariant factorization theorem]
\label{thm:invariant-factorization}
Suppose that the orbit-averaged experiment satisfies
\begin{equation}
    \bar\rho_\eta
    =
    \bigoplus_{z\in\mathsf Z}
    p_\eta(z)\tau_z,
    \label{eq:inv-factor}
\end{equation}
where each $\tau_z$ is a density operator on $\cH_z$ independent of $\eta$.
Then the block-label measurement $M^\star$ is quantum-to-classically sufficient
for $\overline{\cQ}$ relative to the class of all $G$-invariant POVMs.
\end{theorem}

\begin{proof}
Let $N$ be a $G$-invariant POVM with outcome space
$(\mathsf Y,\mathcal Y)$. Since $N(B)\in\cA_G$ for every $B\in\mathcal Y$, the POVM is
block diagonal:
\[
    N(B)
    =
    \bigoplus_{z\in\mathsf Z}N_z(B).
\]
Define
\[
    K_N(z,B)
    :=
    \Tr[\tau_zN_z(B)].
\]
For each $z$, $B\mapsto K_N(z,B)$ is a probability measure, so $K_N$ is a
Markov kernel. Using \eqref{eq:inv-factor},
\[
\begin{aligned}
    P_\eta^N(B)
    &=
    \Tr[\bar\rho_\eta N(B)]\\
    &=
    \sum_{z\in\mathsf Z}
    p_\eta(z)\Tr[\tau_zN_z(B)]\\
    &=
    \sum_{z\in\mathsf Z}
    K_N(z,B)p_\eta(z).
\end{aligned}
\]
Since
\(
    P_\eta^{M^\star}(\{z\})
    =
    p_\eta(z),
\)
we obtain
\(
    P_\eta^N
    =
    K_NP_\eta^{M^\star}.
\)
\end{proof}

The interpretation is direct. The block label contains all dependence on the
orbit parameter, whereas the conditional quantum state within each block is
parameter-free. Once the block label has been observed, every further invariant
measurement is therefore a parameter-independent randomization.

\subsubsection{Necessary conditions and the converse}

Next we note that sufficiency itself always implies an observable-side factorization.

\begin{prop}[Observable-side converse]
\label{prop:observable-side-converse}
Suppose that $M^\star$ is quantum-to-classically sufficient for
$\overline{\cQ}$ relative to all $G$-invariant POVMs. Then for every invariant
effect $E\in\cA_G$ there exists a function
\[
    f_E:\mathsf Z\to[0,1]
\]
such that
\begin{equation}
    \Tr(\bar\rho_\eta E)
    =
    \sum_{z\in\mathsf Z}
    p_\eta(z)f_E(z),
    \qquad
    \eta\in\Theta/G.
    \label{eq:observable-factorization}
\end{equation}
\end{prop}

\begin{proof}
Apply sufficiency to the binary invariant POVM $(E,I-E)$. The corresponding
Markov kernel from the outcome of $M^\star$ gives the coefficients
$f_E(z)$.
\end{proof}

Thus observable-side factorization is a necessary condition for sufficiency
without any identifiability assumption. It does not, however, by itself imply
that the coefficients $f_E(z)$ arise from parameter-free conditional quantum
states. For this stronger conclusion one needs the block probabilities to
identify their coordinates.

We assume that the block index set
$\mathsf Z$ is finite. Define
\[
    p_\eta
    :=
    \bigl(p_\eta(z)\bigr)_{z\in\mathsf Z}
    \in\mathbb R^{|\mathsf Z|}.
\]
We impose the following identifiability condition.

\begin{assumption}[Block-coordinate identifiability]
\label{ass:block-coordinate-identifiability}
If $a:\mathsf Z\to\mathbb R$ satisfies
\[
    \sum_{z\in\mathsf Z}p_\eta(z)a(z)=0
    \qquad
    \text{for every }\eta\in\Theta/G,
\]
then $a(z)=0$ for every $z\in\mathsf Z$.
\end{assumption}

\begin{remark}
\label{rem:block-coordinate-identifiability}
Since $\mathsf Z$ is finite,
Assumption~\ref{ass:block-coordinate-identifiability} is equivalent to
\[
    \operatorname{span}
    \left\{
        p_\eta:\eta\in\Theta/G
    \right\}
    =
    \mathbb R^{|\mathsf Z|}.
\]
It is therefore a statistical identifiability condition on the family of
block probabilities, rather than a consequence of the representation-theoretic
block decomposition. Its role in the converse theorem is to ensure uniqueness
of the coefficients in the observable-side factorization
\eqref{eq:observable-factorization}. 
\end{remark}

Under this additional condition, observable-side factorization can be strengthened
to a state-side factorization, which gives us the following theorem.

\begin{theorem}[Converse invariant factorization theorem]
\label{thm:converse-invariant-factorization}
Suppose that $\cH$ is finite dimensional and Assumption~\ref{ass:block-coordinate-identifiability} holds and
that the block-label measurement $M^\star$ is quantum-to-classically sufficient
for $\overline{\cQ}$ relative to all $G$-invariant POVMs. Then there exist
density operators $\tau_z$, supported on $\cH_z$ and independent of $\eta$,
such that
\begin{equation}
    \bar\rho_\eta
    =
    \bigoplus_{z\in\mathsf Z}
    p_\eta(z)\tau_z,
    \qquad
    \eta\in\Theta/G.
   \label{eq:converse-factorization}
\end{equation}
\end{theorem}

Combining the preceding results gives the following characterization.

\begin{corollary}[Factorization characterization of invariant sufficiency]
\label{cor:factorization-characterization}
Suppose $\mathcal{H}$ is finite dimensional. Under Assumption~\ref{ass:block-coordinate-identifiability}, the following are
equivalent:
\begin{enumerate}[label=\textup{(\roman*)}]
\item the orbit experiment admits the parameter-free conditional block
factorization
\[
    \bar\rho_\eta
    =
    \bigoplus_zp_\eta(z)\tau_z;
\]
\item the block-label measurement $M^\star$ is quantum-to-classically
sufficient relative to all $G$-invariant POVMs.
\end{enumerate}
\end{corollary}

\section{Sufficiency for spectral functionals of qudits}
\label{sec:wss}

We now specialize the preceding framework to spectral inference from
identically prepared finite-dimensional quantum systems.
In this setup, we will consider the Hilbert space $\cH_d=\bbC^d$ and the $d$-dimensional quantum state (called \emph{qudit}) $\rho$ is an element of $\cS(\cH_d)$. Similarly the $n$-fold copy of the same state is an element of $\cS(\cH_d^{\otimes n})$. Throughout this section we distinguish the full quantum state from its
spectral parameter.  
Let \[
    \Theta_{\downarrow}
    :=\left\{
      \btheta=(\theta_1,\ldots,\theta_d)\in[0,1]^d:
      \theta_1\geq\cdots\geq\theta_d,\quad
      \sum_{i=1}^d\theta_i=1
    \right\}.
\]
For a density operator $\rho$, we denote the vector of eigenvalues of $\rho$ as $\spec(\rho)$. We also write
\[
    \btheta=\spec(\rho)\in\Theta_{\downarrow},
    \qquad
    \rho_{\btheta}:=\operatorname{diag}(\theta_1,\ldots,\theta_d),
\]
and, for $U\in U(d)$,
\[
    \rho_{\btheta,U}:=U\rho_{\btheta}U^{\dagger}.
\]
Thus $\rho$ denotes a point in the full state space, whereas
$\btheta$, the spectrum of the state, denotes the corresponding orbit parameter and $U$ the
eigenvector nuisance parameter.  We also write
\[
    \bar\rho_{\btheta}^{(n)}
    :=
    \int_{U(d)}
      (U\rho_{\btheta}U^{\dagger})^{\otimes n}\,\Haar(dU)
\]
for the orbit-averaged $n$-copy state.  The latter is used only as a
symmetry-reduction device: for every $U(d)$-invariant POVM, its outcome
law on $\rho^{\otimes n}$ agrees with that on
$\bar\rho_{\btheta}^{(n)}$.

\subsection{Construction of the sufficient measurement}

Let
\[
    \Lambda=\mathbb Y_{n,d}
    :=
    \left\{
        \lambda=(\lambda_1,\ldots,\lambda_d)\in\mathbb N_0^d:
        \lambda_1\geq\cdots\geq\lambda_d,\quad
        \sum_{i=1}^d\lambda_i=n
    \right\},
\]
the set of ordered partitions of $n$. Equivalently, we denote
\[
    \mathbb Y_{n,d}
    =
    \{\lambda\vdash n:\ell(\lambda)\leq d\},
\]
where partitions with fewer than $d$ nonzero parts are padded with
trailing zeros.
The unitary group $U(d)$
acts through $U\mapsto U^{\otimes n}$, while the symmetric group $S_n$ acts by
permuting tensor factors. Schur-Weyl duality (see Appendix for details) gives
\begin{equation}
    (\bbC^d)^{\otimes n}
    =
    \bigoplus_{\lambda\in\Lambda}
    \cU_\lambda^{(d)}\otimes\cV_\lambda,
\label{eq:Schur_Weyl}
\end{equation}

and
\[
    U^{\otimes n}
    =
    \bigoplus_{\lambda\in\Lambda}
    U_\lambda(U)\otimes I_{\cV_\lambda},
\]
where $U_\lambda$ is the irreducible representation of $U(d)$ on
$\cU_\lambda^{(d)}$.
 The invariant algebra is
\[
    \{U^{\otimes n}:U\in U(d)\}'
    =
    \bigoplus_{\lambda\in\Lambda}
    I_{\cU_\lambda^{(d)}}\otimes\mathcal L(\cV_\lambda).
\]

Its minimal central projections are precisely the Schur-Weyl block projections
$\Pi_\lambda$. Consequently, every $U(d)$-invariant POVM has the form
\[
    N(B)
    =
    \bigoplus_{\lambda\in\Lambda}
    I_{\cU_\lambda^{(d)}}\otimes N_\lambda(B).
\]

Weak Schur sampling is the block-label measurement
\[
    M^{\mathrm{WSS}}
    =
    \{\Pi_\lambda:\lambda\in\Lambda\}.
\]
We next show that weak
Schur sampling is sufficient for the class of invariant POVMs.
For a POVM $M$ on $(\mathbb C^d)^{\otimes n}$, we abbreviate
$P_\rho^M(B):=\operatorname{Tr}\{\rho^{\otimes n}M(B)\}$.

\begin{theorem}[Weak Schur sampling sufficiency]
\label{thm:wss}
For the experiment
\[
    \cQ_{n,d}
    =
    \{\rho^{\otimes n}:
      \rho\in\mathcal S(\bbC^d)\},
\]
weak Schur sampling is quantum-to-classically sufficient relative to the class
of $U(d)$-invariant POVMs. In particular, for every $U(d)$-invariant POVM
$N$ there exists a Markov kernel $K_N$ such that
\[
    P_\rho^N
    =
    K_NP_\rho^{\mathrm{WSS}}
    \qquad
    \text{for every }\rho\in\mathcal S(\bbC^d).
\]
\end{theorem}

\begin{proof}

For spectral inference the nuisance action is
\(
    \rho\mapsto U\rho U^\dagger,
\)
and the orbit parameter, defined in Section \ref{sec:symmetry} is
\(
    \eta=\spec(\rho).
\)
Recall that for $\theta\in\Theta_{\downarrow}$,  the orbit-averaged state is defined as
\begin{equation}
    \overline\rho_{\theta}^{(n)}
    :=
    \int_{U(d)}
    \bigl(U\diag(\theta)U^{\dagger}\bigr)^{\otimes n}
    \,\Haar(dU). \label{averaged-n-qudit}
\end{equation}
For every invariant POVM $M$,
\[
    \Tr\bigl(\rho^{\otimes n}M(B)\bigr)
    =
    \Tr\bigl(
        \overline\rho_{\spec(\rho)}^{(n)}M(B)
    \bigr),
\]
so the orbit-averaged model reproduces all invariant measurement statistics.

The tensor-power state $\rho_\theta^{\otimes n}$ is invariant under
permutations of the tensor factors.  Under the Schur-Weyl decomposition \eqref{eq:Schur_Weyl}
the collective action of $U(d)$ and the permutation action of $S_n$ take
the forms
\begin{equation}
    U^{\otimes n}
    =
    \bigoplus_{\lambda\in\Lambda}
    U_\lambda(U)\otimes I_{\mathcal V_\lambda},
    \qquad
    P_\pi
    =
    \bigoplus_{\lambda\in\Lambda}
    I_{\mathcal U_\lambda^{(d)}}\otimes\pi_\lambda(\pi), \label{sw_action_decomps}
\end{equation}
where $U_\lambda$ and $\pi_\lambda$ are irreducible representations of
$U(d)$ and $S_n$, respectively.

Since $\rho_\theta^{\otimes n}$ commutes with every permutation $P_\pi$,
Schur-Weyl duality implies that it belongs to the commutant of the
$S_n$-action.  Hence there exist positive operators
$R_{\theta,\lambda}$ on $\mathcal U_\lambda^{(d)}$ such that
\begin{equation}
    \rho_\theta^{\otimes n}
    =
    \bigoplus_{\lambda\in\Lambda}
    R_{\theta,\lambda}
    \otimes
    I_{\mathcal V_\lambda}.
    \label{eq:permutation-block-form}
\end{equation}

Using \eqref{averaged-n-qudit}, \eqref{sw_action_decomps}, and \eqref{eq:permutation-block-form}, we have
\[
    \bar\rho_\theta^{(n)}
    =
    \bigoplus_{\lambda\in\Lambda}
    \left\{
        \int_{U(d)}
        U_\lambda(U)
        R_{\theta,\lambda}
        U_\lambda(U)^\dagger\,
        \Haar(dU)
    \right\}
    \otimes
    I_{\mathcal V_\lambda}.
\]
For each $\lambda$, the operator inside braces commutes with
$U_\lambda(U)$ for every $U\in U(d)$.  Since $U_\lambda$ is irreducible,
Schur's lemma implies that it is a scalar multiple of the identity on
$\mathcal U_\lambda^{(d)}$.  Taking traces determines the scalar:
\[
    \int_{U(d)}
    U_\lambda(U)
    R_{\theta,\lambda}
    U_\lambda(U)^\dagger\,
    \Haar(dU)
    =
    \frac{\Tr R_{\theta,\lambda}}
         {\dim\mathcal U_\lambda^{(d)}}
    I_{\mathcal U_\lambda^{(d)}}.
\]

Let $\Pi_\lambda$ denote the projection onto
$\mathcal U_\lambda^{(d)}\otimes\mathcal V_\lambda$, and define
\[
    P_\theta^{(n)}(\lambda)
    :=
    \Tr\!\left(
        \rho_\theta^{\otimes n}\Pi_\lambda
    \right).
\]

From \eqref{eq:permutation-block-form},
\[
    P_\theta^{(n)}(\lambda)
    =
    \dim(\mathcal V_\lambda)\,
    \Tr R_{\theta,\lambda}.
\]
Therefore
\[
    \Tr R_{\theta,\lambda}
    =
    \frac{P_\theta^{(n)}(\lambda)}
         {\dim\mathcal V_\lambda},
\]
and we obtain the normalized block decomposition
\begin{equation}
\label{eq:wss-block-factorization}
    \bar\rho_\theta^{(n)}
    =
    \bigoplus_{\lambda\in\Lambda}
    P_\theta^{(n)}(\lambda)
    \frac{I_{\mathcal U_\lambda^{(d)}}}
         {\dim\mathcal U_\lambda^{(d)}}
    \otimes
    \frac{I_{\mathcal V_\lambda}}
         {\dim\mathcal V_\lambda}.
\end{equation}

Thus the parameter $\theta$ enters the orbit-averaged experiment only
through the block probabilities $P_\theta^{(n)}(\lambda)$, while the
normalized conditional state within each Schur-Weyl block is
\[
    \tau_\lambda
    :=
    \frac{I_{\mathcal U_\lambda^{(d)}}}
         {\dim\mathcal U_\lambda^{(d)}}
    \otimes
    \frac{I_{\mathcal V_\lambda}}
         {\dim\mathcal V_\lambda},
\]
which is independent of $\theta$. The conclusion follows from
Theorem~\ref{thm:invariant-factorization}.
\end{proof}

\begin{remark}
The theorem does not assert
that weak Schur sampling dominates arbitrary  measurements.
Rather, it identifies weak Schur sampling as a sufficient classical
reduction of the quantum experiment relative to the invariant measurement
class appropriate to spectral inference.
\end{remark}

\subsection{Decision-theoretic reductions}

\subsubsection{Bayes risk under invariant prior}

So far, we have considered only the class of invariant POVMs and have shown that WSS is quantum-to-classically sufficient relative to this class. Although this restriction may initially seem arbitrary, we show in this subsection that, for an invariant loss and a unitarily invariant prior, optimization may be restricted to invariant POVMs without changing the optimal Bayes risk.

For the following symmetrization statements, let the compact group
\(G\) also act on the action space \((\mathsf A,\mathcal A)\) through
a measurable group action.

\begin{definition}[Invariant loss and invariant prior]
The loss is $G$-invariant if
\[
    L(g\theta,ga)=L(\theta,a)
    \qquad\text{for all }g\in G.
\]
On the other hand a prior $\Pi$ is $G$-invariant if
\[
    \Pi(gB)=\Pi(B)
\]
for every measurable $B\subseteq\Theta$ and every $g\in G$.
\end{definition}

\begin{definition}[Equivariant procedure]
An action-valued POVM $D$ is equivariant if
\[
    U_gD(A)U_g^{\dagger}=D(gA)
\]
for all $A\in\cA$ and $g\in G$.
\end{definition}

Note that for spectrum inference, we can take
$G=U(d),
    g\rho:=g\rho g^{\dagger}.$
The orbit of $\rho$ is
\[
    [\rho]:=\{U\rho U^{\dagger}:U\in U(d)\}.
\]
The orbit space is naturally identified with $\Theta_{\downarrow}$ through the map
$\rho\mapsto\mathrm{Spec}(\rho)$.

\begin{definition}[Spectrum-only loss]
A loss is spectrum-only if there exists
\[
    \ell:\Theta_{\downarrow}\times\mathsf A\to[0,\infty]
\]
such that
\[
    L(\rho,a)=\ell(\mathrm{Spec}(\rho),a).
\]
\end{definition}

The action space need not itself be $\Theta_{\downarrow}$.  For example:
\begin{itemize}
    \item for spectrum estimation, $\mathsf A=\Theta_{\downarrow}$;
    \item for entropy estimation, $\mathsf A=\R$;
    \item for testing, $\mathsf A=\{0,1\}$;
    \item for estimating several spectral functionals, $\mathsf A=\R^k$.
\end{itemize}

Let $M$ be a POVM on the measurable outcome space
$(\mathsf Y,\mathcal Y)$.  Its unitary twirl is
\[
    \overline M(B)
    :=\int_G U_g^{\dagger} M(B)U_g\,\mu_G(dg).
\]


The following proposition is immediate.
\begin{prop}[Symmetrization of a POVM]\label{prop:symmetrization}
The twirled family $\overline M$ is a POVM and is $G$-invariant.
\end{prop}

The next result shows that, under a $G$-invariant prior and a
$G$-invariant loss, the Bayes optimization may be restricted to
$G$-equivariant decision procedures without changing the Bayes risk.
When the action space carries the trivial $G$-action, equivariance reduces
to invariance, so the optimization may then be restricted to invariant
POVMs. The proof is deferred to the appendix.

\begin{theorem}[Bayes-risk preservation under symmetrization]
\label{thm:symmetrization}
Assume that the experiment is $G$-equivariant, the prior is $G$-invariant,
and the loss is $G$-invariant. Then every decision procedure admits a
$G$-equivariant symmetrization with the same Bayes risk.
\end{theorem}

\begin{remark}
For the spectrum-only decision problems considered below, the action space
carries the trivial $U(d)$-action, i.e.
\[
    Ua=a,
    \qquad
    U\in U(d),\quad a\in\mathsf A.
\]
Equivalently, $UC=C$ for every measurable $C\subseteq\mathsf A$.
Hence a $U(d)$-equivariant action-valued POVM $D$, which satisfies
\[
    U^{\otimes n}D(C)(U^\dagger)^{\otimes n}
    =
    D(UC),
\]
is simply $U(d)$-invariant:
\[
    U^{\otimes n}D(C)(U^\dagger)^{\otimes n}
    =
    D(C).
\]
This applies, for example, when the decision is an estimate of the spectrum,
a spectral functional such as entropy, or a decision in a testing problem
whose hypotheses depend only on the spectrum.  In these cases,
Theorem~\ref{thm:symmetrization} therefore permits the Bayes optimization
to be restricted to invariant action-valued POVMs.
\end{remark}

Let \(\widetilde\Pi\) be a unitarily invariant prior on
\(\mathcal S(\mathbb C^d)\), and let
\[
\Pi
:=
\widetilde\Pi\circ\spec^{-1}
\]
denote the induced prior on \(\Theta_{\downarrow}\).
Given the weak Schur sampling outcome \(\lambda\), the posterior
distribution of the spectrum is
\[
\Pi(d\theta\mid\lambda)
=
\frac{P_\theta^{(n)}(\lambda)\Pi(d\theta)}
{\int_{\Theta_{\downarrow}}
P_\theta^{(n)}(\lambda)\Pi(d\theta)}.
\]
This formula applies to outcomes \(\lambda\) having positive
prior-predictive probability; on outcomes of prior-predictive
probability zero, the posterior may be defined arbitrarily.

\begin{definition}[Posterior Bayes action]
A posterior Bayes action is any measurable selection
\[
    a_\Pi(\lambda)
    \in\argmin_{a\in\mathsf A}
      \int_{\Theta_{\downarrow}}\ell(\theta,a)
      \Pi(d\theta\mid\lambda).
\]
\end{definition}

\begin{theorem}[Bayes optimality of weak Schur sampling]\label{thm:WSS_bayes-optimal}

Let \(\widetilde\Pi\) be a unitarily invariant prior on the state
space, and suppose that
\[
L(\rho,a)=\ell(\spec(\rho),a).
\]
Then, for every quantum decision procedure \(D\), there
exists a randomized decision rule \(\delta_D\) based on weak Schur
sampling such that
\[
R_{\widetilde\Pi}(D)
=
R_{\widetilde\Pi}(M^{\mathrm{WSS}},\delta_D).
\]
Consequently,
\[
\inf_D R_{\widetilde\Pi}(D)
=
\inf_\delta
R_{\widetilde\Pi}(M^{\mathrm{WSS}},\delta).
\]
If a measurable posterior Bayes action \(a_\Pi(\lambda)\) exists, then
weak Schur sampling followed by \(a_\Pi\) attains this common infimum
and is therefore Bayes optimal among all quantum decision
procedures.
\end{theorem}

\begin{proof}
Let \(D\) be an arbitrary quantum decision procedure.
Because the action space carries the trivial \(U(d)\)-action,
Theorem~\ref{thm:symmetrization} produces a \(U(d)\)-invariant
action-valued POVM \(\overline D\) satisfying
\[
R_{\widetilde\Pi}(\overline D)
=
R_{\widetilde\Pi}(D).
\]
By Theorem~\ref{thm:wss} and
Proposition~\ref{prop:risk-transfer}, there exists a randomized
decision rule \(\delta_{\overline D}\) based on weak Schur sampling
whose pointwise risk agrees with that of \(\overline D\). Hence
\[
R_{\widetilde\Pi}
(M^{\mathrm{WSS}},\delta_{\overline D})
=
R_{\widetilde\Pi}(D).
\]
It follows that optimizing over all quantum procedures is
equivalent to optimizing over classical decision rules based on the
weak Schur sampling outcome. The latter is an ordinary classical Bayes problem. When a measurable
posterior Bayes action exists, it minimizes that problem and
therefore attains the common infimum.
\end{proof}

\begin{remark}
The weak Schur sampling measurement itself does not depend on the
particular invariant prior or spectrum-only loss; these enter only through
the classical decision rule applied to the observed Young diagram.
Thus Bayes optimality refers to weak Schur sampling followed by an
appropriate posterior Bayes action. In particular, the commonly used
estimator $\lambda/n$ need not be Bayes optimal for a general prior and loss.
\end{remark}

\subsubsection{Minimax reduction}
\label{sec:wss-minimax}

We retain the notation and assumptions of the preceding subsection.
For \(\btheta\in\Theta_{\downarrow}\) and \(U\in U(d)\), write

\[
\rho_{\btheta}:=\diag(\theta_1,\ldots,\theta_d), \quad   \rho_{\btheta,U}:=U\rho_{\btheta}U^{\dagger}.
\]

The spectrum \(\btheta\) is the parameter of interest and \(U\) is a
nuisance parameter. Let \(\mathsf K\subseteq\Theta_{\downarrow}\) be nonempty, let
\((\mathsf A,\mathcal A)\) be a standard Borel action space, and let
\[
    \ell:\Theta_{\downarrow}\times\mathsf A\rightarrow[0,\infty]
\] 
be a measurable spectrum-only loss. For an action-valued POVM \(D_n\),
define
\[
    R_n(\btheta,U;D_n)
    :=\int_{\mathsf A}\ell(\btheta,a)
      \Tr\!\left[\rho_{\btheta,U}^{\otimes n}D_n(da)\right]
\]
and
\[
    \mathfrak R_n^{\mathrm{all}}(\mathsf K;\ell)
    :=\inf_{D_n}\sup_{\btheta\in\mathsf K}\sup_{U\in U(d)}
      R_n(\btheta,U;D_n).
\]

A randomized WSS decision rule is a Markov kernel \(\delta_n\) from
\(\mathbb Y_{n,d}\) to \(\mathsf A\). Define
\[
    R_n^{\WSS}(\btheta;\delta_n)
    :=\sum_{\lambda\in\mathbb Y_{n,d}}
      P_{\btheta}^{(n)}(\lambda)
      \int_{\mathsf A}\ell(\btheta,a)\,\delta_n(\lambda,da)
\]
and
\[
    \mathfrak R_n^{\WSS}(\mathsf K;\ell)
    :=\inf_{\delta_n}\sup_{\btheta\in\mathsf K}
      R_n^{\WSS}(\btheta;\delta_n).
\]

\subsubsection{Exact minimax reduction}

For an action-valued POVM \(D_n\), define its unitary twirl by
\[
    \overline D_n(B)
    :=\int_{U(d)}(V^{\otimes n})^{\dagger}D_n(B)V^{\otimes n}\,\Haar(dV),
    \qquad B\in\mathcal A.
\]
Because the action space carries the trivial \(U(d)\)-action,
\(\overline D_n\) is a \(U(d)\)-invariant POVM.

\begin{prop}[Minimax symmetrization]
\label{prop:minimax-symmetrization}
For every action-valued POVM \(D_n\),
\[
    R_n(\btheta,U;\overline D_n)
    =\int_{U(d)}R_n(\btheta,V;D_n)\,\Haar(dV)
\]
for every \(\btheta\in\Theta_{\downarrow}\) and \(U\in U(d)\).
Consequently,
\[
    \sup_{\btheta\in\mathsf K}\sup_{U\in U(d)}
      R_n(\btheta,U;\overline D_n)
    \leq
    \sup_{\btheta\in\mathsf K}\sup_{U\in U(d)}
      R_n(\btheta,U;D_n),
\]
and
\begin{equation}
    \mathfrak R_n^{\mathrm{all}}(\mathsf K;\ell)
    =\inf_{\substack{D_n:\,D_n\ \mathrm{is}\ U(d)\text{-}\mathrm{invariant}}}
      \sup_{\btheta\in\mathsf K}R_n(\btheta,I_d;D_n).
\label{eq:minimax-invariant-reduction}
\end{equation}
\end{prop}

\begin{proof}
Tonelli's theorem, cyclicity of the trace, and
\(V\rho_{\btheta,U}V^{\dagger}=\rho_{\btheta,VU}\) give
\[
\begin{aligned}
    R_n(\btheta,U;\overline D_n)
    &=\int_{U(d)}R_n(\btheta,VU;D_n)\,\Haar(dV)\\
    &=\int_{U(d)}R_n(\btheta,V;D_n)\,\Haar(dV),
\end{aligned}
\]
where the second equality follows from right invariance of Haar
measure. The risk inequality follows immediately. Taking the infimum
over all POVMs and over the invariant subclass proves
\eqref{eq:minimax-invariant-reduction}.
\end{proof}

\begin{theorem}[Exact minimax reduction to weak Schur sampling]
\label{thm:wss-exact-minimax}
For every \(n\), every nonempty
\(\mathsf K\subseteq\Theta_{\downarrow}\), and every nonnegative
spectrum-only loss \(\ell\),
\[
    \mathfrak R_n^{\mathrm{all}}(\mathsf K;\ell)
    =\mathfrak R_n^{\WSS}(\mathsf K;\ell).
\]
\end{theorem}

\begin{proof}
If \(\delta_n\) is a decision rule based on WSS, then
\[
    D_{n,\delta_n}(B)
    :=\sum_{\lambda\in\mathbb Y_{n,d}}
      \delta_n(\lambda,B)\Pi_\lambda
\]
is an action-valued POVM and
\[
    R_n(\btheta,U;D_{n,\delta_n})
    =R_n^{\WSS}(\btheta;\delta_n)
\]
for every \((\btheta,U)\). Hence
\[
    \mathfrak R_n^{\mathrm{all}}(\mathsf K;\ell)
    \leq\mathfrak R_n^{\WSS}(\mathsf K;\ell).
\]

Conversely, let \(D_n\) be arbitrary and let \(\overline D_n\) be its
twirl. By Theorem~\ref{thm:wss}, there is a parameter-independent
Markov kernel \(\delta_{\overline D_n}\) such that
\[
    P_\rho^{\overline D_n}
    =\delta_{\overline D_n}P_\rho^{\WSS}
\]
for every density operator \(\rho\). Therefore
\[
    R_n^{\WSS}(\btheta;\delta_{\overline D_n})
    =R_n(\btheta,U;\overline D_n)
\]
for every \((\btheta,U)\). Proposition~\ref{prop:minimax-symmetrization}
then gives
\[
\begin{aligned}
    \mathfrak R_n^{\WSS}(\mathsf K;\ell)
    &\leq
    \sup_{\btheta\in\mathsf K}\sup_{U\in U(d)}
      R_n(\btheta,U;\overline D_n)\\
    &\leq
    \sup_{\btheta\in\mathsf K}\sup_{U\in U(d)}
      R_n(\btheta,U;D_n).
\end{aligned}
\]
Taking the infimum over \(D_n\) proves the reverse inequality.
\end{proof}


Let \(\psi:\Theta_{\downarrow}\to\R^q\) be a spectral functional and
consider squared Euclidean loss
\[
    \ell^2_\psi(\btheta,\ba):=\|\psi(\btheta)-\ba\|^2.
\]
For squared loss, write
\[
\begin{aligned}
\mathfrak R_{n,\psi}^{\mathrm{all}}(\mathsf K)
&:=\mathfrak R_n^{\mathrm{all}}(\mathsf K;\ell_\psi^2),\\
\mathfrak R_{n,\psi}^{\WSS}(\mathsf K)
&:=\mathfrak R_n^{\WSS}(\mathsf K;\ell_\psi^2).
\end{aligned}
\]

\subsection{Asymptotic risk of spectral functionals}

Recall that
\[
    \Theta_{\downarrow}
    :=\left\{
      \btheta=(\theta_1,\ldots,\theta_d)\in[0,1]^d:
      \theta_1\geq\cdots\geq\theta_d,\quad
      \sum_{i=1}^d\theta_i=1
    \right\}
\]
is the ordered probability simplex, and let
\[
    \Theta_{\mathrm{reg}}
    :=\left\{
      \btheta\in\Theta_{\downarrow}:
      \theta_1>\cdots>\theta_d>0
    \right\}
\]
be its regular part. Throughout, \(d\) is fixed.

Weak Schur sampling applied to $\rho_{\btheta}^{\otimes n}$ produces
\[
    \Lambda_n=(\Lambda_{n,1},\ldots,\Lambda_{n,d})
    \in\mathbb Y_{n,d}.
\]

Write \(P_{\btheta}^{(n)}\) for the law of \(\Lambda_n\). By
Schur-Weyl duality,
\begin{equation}
    P_{\btheta}^{(n)}(\lambda)
    =\dim(\cV_\lambda)s_\lambda(\btheta),
    \qquad \lambda\in\mathbb Y_{n,d},
\label{eq:wss-law-bayes}
\end{equation}
where \(\cV_\lambda\) is the irreducible \(S_n\)-module indexed by
\(\lambda\), and \(s_\lambda\) is the Schur polynomial. Define the
empirical Young-diagram estimator
\[
    \widehat{\btheta}_n:=\frac{\Lambda_n}{n}.
\]

For a prior \(\Pi\) on \(\Theta_{\downarrow}\), define
\[
    r_{n,\psi}^{\WSS}(\Pi)
    :=\inf_\delta
      \int_{\Theta_{\downarrow}}
      \E_{\btheta}^{\WSS}
      \|\delta(\Lambda_n)-\psi(\btheta)\|^2
      \,\Pi(d\btheta).
\]
Under squared loss, the Bayes rule is the posterior mean, and hence
\[
    r_{n,\psi}^{\WSS}(\Pi)
    =\E\!\left[\Tr\Var\{\psi(\btheta)\mid\Lambda_n\}\right].
\]
Set

\begin{equation}
    \Sigma(\btheta)
      :=\diag(\btheta)-\btheta\btheta^\top, \qquad 
    \cI_\psi(\btheta)
      :=\Tr\!\left[
        D\psi(\btheta)\Sigma(\btheta)D\psi(\btheta)^\top
      \right].\label{eq:wss-risk-constant}
\end{equation}

\subsubsection{Assumptions and main theorem}

\begin{assumption}[Prior]
\label{ass:prior}
The prior \(\Pi\) is absolutely continuous with respect to
\((d-1)\)-dimensional Lebesgue measure on the simplex. Write
\[
    \Pi(d\btheta)=\pi(\btheta)\,d\btheta.
\]
For \(\Pi\)-almost every \(\btheta\in\Theta_{\mathrm{reg}}\) with
\(\pi(\btheta)>0\), the density \(\pi\) is continuous at \(\btheta\).
\end{assumption}

\begin{assumption}[Smooth functional]
\label{ass:psi}
The function \(\psi\) extends to a continuously differentiable map on
an open neighborhood of the closed simplex \(\Theta_{\downarrow}\).
\end{assumption}

Since \(\Theta_{\downarrow}\) is compact and convex,
Assumption~\ref{ass:psi} implies that
\begin{equation}
    L_\psi
    :=\sup_{\btheta\in\Theta_{\downarrow}}
      \|D\psi(\btheta)\|_{\mathrm{op}}<\infty
\label{eq:psi-lipschitz-constant}
\end{equation}
and that \(\psi\) is \(L_\psi\)-Lipschitz on
\(\Theta_{\downarrow}\).

\begin{theorem}[Bayesian risk asymptotics for WSS]
\label{thm:WSS_Bayes_asymp}
Suppose Assumptions~\ref{ass:prior} and~\ref{ass:psi} hold. Then
\begin{equation}
    r_{n,\psi}^{\WSS}(\Pi)
    =\frac1n
      \int_{\Theta_{\downarrow}}
      \cI_\psi(\btheta)\,\Pi(d\btheta)
      +o(n^{-1}).
\label{eq:wss-bayes-asymptotics}
\end{equation}
\end{theorem}

The lower bound follows from an exact randomization from the multinomial
experiment to WSS. In particular, every decision rule based on WSS can therefore be
implemented in the more informative multinomial experiment. Classical
posterior asymptotics for the multinomial model then yield the required
Bayes lower bound. For the upper bound, we analyze the explicit estimator
$\psi(\Lambda_n/n)$. A quantitative comparison between the WSS and
multinomial laws yields its pointwise squared-risk expansion, together with
a uniform version on compact subsets of $\Theta_{\mathrm{reg}}$. A global
risk bound permits integration against the prior. The same uniform
expansion is subsequently used in the minimax analysis.

Combining Theorem \ref{thm:WSS_bayes-optimal} and Theorem \ref{thm:WSS_Bayes_asymp} we obtain the following theorem.
\begin{theorem}[Bayes risk asymptotics over all POVMs]\label{thm:bayes-optimal}

Let \(\widetilde\Pi\) be a unitarily invariant prior on \(\mathcal S(\mathbb C^d)\), let \(\Pi=\widetilde\Pi\circ\spec^{-1}\), and consider squared loss \(\ell^2_\psi(\rho,a)=\|\psi(\spec\rho)-a\|^2\).
Suppose the assumptions~\ref{ass:prior} and~\ref{ass:psi} hold. Then
\[
    \inf_D R_{\widetilde\Pi}(D)
    =\frac1n
      \int_{\Theta_{\downarrow}}
      \cI_\psi(\btheta)\,\Pi(d\btheta)
      +o(n^{-1}).
\]
Moreover, for every $n$, the Bayes optimum over all quantum
decision procedures is attained by weak Schur sampling followed by the
posterior mean
\[
    \widehat\psi_\Pi(\lambda)
    =
    \E_\Pi[\psi(\btheta)\mid\lambda].
\]
\end{theorem}

\subsubsection{Asymptotic minimax risk}

We next derive the first-order minimax risk over compact subsets of the
regular spectrum region.  The restriction to
$\Theta_{\mathrm{reg}}$ is required only for the uniform asymptotic
comparison used below; the exact finite-sample minimax reduction to WSS
continues to hold without this restriction. Let
\[
    \mathsf H:=\{x\in\R^d:\bone^\top x=1\},
\]
and write \(\operatorname{int}_{\mathsf H}\) for interior relative to
\(\mathsf H\).

\begin{theorem}[Asymptotic minimax risk]
\label{thm:wss-global-asymptotic-minimax}
Suppose that Assumption~\ref{ass:psi} holds. Let
$\mathsf K\subset\Theta_{\mathrm{reg}}$ be nonempty and compact,
and suppose that
\begin{equation}
\mathsf K
=\overline{\operatorname{int}_{\mathsf H}(\mathsf K)},
\label{eq:regular-set-interior-condition}
\end{equation}
where the closure is also taken relative to $\mathsf H$. Then
\begin{equation}
    \mathfrak R_{n,\psi}^{\mathrm{all}}(\mathsf K)
    =\mathfrak R_{n,\psi}^{\WSS}(\mathsf K)
    =\frac1n\sup_{\btheta\in\mathsf K}\cI_\psi(\btheta)
      +o(n^{-1}).
\label{eq:wss-global-asymptotic-minimax}
\end{equation}
Moreover, \(\widehat\psi_n=\psi(\Lambda_n/n)\) is asymptotically
minimax on \(\mathsf K\).
\end{theorem}

\begin{proof}
The finite-sample equality follows from
Theorem~\ref{thm:wss-exact-minimax}. For the upper bound, use \(\widehat\psi_n=\psi(\Lambda_n/n)\). By
Lemma~\ref{lem:uniform-wss-functional-risk},
\begin{equation}
    \limsup_{n\to\infty}
      n\,\mathfrak R_{n,\psi}^{\WSS}(\mathsf K)
    \leq
    \limsup_{n\to\infty}
      \sup_{\btheta\in\mathsf K}
      n\,\E_{\btheta}^{\WSS}
      \|\widehat\psi_n-\psi(\btheta)\|^2
    =\sup_{\btheta\in\mathsf K}\cI_\psi(\btheta). \label{eq:global-minimax-upper}
\end{equation}

For the lower bound, set
\[
    J_{\mathsf K}:=\sup_{\btheta\in\mathsf K}\cI_\psi(\btheta)
\]
and fix \(\varepsilon>0\). By continuity of \(\cI_\psi\), compactness
of \(\mathsf K\), and
\eqref{eq:regular-set-interior-condition}, there exist
\(\btheta_\varepsilon\in
\operatorname{int}_{\mathsf H}(\mathsf K)\) and a relative open ball
\(B_{\mathsf H}(\btheta_\varepsilon,r_\varepsilon)\) whose closure is
contained in \(\operatorname{int}_{\mathsf H}(\mathsf K)\), such that
\[
    \cI_\psi(\btheta)\geq J_{\mathsf K}-\varepsilon
\]
throughout that ball. Let \(\Pi_\varepsilon\) have a smooth probability
density, with respect to \((d-1)\)-dimensional Lebesgue measure on
\(\mathsf H\), supported in the ball. Then
\(\Pi_\varepsilon\) satisfies Assumption~\ref{ass:prior}. Since maximum
risk dominates Bayes risk,
\[
    \mathfrak R_{n,\psi}^{\WSS}(\mathsf K)
    \geq r_{n,\psi}^{\WSS}(\Pi_\varepsilon).
\]
Theorem~\ref{thm:WSS_Bayes_asymp} therefore gives
\[
\begin{aligned}
    \liminf_{n\to\infty}
      n\,\mathfrak R_{n,\psi}^{\WSS}(\mathsf K)
    &\geq
    \lim_{n\to\infty}
      n\,r_{n,\psi}^{\WSS}(\Pi_\varepsilon)\\
    &=\int\cI_\psi(\btheta)\,\Pi_\varepsilon(d\btheta)\\
    &\geq J_{\mathsf K}-\varepsilon.
\end{aligned}
\]
Letting \(\varepsilon\downarrow0\) proves the lower bound and hence
the theorem.
\end{proof}

\subsubsection{Consequences and the role of regularity}

\begin{corollary}[Spectrum estimation]
\label{cor:wss-minimax-spectrum}
Let \(\psi(\btheta)=\btheta\).
\begin{enumerate}[label=\textup{(\roman*)}]
\item Under Assumption~\ref{ass:prior},
\[
    r_{n,\btheta}^{\WSS}(\Pi)
    =\frac1n
      \int_{\Theta_{\downarrow}}
      \bigl(1-\|\btheta\|^2\bigr)\,\Pi(d\btheta)
      +o(n^{-1}).
\]
\item If \(\mathsf K\) satisfies the assumptions of
Theorem~\ref{thm:wss-global-asymptotic-minimax}, then
\[
    \mathfrak R_{n,\btheta}^{\mathrm{all}}(\mathsf K)
    =\mathfrak R_{n,\btheta}^{\WSS}(\mathsf K)
    =\frac1n\sup_{\btheta\in\mathsf K}
      \bigl(1-\|\btheta\|^2\bigr)+o(n^{-1}).
\]
\end{enumerate}
\end{corollary}

\begin{proof}
For the identity functional, \(D\psi(\btheta)=I_d\), and therefore
\[
    \cI_\psi(\btheta)
    =\Tr\Sigma(\btheta)
    =1-\|\btheta\|^2.
\]
Apply Theorems~\ref{thm:WSS_Bayes_asymp} and
\ref{thm:wss-global-asymptotic-minimax}.
\end{proof}

\begin{corollary}[Scalar spectral functional]
Suppose that $\Pi$ satisfies Assumption~\ref{ass:prior}, and let
$f:\Theta_{\downarrow}\to\mathbb R$ satisfy
Assumption~\ref{ass:psi}. Then
\[
r_{n,f}^{\WSS}(\Pi)
=\frac1n\int_{\Theta_{\downarrow}}
\nabla f(\btheta)^\top
\{\diag(\btheta)-\btheta\btheta^\top\}
\nabla f(\btheta)\,\Pi(d\btheta)
+o(n^{-1}).
\]
\end{corollary}
\begin{remark}[Purity estimation]
A basic nonlinear spectral functional is the purity
\[
    f(\btheta)
    :=
    \Tr(\rho_{\btheta}^2)
    =
    \sum_{i=1}^d \theta_i^2.
\]
For a $d$-dimensional state,
\[
    \frac{1}{d}\leq f(\btheta)\leq 1,
\]
with $f(\btheta)=1$ if and only if the state is pure, while
$f(\btheta)=1/d$ for the maximally mixed state. Thus purity provides a
natural measure of the degree of mixedness of the state.

Since
$\nabla f(\btheta)=2\btheta,
$
the corresponding asymptotic risk constant is
\begin{align*}
    \cI_f(\btheta)
    &=
    \nabla f(\btheta)^\top
    \Sigma(\btheta)
    \nabla f(\btheta)\\
    &=
    4\btheta^\top
    \bigl\{
        \diag(\btheta)-\btheta\btheta^\top
    \bigr\}
    \btheta\\
    &=
    4\left\{
        \sum_{i=1}^d\theta_i^3
        -
        \left(\sum_{i=1}^d\theta_i^2\right)^2
    \right\}.
\end{align*}
Consequently, under Assumption~\ref{ass:prior},
Theorem~\ref{thm:WSS_Bayes_asymp} gives
\[
    r_{n,f}^{\WSS}(\Pi)
    =
    \frac{4}{n}
    \int_{\Theta_{\downarrow}}
    \left\{
        \sum_{i=1}^d\theta_i^3
        -
        \left(\sum_{i=1}^d\theta_i^2\right)^2
    \right\}
    \Pi(d\btheta)
    +o(n^{-1}).
\]
Likewise, for a compact set $\mathsf K$ satisfying the assumptions of
Theorem~\ref{thm:wss-global-asymptotic-minimax},
\[
    \mathfrak R_{n,f}^{\mathrm{all}}(\mathsf K)
    =
    \frac{4}{n}
    \sup_{\btheta\in\mathsf K}
    \left\{
        \sum_{i=1}^d\theta_i^3
        -
        \left(\sum_{i=1}^d\theta_i^2\right)^2
    \right\}
    +o(n^{-1}).
\]

Moreover, by Lemma~\ref{lem:uniform-wss-functional-risk} and the
dominated-convergence argument used in the proof of
Theorem~\ref{thm:WSS_Bayes_asymp}, the plug-in estimator
\[
    \widehat f_n
    :=
    f(\Lambda_n/n)
    =
    \sum_{i=1}^d
    \left(\frac{\Lambda_{n,i}}{n}\right)^2
\]
attains the optimal Bayes risk up to $o(n^{-1})$ for every prior
satisfying Assumption~\ref{ass:prior}. By
Theorem~\ref{thm:wss-global-asymptotic-minimax}, it is also
asymptotically minimax on $\mathsf K$.
\end{remark}

\begin{remark}[Pointwise, integrated, and uniform uses of regularity]

The asymptotic comparison used above is uniform on compact subsets of
$\Theta_{\mathrm{reg}}$, but its constants deteriorate as an eigenvalue
approaches zero or two eigenvalues collide. This causes no difficulty for
Theorem~\ref{thm:WSS_Bayes_asymp}, since an absolutely continuous prior
assigns zero mass to the singular strata. The minimax problem is different,
because it involves a supremum over the parameter space. We therefore
restrict $\mathsf K$ to a compact subset of $\Theta_{\mathrm{reg}}$ and use
the uniform form of the risk expansion.

A typical admissible parameter set is
\[
    \mathsf K_{\varepsilon,\gamma}
    :=\left\{
      \btheta\in\Theta_{\downarrow}:
      \theta_d\geq\varepsilon,\quad
      \theta_i-\theta_{i+1}\geq\gamma,
      \quad 1\leq i<d
    \right\},
\]
provided it has nonempty interior relative to \(\mathsf H\). It is
then compact, contained in \(\Theta_{\mathrm{reg}}\), and equal to the
closure of its relative interior.
\end{remark}

\begin{remark}[Singular spectra]
The finite-sample equality
\[
    \mathfrak R_n^{\mathrm{all}}(\mathsf K;\ell)
    =
    \mathfrak R_n^{\WSS}(\mathsf K;\ell)
\]
continues to hold when $\mathsf K$ contains repeated or zero eigenvalues.
The regular first-order risk formula need not. At an eigenvalue collision,
the fluctuations of the ordered WSS spectrum generally involve eigenvalues
of Gaussian Hermitian blocks \cite{Meliot12}, rather than an ordinary
Gaussian vector. Such singularities require a separate local limit and
minimax analysis, which is beyond the scope of the present paper.
\end{remark}

\section{Sufficiency for thermal functionals of quantum Gaussian states}
\label{sec:thermal}
In this section, we study a second example in which symmetry reduction
yields a sufficient POVM. Recall the definition of thermal states given in \eqref{eq:thermal-state}.
 We consider $n\geq 2$ identical copies of the thermal state with an unknown common displacement. The corresponding experiment is:
\begin{equation}
    \cQ_n
    :=
    \left\{
        \rho_{z,N}^{\otimes n}
        :
        z\in\bbC,\ N>0
    \right\}.
    \label{eq:ncopy-model}
\end{equation}
The parameter of interest is $N$, or any smooth functional of $N$, while
$z\in\bbC$
is an unknown nuisance displacement.

\subsection{The nuisance parameter and invariance}
\label{subsec:gaussian-symmetry}

The displacement group $(\bbC,+)$ acts collectively on the $n$-copy
experiment according to
\begin{equation}
    \rho_{z,N}^{\otimes n}
    \longmapsto
    D(w)^{\otimes n}
    \rho_{z,N}^{\otimes n}
    \bigl(D(w)^{\otimes n}\bigr)^\dagger
    =
    \rho_{z+w,N}^{\otimes n},
    \qquad
    w\in\bbC.
    \label{eq:collective-displacement-action}
\end{equation}
Thus the group acts only on the displacement parameter $z$, while leaving
the parameter of interest $N$ unchanged. The displacement is therefore a
group-generated nuisance parameter.

This motivates the following invariant measurement class.

\begin{definition}[Collective-displacement invariant POVM]
\label{def:collective-displacement-invariant}
A POVM $M$ on $\cH^{\otimes n}$ is called
\emph{collective-displacement invariant} if
\begin{equation}
    \bigl(D(w)^{\otimes n}\bigr)^\dagger
    M(B)
    D(w)^{\otimes n}
    =
    M(B)
    \label{eq:invariant-povm}
\end{equation}
for every $w\in\bbC$ and every measurable outcome event $B$.
We denote the class of such measurements by
$\cM_{\mathrm{inv}}^{(n)}$.
\end{definition}

For every $M\in\cM_{\mathrm{inv}}^{(n)}$, the induced outcome law is
independent of the nuisance displacement. Indeed, since
\[
    \rho_{z,N}^{\otimes n}
    =
    D(z)^{\otimes n}
    \rho_{0,N}^{\otimes n}
    \bigl(D(z)^{\otimes n}\bigr)^\dagger,
\]
cyclicity of the trace and \eqref{eq:invariant-povm} give
\begin{align}
    P_{z,N}^{M}(B)
    &=
    \Tr\!\left[
        D(z)^{\otimes n}
        \rho_{0,N}^{\otimes n}
        \bigl(D(z)^{\otimes n}\bigr)^\dagger
        M(B)
    \right]
    \nonumber\\
    &=
    \Tr\!\left[
        \rho_{0,N}^{\otimes n}
        \bigl(D(z)^{\otimes n}\bigr)^\dagger
        M(B)
        D(z)^{\otimes n}
    \right]
    \nonumber\\
    &=
    \Tr\!\left[
        \rho_{0,N}^{\otimes n}M(B)
    \right]
    =
    P_{0,N}^{M}(B).
    \label{eq:invariant-law-independent-z}
\end{align}
Hence every measurement in $\cM_{\mathrm{inv}}^{(n)}$ induces a classical
experiment indexed only by the thermal parameter $N$.

We next separate the common displacement from the relative modes.
Kumagai and Hayashi~\cite[Section~2]{KumagaiHayashi2013} construct a
passive unitary transformation, which they call a \emph{concentrating
operator}, that transfers a common displacement of the $n$ modes into a
single collective mode. Their state-concentration relation is precisely
\eqref{eq:concentrated-state} below. We also require the corresponding
transformation of the displacement operators, since this operator identity
will be used to characterize the invariant measurement class.

\begin{lemma}[Concentration of a common displacement]
\label{lem:concentration}
There exists a passive unitary $U_n$ such that, for every
$z\in\bbC$ and $N>0$,
\begin{equation}
    U_n
    \rho_{z,N}^{\otimes n}
    U_n^\dagger
    =
    \rho_{\sqrt n\,z,N}
    \otimes
    \phi_N^{\otimes(n-1)}.
    \label{eq:concentrated-state}
\end{equation}
Moreover, for every $w\in\bbC$,
\begin{equation}
    U_n
    D(w)^{\otimes n}
    U_n^\dagger
    =
    D(\sqrt n\,w)
    \otimes
    I^{\otimes(n-1)}.
    \label{eq:concentrated-action}
\end{equation}
\end{lemma}

The state identity \eqref{eq:concentrated-state} is the concentration
relation of Kumagai and Hayashi~\cite[Eq.~(2.5)]{KumagaiHayashi2013}.
For completeness, a direct proof of Lemma~\ref{lem:concentration}, based
on their two-mode Hamiltonian construction, is given in
the appendix. In particular, the appendix derives
the displacement covariance \eqref{eq:concentrated-action}, which will be
used explicitly below.

Lemma~\ref{lem:concentration} separates the $n$-copy model into a
collective mode carrying the nuisance displacement and $n-1$ relative modes
that depend only on the thermal parameter:
\begin{equation}
    U_n
    \rho_{z,N}^{\otimes n}
    U_n^\dagger
    =
    \underbrace{
        \rho_{\sqrt n\,z,N}
    }_{\substack{\text{collective mode}\\\text{contains }z}}
    \otimes
    \underbrace{
        \phi_N^{\otimes(n-1)}
    }_{\substack{\text{relative modes}\\\text{depend only on }N}}.
    \label{eq:collective-relative}
\end{equation}

The statistical importance of this decomposition goes beyond the fact that
the outcome law of an invariant measurement does not depend on $z$.
After conjugation by $U_n$, collective-displacement invariance becomes
invariance under arbitrary Weyl displacements acting on the first mode
alone. Irreducibility of the one-mode Weyl representation then forces every
invariant measurement to act trivially on the collective mode. Consequently,
the classical experiment induced by any measurement in
$\cM_{\mathrm{inv}}^{(n)}$ is determined entirely by a measurement on the
$n-1$ centered thermal relative modes. We make this statement precise next.

\begin{prop}[Structure of invariant POVMs]
\label{prop:invariant-structure}
Let $M\in\cM_{\mathrm{inv}}^{(n)}$ and define
\[
    \widetilde M(B)
    :=
    U_nM(B)U_n^\dagger.
\]
Then there exists a POVM $M_0$ on $\cH^{\otimes(n-1)}$ such that
\begin{equation}
    \widetilde M(B)
    =
    I\otimes M_0(B)
    \label{eq:invariant-factor-form}
\end{equation}
for every measurable outcome event $B$.
\end{prop}

\begin{proof}
By \eqref{eq:concentrated-action},
\[
    U_nD(w)^{\otimes n}U_n^\dagger
    =
    D(\sqrt n\,w)\otimes I^{\otimes(n-1)}.
\]
Therefore, using the invariance of $M$,
\begin{align*}
\left[
    \widetilde M(B),
    D(\sqrt n\,w)\otimes I^{\otimes(n-1)}
\right]
&=
\left[
    U_nM(B)U_n^\dagger,
    U_nD(w)^{\otimes n}U_n^\dagger
\right]
\\
&=
U_n
\left[
    M(B),
    D(w)^{\otimes n}
\right]
U_n^\dagger
=
0
\end{align*}
for every $w\in\bbC$.

The one-mode Weyl representation
$w\mapsto D(w)$ on $L^2(\mathbb R)$ is irreducible; see
\cite[Chapter~1, Sections~3 and~5]{Folland1989}.
Hence, by Schur's lemma for irreducible unitary representations
\cite[Theorem~3.5]{FollandAbstract},
\[
    \{D(w):w\in\bbC\}'
    =
    \bbC I.
\]
It follows that
\begin{equation}
    \left\{
        D(w)\otimes I^{\otimes(n-1)}
        :
        w\in\bbC
    \right\}'
    =
    I\otimes
    \mathcal L\!\left(\cH^{\otimes(n-1)}\right).
    \label{eq:weyl-tensor-commutant}
\end{equation}

Consequently,
\[
    \widetilde M(B)
    =
    I\otimes M_0(B)
\]
for some bounded positive operator $M_0(B)$ on
$\cH^{\otimes(n-1)}$.

It remains to verify that $B\mapsto M_0(B)$ is a POVM. Since
$
    \widetilde M(\Omega)
    =
    I_{\cH^{\otimes n}}$,  \eqref{eq:invariant-factor-form} gives
$
    I\otimes M_0(\Omega)
    =
    I\otimes I,
$
and hence
$
    M_0(\Omega)=I.
$
Moreover, if $B_1,B_2,\ldots$ are pairwise disjoint, then the weak
operator countable additivity of $\widetilde M$ gives
\[
    I\otimes
    M_0\!\left(\bigcup_{k=1}^{\infty}B_k\right)
    =
    \sum_{k=1}^{\infty}
    I\otimes M_0(B_k)
\]
in the weak operator topology. Taking matrix elements against vectors of
the form $\xi\otimes\eta$, with $\|\xi\|=1$, shows that
\[
    M_0\!\left(\bigcup_{k=1}^{\infty}B_k\right)
    =
    \sum_{k=1}^{\infty}M_0(B_k)
\]
weakly on $\cH^{\otimes(n-1)}$. Thus $M_0$ is a POVM.
\end{proof}

Combining Proposition~\ref{prop:invariant-structure} with
Lemma~\ref{lem:concentration}, for every
$M\in\cM_{\mathrm{inv}}^{(n)}$ we obtain
\begin{align}
    P_{z,N}^{M}(B)
    &=
    \Tr\!\left[
        \rho_{z,N}^{\otimes n}M(B)
    \right]
    \nonumber\\
    &=
    \Tr\!\left[
        \left(
            \rho_{\sqrt n\,z,N}
            \otimes
            \phi_N^{\otimes(n-1)}
        \right)
        \left(
            I\otimes M_0(B)
        \right)
    \right]
    \nonumber\\
    &=
    \Tr\!\left[
        \phi_N^{\otimes(n-1)}
        M_0(B)
    \right].
    \label{eq:reduced-outcome-law}
\end{align}

Thus every collective-displacement invariant measurement on the original
$n$-mode experiment induces exactly the same classical experiment as some
POVM acting on the $n-1$ centered thermal relative modes. Conversely, every
POVM $M_0$ on the relative modes gives an invariant POVM on the original
experiment through
\[
    M(B)
    =
    U_n^\dagger
    \bigl(I\otimes M_0(B)\bigr)
    U_n.
\]
Hence the original quantum decision problem, restricted to the class of invariant measurements, is
exactly equivalent, at the level of induced classical experiments, to the
ordinary measurement problem for the centered thermal family
\[
    \left\{
        \phi_N^{\otimes(n-1)}
        :
        N>0
    \right\}.
\]
The next step is therefore to identify a single measurement on this reduced
family that dominates all other POVMs.

\subsection{Construction of the sufficient measurement}
\label{subsec:number-sufficiency}

Set
$m:=n-1$.
The reduced state is
$\phi_N^{\otimes m} $.
For
\[
    \bm k=(k_1,\ldots,k_m)\in\N_0^m,
\qquad
    |\bm k|
    :=
    \sum_{j=1}^m k_j,
\] write
$ |\bm k\rangle
    =
    |k_1\rangle\otimes\cdots\otimes|k_m\rangle.
$
From \eqref{eq:thermal-state},
\begin{equation}
    \phi_N^{\otimes m}
    =
    \frac{1}{(N+1)^m}
    \sum_{\bm k\in\N_0^m}
    \left(
        \frac{N}{N+1}
    \right)^{|\bm k|}
    |\bm k\rangle\langle\bm k|.
    \label{eq:thermal-product}
\end{equation}

For $s\in\N_0$, define the projection onto the total residual photon number-$s$
subspace by
\begin{equation}
    \Pi_s^{(m)}
    :=
    \sum_{\substack{\bm k\in\N_0^m\\|\bm k|=s}}
    |\bm k\rangle\langle\bm k|.
    \label{eq:total-number-proj}
\end{equation}
Its dimension is
\begin{equation}
    d_{m,s}
    :=
    \Tr\Pi_s^{(m)}
    =
    \binom{s+m-1}{m-1}.
    \label{eq:block-dimension}
\end{equation}
Define
\begin{equation}
    \tau_s^{(m)}
    :=
    \frac{\Pi_s^{(m)}}{d_{m,s}}.
    \label{eq:conditional-block-state}
\end{equation}
Then
$\tau_s^{(m)}
$
does not depend on $N$. Grouping \eqref{eq:thermal-product} according to total residual photon number yields
\begin{equation}
    \phi_N^{\otimes m}
    =
    \bigoplus_{s=0}^{\infty}
    p_N^{(m)}(s)\tau_s^{(m)},
   \label{eq:thermal-factorization}
\end{equation}
where
\begin{equation}
    p_N^{(m)}(s)
    =
    \binom{s+m-1}{m-1}
    \frac{1}{(N+1)^m}
    \left(
        \frac{N}{N+1}
    \right)^s,
    \qquad
    s\in\N_0.
    \label{eq:negative-binomial}
\end{equation}
Thus $p_N^{(m)}$ is a negative-binomial distribution and \eqref{eq:thermal-factorization} is precisely a
parameter-free conditional block factorization: all dependence on $N$ appears
through the classical weights $p_N^{(m)}(s)$.

Define the total-number measurement on the original $n$ modes by
\begin{equation}
    M_n^\star(\{s\})
    :=
    U_n^\dagger
    \left(
        I\otimes\Pi_s^{(n-1)}
    \right)
    U_n,
    \qquad
    s\in\N_0.
    \label{eq:canonical-number-measurement}
\end{equation}

\begin{theorem}[Quantum-to-classical sufficiency for displaced thermal states]
\label{thm:thermal-q2c}
For the experiment
\[
    \cQ_n
    =
    \{
        \rho_{z,N}^{\otimes n}
        :
        z\in\bbC,\ N>0
    \},
\qquad n\ge2,
\]
the measurement $M_n^\star$ in
\eqref{eq:canonical-number-measurement} is quantum-to-classically sufficient
relative to the class
$\cM_{\mathrm{inv}}^{(n)}$ of collective-displacement invariant POVMs.

More precisely, for every
$ M\in\cM_{\mathrm{inv}}^{(n)}$
with outcome space $(\mathsf Y,\mathcal Y)$, there exists a
parameter-independent Markov kernel
$ K_M:
    \N_0\times\mathcal Y
    \rightarrow[0,1]
$
such that
\begin{equation}
    P_{z,N}^{M}
    =
    K_M P_{z,N}^{M_n^\star}
    \label{eq:q2c-kernel}
\end{equation}
for every $z\in\bbC$ and $N>0$.
\end{theorem}

\begin{proof}
By Proposition~\ref{prop:invariant-structure},
\[
    U_nM(B)U_n^\dagger
    =
    I\otimes M_0(B)
\]
for some POVM $M_0$ on the $m=n-1$ relative modes.

Define
\begin{equation}
    K_M(s,B)
    :=
    \Tr\left[
        \tau_s^{(m)}M_0(B)
    \right].
    \label{eq:thermal-kernel}
\end{equation}
Since $\tau_s^{(m)}$ is a density operator and $M_0$ is a POVM,
$B\longmapsto K_M(s,B)$ 
is a probability measure for each $s$. Hence $K_M$ is a Markov kernel.

Using \eqref{eq:reduced-outcome-law} and
\eqref{eq:thermal-factorization},
\[
\begin{aligned}
    P_{z,N}^{M}(B)
    &=
    \Tr\left[
        \phi_N^{\otimes m}M_0(B)
    \right]
    \\
    &=
    \sum_{s=0}^{\infty}
    p_N^{(m)}(s)
    \Tr\left[
        \tau_s^{(m)}M_0(B)
    \right]
    \\
    &=
    \sum_{s=0}^{\infty}
    K_M(s,B)p_N^{(m)}(s).
\end{aligned}
\]
We also verify that
$M_n^\star\in\cM_{\mathrm{inv}}^{(n)}.$
Recall that, for every $s\in\N_0$,
\[
    M_n^\star(\{s\})
    =
    U_n^\dagger
    \left(
        I\otimes\Pi_s^{(n-1)}
    \right)
    U_n.
\]
By \eqref{eq:concentrated-action}, we observe that
\begin{align*}
&\bigl(D(w)^{\otimes n}\bigr)^\dagger
M_n^\star(\{s\})
D(w)^{\otimes n}
\\
=&
U_n^\dagger
\left(
    D(\sqrt n\,w)^\dagger\otimes I^{\otimes(n-1)}
\right)
\left(
    I\otimes\Pi_s^{(n-1)}
\right)
\left(
    D(\sqrt n\,w)\otimes I^{\otimes(n-1)}
\right)
U_n
\\
=&
U_n^\dagger
\left(
    D(\sqrt n\,w)^\dagger D(\sqrt n\,w)
    \otimes\Pi_s^{(n-1)}
\right)
U_n
\\
= &
U_n^\dagger
\left(
    I\otimes\Pi_s^{(n-1)}
\right)
U_n
\\
= &
M_n^\star(\{s\}).
\end{align*}

Since the outcome space is countable, for every
$B\subseteq\N_0$ we have
\[
    M_n^\star(B)
    =
    \sum_{s\in B}M_n^\star(\{s\}),
\]
where the sum converges in the weak operator topology. Hence
\[
    \bigl(D(w)^{\otimes n}\bigr)^\dagger
    M_n^\star(B)
    D(w)^{\otimes n}
    =
    M_n^\star(B)
\]
for every $w\in\bbC$ and every $B\subseteq\N_0$. Therefore,
$M_n^\star\in\cM_{\mathrm{inv}}^{(n)}.$ Next we note that
\begin{align*}
    P_{z,N}^{M_n^\star}(\{s\})&= \Tr[\rho_{z,N}^{\otimes n}U_n^\dagger
    \left(
        I\otimes\Pi_s^{(n-1)}
    \right)
    U_n]\\
    &=\Tr[ \left(\rho_{\sqrt n z,N}
    \otimes
    \phi_N^{\otimes(n-1)}\right)
    \left(
        I\otimes\Pi_s^{(n-1)}
    \right)
    ]\\
    &= \Tr[
    \phi_N^{\otimes(n-1)}
    \Pi_s^{(n-1)}
    ]\\
    &=p_N^{(m)}(s),
\end{align*}
independently of $z$. Therefore
\[
    P_{z,N}^{M}
    =
    K_MP_{z,N}^{M_n^\star}.
\]
\end{proof}

\begin{remark}[The classical experiment]
The sufficient classical experiment generated by
$M_n^\star$ is therefore
\[
    \cE_{n}^{\mathrm{th}}
    =
    \left\{
        p_N^{(n-1)}
        :
        N>0
    \right\},
\]
where
$p_N^{(n-1)}$ is the law of $S
    \sim
    \operatorname{NB}
    \left(
        n-1,\frac{1}{N+1}
    \right).
$
\end{remark}

\subsection{Decision-theoretic reduction}
\label{subsec:thermal-decision}

The preceding theorem concerns the class of displacement-invariant
measurements. We now show that, under the regularity conditions stated
below, restricting attention to this class does not increase either the
nuisance-robust minimax risk or the Bayes-minimax risk.

Let \(\mathsf A\) be a compact metric action space equipped with its Borel
\(\sigma\)-field \(\mathcal A\), and let
\[
    L:K\times\mathsf A\rightarrow[0,\infty)
\]
be a bounded continuous loss, where
$ K\subset(0,\infty)
$
is compact. The loss depends on the state only through the number parameter
\(N\); the displacement \(z\) is a nuisance parameter and does not enter the
loss.

For an action-valued POVM \(D\) on \(\mathsf A\), define
\begin{equation}
    R_{z,N}(D)
    :=
    \int_{\mathsf A}
        L(N,a)\,
        \Tr\!\left[
            \rho_{z,N}^{\otimes n}D(da)
        \right].
    \label{eq:risk}
\end{equation}

The next theorem reduces the optimization over all POVMs to the class of
displacement-invariant POVMs. In contrast with the unitary symmetry
considered in Section~\ref{sec:wss}, the displacement group
$G=(\bbC,+)$ is noncompact, so normalized Haar averaging is unavailable
and a noncompact Hunt-Stein argument is required.

The general noncompact Hunt-Stein reduction for minimax quantum
decision problems goes back to Bogomolov
\cite{bogomolov1982minimax}.  Kumagai and Hayashi
\cite{KumagaiHayashi2013} give a particularly transparent implementation
of the corresponding asymptotically invariant averaging argument for
binary hypothesis testing.  In the present setting, however, we also
require a Bayes-minimax formulation in which a proper prior is placed on
the thermal parameter while the displacement is retained as a nuisance
parameter and treated in the minimax sense.  We therefore give a
self-contained argument for general action-valued POVMs.

The proof combines two standard ingredients. First, we use sequential
compactness of the class of POVMs for the topology of pointwise ultraweak
convergence of their associated positive unital maps. Second, we use
asymptotically invariant averaging over the displacement group. The
resulting argument yields both the usual minimax reduction, which in the
present model is a specialization of Bogomolov's general theorem, and the
Bayes-minimax reduction needed for the thermal analysis.

\begin{theorem}[Hunt-Stein reduction for displacement nuisance]
\label{thm:hunt-stein-thermal}
Let \(K\subset(0,\infty)\) be compact, let \(\mathsf A\) be a compact metric
space, and suppose that
\[
    L:K\times\mathsf A\rightarrow[0,\infty)
\]
is bounded and continuous. Then
\begin{align}
    \inf_D
    \sup_{\substack{z\in\bbC\\N\in K}}
    R_{z,N}(D)
    &=
    \inf_{D\in\cM_{\mathrm{inv}}^{(n)}}
    \sup_{\substack{z\in\bbC\\N\in K}}
    R_{z,N}(D)
    \nonumber\\
    &=
    \inf_{D\in\cM_{\mathrm{inv}}^{(n)}}
    \sup_{N\in K}
    R_{0,N}(D).
    \label{eq:minimax-symmetry}
\end{align}

Further, let \(\Pi\) be a probability measure supported on \(K\), and
define the Bayes-minimax criterion
\begin{equation}
    r_n^{\mathrm{BM}}(\Pi)
    :=
    \inf_D
    \sup_{z\in\bbC}
    \int_K
        R_{z,N}(D)\,
        \Pi(dN).
    \label{eq:bayes-minimax}
\end{equation}
Then
\begin{equation}
    r_n^{\mathrm{BM}}(\Pi)
    =
    \inf_{D\in\cM_{\mathrm{inv}}^{(n)}}
    \int_K
        R_{0,N}(D)\,
        \Pi(dN).
    \label{eq:bayes-minimax-invariant}
\end{equation}
\end{theorem}

For a randomized decision rule $\delta$, we use the shorthand
\[
\E_N L(N,\delta(S))
:=
\sum_{s=0}^{\infty}
p_N^{(n-1)}(s)
\int_{\mathsf A}L(N,a)\,\delta(s,da).
\]
Combining Theorem~\ref{thm:hunt-stein-thermal} with
Theorem~\ref{thm:thermal-q2c} yields an exact reduction to the
negative-binomial experiment.

\begin{corollary}[Exact minimax reduction]
\label{cor:thermal-minimax}
Under the assumptions of Theorem~\ref{thm:hunt-stein-thermal},
\begin{equation}
    \inf_D
    \sup_{\substack{z\in\bbC\\N\in K}}
    R_{z,N}(D)
    =
    \inf_{\delta}
    \sup_{N\in K}
    \E_N
    L\!\left(N,\delta(S)\right),
    \label{eq:exact-minimax-reduction}
\end{equation}
where the infimum on the right is over randomized decision rules based on
$S\sim p_N^{(n-1)}$
with \(p_N^{(n-1)}\) given by
\eqref{eq:negative-binomial}.

Likewise, for every probability measure \(\Pi\) supported on \(K\),
\begin{equation}
    r_n^{\mathrm{BM}}(\Pi)
    =
    \inf_\delta
    \int_K
        \E_N
        L\!\left(N,\delta(S)\right)
        \Pi(dN).
    \label{eq:exact-bayes-reduction}
\end{equation}
Thus the nuisance-robust quantum decision problem is exactly an ordinary
classical decision problem based on the total residual photon count.
\end{corollary}

\begin{proof}
By Theorem~\ref{thm:hunt-stein-thermal}, the nuisance-robust minimax
problem may be restricted to displacement-invariant POVMs:
\begin{equation}
    \inf_D
    \sup_{\substack{z\in\bbC\\N\in K}}
    R_{z,N}(D)
    =
    \inf_{D\in\cM_{\mathrm{inv}}^{(n)}}
    \sup_{N\in K}
    R_{0,N}(D).
    \label{eq:corollary-minimax-invariant-stage}
\end{equation}

By Theorem~\ref{thm:thermal-q2c}, \(M_n^\star\) is
quantum-to-classically sufficient for the displaced thermal experiment
relative to \(\cM_{\mathrm{inv}}^{(n)}\). More precisely, for every
invariant action-valued POVM \(D\), there is a parameter-independent
Markov kernel \(K_D\) such that
\begin{equation}
    P_{z,N}^{D}
    =
    K_D P_{z,N}^{M_n^\star}
    \label{eq:thermal-risk-randomization}
\end{equation}
for every \(z\in\bbC\) and \(N>0\). Consequently, by Proposition \ref{prop:risk-transfer}, every decision
procedure based on an invariant POVM has the same risk function as a
randomized classical decision rule based on the outcome of \(M_n^\star\).
The distribution of this outcome is independent of \(z\) and is given by
$S\sim p_N^{(n-1)}.$

Conversely, every randomized classical decision rule based on \(S\) can
be implemented by classical post-processing of \(M_n^\star\), and the
resulting POVM is displacement invariant. Therefore,
\begin{equation}
    \inf_{D\in\cM_{\mathrm{inv}}^{(n)}}
    \sup_{N\in K}
    R_{0,N}(D)
    =
    \inf_{\delta}
    \sup_{N\in K}
    \E_N
    L\!\left(N,\delta(S)\right),
    \label{eq:invariant-to-negative-binomial-minimax}
\end{equation}
where the infimum on the right is over randomized decision rules based on
\(S\). Combining
\eqref{eq:corollary-minimax-invariant-stage} and
\eqref{eq:invariant-to-negative-binomial-minimax} proves
\eqref{eq:exact-minimax-reduction}.

The Bayes-minimax identity follows in the same way. Indeed,
Theorem~\ref{thm:hunt-stein-thermal} gives
\begin{equation}
    r_n^{\mathrm{BM}}(\Pi)
    =
    \inf_{D\in\cM_{\mathrm{inv}}^{(n)}}
    \int_K
        R_{0,N}(D)\,
        \Pi(dN).
    \label{eq:corollary-bayes-invariant-stage}
\end{equation}
Quantum-to-classical sufficiency identifies the risk functions generated
by invariant quantum procedures with those generated by randomized
classical decision rules based on \(S\). Hence
\begin{equation}
    r_n^{\mathrm{BM}}(\Pi)
    =
    \inf_{\delta}
    \int_K
        \E_N
        L\!\left(N,\delta(S)\right)
        \Pi(dN),
\end{equation}
which proves \eqref{eq:exact-bayes-reduction}.
\end{proof}

\begin{remark}[Application to squared loss]
\label{rem:thermal-squared-loss-compactification}
The compact-action assumption causes no restriction for the squared-loss
problems considered below. Let \(K\subset(0,\infty)\) be compact, let
\[
    \psi:K\rightarrow\bbR^q
\]
be continuous, and consider
\[
    \ell^2_\psi(N,a)
    :=
    \|\psi(N)-a\|^2.
\]
Set
\[
    C_K
    :=
    \operatorname{conv}
    \{\psi(N):N\in K\}.
\]
Since \(\psi(K)\) is compact in the finite-dimensional space \(\bbR^q\),
its convex hull \(C_K\) is compact and convex.

Let
\[
    \pi_{C_K}:\bbR^q\rightarrow C_K
\]
denote the metric projection onto \(C_K\). Since \(\psi(N)\in C_K\), the
projection property gives
\begin{equation}
    \|\psi(N)-\pi_{C_K}(a)\|
    \leq
    \|\psi(N)-a\|
    \label{eq:projection-risk-improvement}
\end{equation}
for every \(N\in K\) and \(a\in\bbR^q\). Consequently, post-processing any
decision procedure by
\[
    a\longmapsto\pi_{C_K}(a)
\]
cannot increase its risk. The action space may therefore be restricted to
the compact set \(C_K\).

On \(K\times C_K\), the squared loss
\[
    (N,a)\longmapsto\|\psi(N)-a\|^2
\]
is bounded and continuous. Hence Theorem~\ref{thm:hunt-stein-thermal} and
Corollary~\ref{cor:thermal-minimax} apply.
\end{remark}

\begin{remark}[Why a Bayes-minimax formulation is used]
There is no proper translation-invariant probability measure on
$\bbC.$
Consequently, unlike the compact-group spectral problem, one cannot place a
proper displacement-invariant prior on the nuisance parameter \(z\) and
obtain an ordinary Bayes symmetrization theorem.

The proper-prior statement above therefore places a prior on the parameter
of interest \(N\) and treats the displacement parameter in the minimax
sense. Equivalently, one may regard the displacement averaging as a
generalized Bayes construction based formally on the improper Haar measure
on \(\bbC\), but no improper prior is required for
Theorem~\ref{thm:hunt-stein-thermal}.

\end{remark}

\subsection{Smooth thermal functionals and asymptotic risks}
\label{subsec:thermal-asymptotic-risks}

The exact reduction in Corollary~\ref{cor:thermal-minimax} leaves an
ordinary one-dimensional exponential family. We now derive its
first-order Bayes and minimax risks.

Set
$
    m:=n-1.
$
Under \(N>0\), the sufficient observation has distribution
$S\sim p_N^{(m)},$
where
\begin{equation}
    p_N^{(m)}(s)
    =
    \binom{s+m-1}{m-1}
    \frac{1}{(N+1)^m}
    \left(\frac{N}{N+1}\right)^s,
    \qquad
    s\in\mathbb N_0.
    \label{eq:thermal-negative-binomial-asymptotic}
\end{equation}
Equivalently,
$S=X_1+\cdots+X_m,$
where \(X_1,\ldots,X_m\) are independent geometric random variables with
\begin{equation}
    \Pr_N(X_j=k)
    =
    \frac{1}{N+1}
    \left(\frac{N}{N+1}\right)^k,
    \qquad
    k\in\mathbb N_0.
    \label{eq:thermal-geometric-law}
\end{equation}
It can be easily verified that
$\E_NS=mN,
    \Var_N(S)=mN(N+1)$, 
and hence the sample mean  
$\widehat N_m:=\frac{S}{m}$ 
is the natural estimator which is unbiased, and satisfies
\begin{equation}
    \Var_N(\widehat N_m)
    =
    \frac{N(N+1)}{m}.
    \label{eq:natural-thermal-estimator-variance}
\end{equation}
The Fisher information of the negative-binomial experiment is
\begin{equation}
    I_m(N)
    =
    mI(N)
    =
    \frac{m}{N(N+1)}.
    \label{eq:negative-binomial-fisher-information}
\end{equation}

Let
$\psi:[0,\infty)\rightarrow\bbR$
be a thermal functional and consider squared loss
\begin{equation}
    \ell^2_\psi(N,a)
    :=
    (\psi(N)-a)^2.
    \label{eq:thermal-functional-squared-loss}
\end{equation}
Standard computation shows that for \(N>0\), the efficient asymptotic variance for estimating
\(\psi(N)\) is
\begin{equation}
    V_\psi(N)
    :=
    \frac{\{\psi'(N)\}^2}{I(N)}
    =
    \{\psi'(N)\}^2N(N+1).
    \label{eq:thermal-efficient-variance}
\end{equation}

We impose the following smoothness condition.

\begin{assumption}[Smooth thermal functional]
\label{ass:thermal-smooth-functional}
The function
\[
    \psi:[0,\infty)\rightarrow\bbR
\]
is twice continuously differentiable, with the derivative at zero
understood as a right derivative, and
\begin{equation}
    \sup_{x\geq0}|\psi''(x)|<\infty.
    \label{eq:thermal-bounded-second-derivative}
\end{equation}

\end{assumption}

Let $
    \psi(\widehat N_m)
    =
    \psi\!\left(\frac{S}{m}\right)
$ be the plug-in estimator.
We first establish a uniform risk expansion for the plug-in estimator.

\begin{lemma}[Uniform plug-in risk expansion]
\label{lem:thermal-uniform-plug-in-risk}
Let \(K\subset(0,\infty)\) be compact and suppose that
Assumption~\ref{ass:thermal-smooth-functional} holds. Then
\begin{equation}
    \sup_{N\in K}
    \left|
        m\E_N
        \left[
            \{\psi(\widehat N_m)-\psi(N)\}^2
        \right]
        -
        V_\psi(N)
    \right|
    \rightarrow0.
    \label{eq:uniform-thermal-plug-in-risk}
\end{equation}
Consequently, \(\psi(\widehat N_m)\) is uniformly first-order efficient
on \(K\).
\end{lemma}

We next record the posterior-variance consequence of the regular
exponential-family structure.

\begin{lemma}[Posterior variance in the negative-binomial experiment]
\label{lem:nb-posterior-variance}
Let \(J=[a,b]\subset(0,\infty)\), where \(a<b\), and let
\(\Pi\) be a probability measure supported on \(J\), with density
\(\pi\) satisfying
\[
    \pi\in C^1(J),
    \qquad
    0<c_\pi\leq \pi(N)\leq C_\pi<\infty,
    \qquad N\in J.
\]
Suppose that
\[
    S\sim p_N^{(m)},
    \qquad
    p_N^{(m)}(s)
    =
    \binom{s+m-1}{s}
    \frac{N^s}{(N+1)^{s+m}},
    \qquad s\in\mathbb N_0.
\]
Let \(\psi:[0,\infty)\to\mathbb R\) satisfy Assumption~\ref{ass:thermal-smooth-functional}.
Write \(N'\) for the variable distributed according to the posterior
given \(S\), and define
\[
    V_\psi(N)
    :=
    N(N+1)\{\psi'(N)\}^2.
\]
Then, for every \(N_0\in(a,b)\),
\begin{equation}
    \mathbb E_{N_0}
    \left|
        m\operatorname{Var}\bigl(\psi(N')\mid S\bigr)
        -
        V_\psi(N_0)
    \right|
    \rightarrow 0.
    \label{eq:nb-pointwise-posterior-functional-variance}
\end{equation}
Moreover,
\begin{equation}
    m\int_J
    \mathbb E_N
    \left[
        \operatorname{Var}\bigl(\psi(N')\mid S\bigr)
    \right]
    \Pi(dN)
    \rightarrow
    \int_J V_\psi(N)\,\Pi(dN).
    \label{eq:nb-integrated-posterior-functional-variance}
\end{equation}
\end{lemma}

We can now derive the Bayes-minimax risk.
For the squared-loss problems below, it is convenient to absorb the
measurement and the estimator into a single estimate-valued POVM.
Thus, for an action-valued POVM
\[
    D:\mathcal B(\bbR)\rightarrow
    \mathcal L(\cH^{\otimes n}),
\]
define
\begin{equation}
    R_{z,N}^{\psi}(D)
    :=
    \int_{\bbR}
        \{a-\psi(N)\}^2
        \Tr\!\left[
            \rho_{z,N}^{\otimes n}D(da)
        \right].
    \label{eq:thermal-functional-risk}
\end{equation}
Accordingly, in what follows an infimum over \(D\) is understood to be
over all estimate-valued POVMs on \(\cH^{\otimes n}\).

Recall that \(M_n^\star\) denotes the total residual photon-number
measurement, with outcome space \(\mathbb N_0\). Let
\[
    \delta_n:
    \mathbb N_0\times\mathcal B(\bbR)
    \rightarrow[0,1]
\]
be a Markov kernel. Its classical post-processing of \(M_n^\star\)
is the  POVM
\begin{equation}
    D_{n,\delta_n}(B)
    :=
    \sum_{s=0}^{\infty}
        \delta_n(s,B)\,
        M_n^\star(\{s\}),
    \qquad
    B\in\mathcal B(\bbR).
    \label{eq:thermal-postprocessed-povm}
\end{equation}
Since the outcome \(S\) of \(M_n^\star\) has law \(p_N^{(n-1)}\),
independently of \(z\),
\begin{equation}
\begin{aligned}
    R_{z,N}^{\psi}(D_{n,\delta_n})
    &=
    \sum_{s=0}^{\infty}
        p_N^{(n-1)}(s)
        \int_{\bbR}
            \{a-\psi(N)\}^2
            \delta_n(s,da).
\end{aligned}
    \label{eq:thermal-postprocessed-risk}
\end{equation}
If the decision rule is deterministic, say
\(\delta_n(s,da)=\delta_{t_n(s)}(da)\) for a measurable function
\(t_n:\mathbb N_0\to\bbR\), then
\begin{equation}
    R_{z,N}^{\psi}(D_{n,\delta_n})
    =
    \E_N
    \left[
        \{t_n(S)-\psi(N)\}^2
    \right].
    \label{eq:thermal-deterministic-postprocessed-risk}
\end{equation}

\begin{theorem}[Asymptotic Bayes-minimax risk]
\label{thm:thermal-asymptotic-bayes-risk}
Let \(J\), \(\Pi\), and \(\psi\) satisfy the assumptions of
Lemma~\ref{lem:nb-posterior-variance}. Define
\begin{equation}
    r_{n,\psi}^{\mathrm{BM}}(\Pi)
    :=
    \inf_D
    \sup_{z\in\bbC}
    \int_J
        R_{z,N}^{\psi}(D)\,
        \Pi(dN),
    \label{eq:thermal-functional-bayes-minimax-risk}
\end{equation}
Then

\begin{equation}
    r_{n,\psi}^{\mathrm{BM}}(\Pi)
    =
    \frac{1}{n-1}
    \int_J
        \{\psi'(N)\}^2N(N+1)\,
        \Pi(dN)
    +
    o(n^{-1}).
    \label{eq:thermal-asymptotic-bayes-risk}
\end{equation}

More precisely, for every \(n\), the Bayes-minimax optimum is attained
by total residual photon-number measurement followed by the posterior
mean
\begin{equation}
    \widehat\psi_{\Pi,n}(s)
    :=
    \E_\Pi[\psi(N)\mid S=s],
    \qquad s\in\mathbb N_0.
    \label{eq:thermal-posterior-mean-estimator}
\end{equation}
\end{theorem}

\begin{proof}
By Corollary~\ref{cor:thermal-minimax}, the quantum Bayes-minimax
problem is exactly the classical Bayes problem based on
$S\sim p_N^{(n-1)}$.
Under squared loss, the Bayes rule is the posterior mean
\[
    \widehat\psi_{\Pi,n}(S)
    =
    \E_\Pi[\psi(N)\mid S].
\]
Hence the finite-sample Bayes-minimax optimum is attained by total
residual photon-number measurement followed by
\(\widehat\psi_{\Pi,n}\), and its risk is
\begin{equation}
    r_{n,\psi}^{\mathrm{BM}}(\Pi)
    =
    \int_J
        \E_N
        \left[
            \Var_\Pi\{\psi(N')\mid S\}
        \right]
        \Pi(dN),
    \label{eq:thermal-bayes-risk-posterior-variance}
\end{equation}
where \(N'\) denotes a draw from the posterior distribution given \(S\).

Lemma~\ref{lem:nb-posterior-variance} therefore gives
\[
    (n-1)\,
    r_{n,\psi}^{\mathrm{BM}}(\Pi)
    \rightarrow
    \int_J
        V_\psi(N)\,
        \Pi(dN).
\]
Since
\[
    V_\psi(N)
    =
    \{\psi'(N)\}^2N(N+1),
\]
the claimed expansion follows.
\end{proof}

We next consider the minimax risk.

\begin{theorem}[Asymptotic minimax risk]
\label{thm:thermal-global-minimax-risk}
Let \(K\subset(0,\infty)\) be nonempty and compact, and suppose that
\begin{equation}
    K=\overline{\operatorname{int}(K)}.
    \label{eq:thermal-regular-compact-set}
\end{equation}
Suppose that Assumption~\ref{ass:thermal-smooth-functional} holds. Define
\begin{equation}
    R_{n,\psi}^{\mathrm{all}}(K)
    :=
    \inf_D
    \sup_{\substack{z\in\bbC\\N\in K}}
        R_{z,N}^{\psi}(D).
    \label{eq:thermal-global-minimax-risk-definition}
\end{equation}
Then
\begin{equation}
    R_{n,\psi}^{\mathrm{all}}(K)
    =
    \frac{1}{n-1}
    \sup_{N\in K}
    \{\psi'(N)\}^2N(N+1)
    +
    o(n^{-1}).
    \label{eq:thermal-global-minimax-risk}
\end{equation}
Moreover, total residual photon-number measurement followed by the
plug-in estimator
\begin{equation}
    \widehat\psi_n
    :=
    \psi\!\left(\frac{S}{n-1}\right)
    \label{eq:thermal-asymptotically-minimax-estimator}
\end{equation}
is asymptotically minimax on \(K\).
\end{theorem}

\begin{proof}
By Corollary~\ref{cor:thermal-minimax}, the quantum minimax problem is
exactly the classical minimax problem based on
$S\sim p_N^{(n-1)}.$
For the upper bound, Lemma~\ref{lem:thermal-uniform-plug-in-risk},
applied to
$\widehat\psi_n
    =
    \psi\!\left(S/(n-1)\right),
$
gives
\[
    \limsup_{n\to\infty}
    (n-1)R_{n,\psi}^{\mathrm{all}}(K)
    \leq
    \sup_{N\in K}V_\psi(N).
\]
In particular, the total residual photon-number measurement followed
by \(\widehat\psi_n\) attains this asymptotic upper bound.

For the lower bound, fix \(\varepsilon>0\). Since \(V_\psi\) is
continuous, \(K\) is compact, and
$K=\overline{\operatorname{int}(K)}$,
there exists a closed interval
\(J_\varepsilon\subset\operatorname{int}(K)\) such that
\[
    V_\psi(N)
    \geq
    \sup_{M\in K}V_\psi(M)-\varepsilon,
    \qquad N\in J_\varepsilon.
\]
Choose a prior \(\Pi_\varepsilon\) supported on \(J_\varepsilon\)
satisfying the assumptions of
Theorem~\ref{thm:thermal-asymptotic-bayes-risk}.
Since maximal risk dominates Bayes-minimax risk,
\[
    R_{n,\psi}^{\mathrm{all}}(K)
    \geq
    r_{n,\psi}^{\mathrm{BM}}(\Pi_\varepsilon).
\]
Therefore, by Theorem~\ref{thm:thermal-asymptotic-bayes-risk},
\[
\begin{aligned}
    \liminf_{n\to\infty}
    (n-1)R_{n,\psi}^{\mathrm{all}}(K)
    &\geq
    \int_{J_\varepsilon}
        V_\psi(N)\,
        \Pi_\varepsilon(dN)
    \\
    &\geq
    \sup_{N\in K}V_\psi(N)-\varepsilon.
\end{aligned}
\]
Letting \(\varepsilon\downarrow0\) gives the lower bound.
\end{proof}

\section{Discussion}
\label{sec:discussion}

We have used the comparison theory of statistical experiments to identify
exact classical reductions of quantum statistical models. The central
question is whether a physically or statistically motivated class of
measurements has a greatest induced classical experiment and, when it does,
what structure selects a representative measurement. This differs from
proving the optimality of a measurement for one loss function: a greatest
experiment supports, through parameter-independent classical
post-processing, every decision procedure available from the admissible
measurement class.

In the two examples studied here, symmetry first justifies or motivates the
admissible class, while a parameter-free conditional block decomposition
identifies its greatest induced experiment. For invariant spectral
inference, the block label is the Schur-Weyl representation label, and weak
Schur sampling is sufficient relative to the class of unitarily invariant
POVMs. For displaced thermal states, concentration separates the nuisance
displacement from the relative modes, and total residual photon number measurement
on those modes is sufficient relative to the displacement-invariant class.
These results reduce the corresponding quantum decision problems to
classical experiments based, respectively, on a Young diagram and a
negative-binomial statistic based on photon count. Optimization over estimators therefore reduces to a classical Bayes or minimax problem.

The examples also indicate several directions for further work. Both
reductions arise from symmetry, but symmetry is not intrinsic to the
definition of quantum-to-classical sufficiency. It would be useful to find
other structural conditions that produce a greatest induced experiment or
to determine when no such experiment can exist. A related question is how
the answer depends on the admissible measurement class. Besides invariant
POVMs, natural candidates include separable, local, sequential, or
otherwise experimentally constrained measurements. Such classes need not
possess a greatest element, so both existence and nonexistence criteria are
relevant.

A further direction is an approximate and asymptotic version of the
theory. For POVMs $M$ and $N$ on standard Borel outcome spaces, define the
model-relative post-processing deficiency
\[
    \delta_{\cQ}(M,N)
    :=
    \inf_K\sup_{\theta\in\Theta}
    \left\lVert
        K P_\theta^M-P_\theta^N
    \right\rVert_{\mathrm{TV}},
\]
where the infimum is over Markov kernels from the outcome space of $M$ to
that of $N$. For a class $\mathfrak M$, set
\[
    \Delta_{\cQ}(M;\mathfrak M)
    :=
    \sup_{N\in\mathfrak M}\delta_{\cQ}(M,N).
\]
Exact quantum-to-classical sufficiency implies
$\Delta_{\cQ}(M;\mathfrak M)=0$. For a sequence of experiments $\cQ_n$ and
measurement classes $\mathfrak M_n$, it is therefore natural to call
$M_n$ asymptotically sufficient relative to $\mathfrak M_n$ if
\[
    \Delta_{\cQ_n}(M_n;\mathfrak M_n)\rightarrow0.
\]
Under bounded losses, these deficiencies also control uniformly the loss
incurred when procedures based on admissible measurements are transferred
to $M_n$.

Quantum local asymptotic normality suggests one possible route to proving
such results. If a sequence of quantum experiments is asymptotically
equivalent, uniformly on the relevant local parameter sets, to a quantum
Gaussian experiment, and if the limiting Gaussian experiment admits a
measurement that approximately dominates the relevant limiting
measurement class, then the corresponding measurement may be transported
back to the original experiments. Making this argument rigorous requires
control of both directions of the q-LAN approximation and of the induced
measurement classes. Since formulations of q-LAN apply to general
sequences of quantum statistical models and are not confined to i.i.d.
experiments \cite{MR3320317}, they may also provide tools for studying
dependent models when the required approximation channels can be
constructed.
Recent work establishes asymptotic equivalence between a stationary quantum Gaussian time-series experiment and a classical nonlinear regression experiment \cite{nussbaum2025asymptotic}. It would be interesting to determine whether the resulting channel construction implies asymptotic sufficiency in the sense described above.

\section*{Acknowledgements.}

The author would like to thank Rathindra Nath Karmakar and Rohan Sarkar for helpful discussions. GPT was used to assist with reviewing the manuscript, checking mathematical arguments, and copyediting. The author has independently verified the content and takes full responsibility for its correctness.

\bibliographystyle{abbrv}
\bibliography{qsuf}

\appendix

\section{Young diagrams and Schur-Weyl duality}

Let $\cH_d:=\C^d$.  We first define the  representations of the group $U(d)$ and $S(n)$ (the group of unitary matrices and permutations respectively) on $(\mathbb{C}^d)^{\otimes n}$. Let $f_{\mathbf{a}}=f_{{\mathbf{a}_1}}\otimes\ldots\otimes f_{{\mathbf{a}_n}}$, where $f_1,\ldots,f_d$ are basis elements of $\mathbb{C}^d$ and $\mathbf{a}_i\in\{1,\ldots,d\}$ (note that $f_{\mathbf{a}}$'s span $(\mathbb{C}^d)^{\otimes n}$) and consider the following actions:
\begin{align}
    \pi_n(T): f_{\mathbf{a}_1}\otimes\ldots\otimes f_{\mathbf{a}_n}&\rightarrow Tf_{\mathbf{a}_1}\otimes\ldots\otimes Tf_{\mathbf{a}_n}, \quad T\in U(d)\label{T_action}\\
    \tilde{\pi}_n(\sigma): f_{\mathbf{a}_1}\otimes\ldots\otimes f_{\mathbf{a}_n}&\rightarrow f_{\mathbf{a}_{\sigma^{-1}(1)}}\otimes \ldots\otimes f_{\mathbf{a}_{\sigma^{-1}(n)}}, \quad \sigma\in S(n)\label{sigma_action}.
\end{align}
It can be shown that the representation space can be decomposed into direct sum of subspaces each of which is a tensor product of irreducible representations of $U(d)$ and $S(n)$. This decomposition is called Schur-Weyl duality in the literature (for a detailed account see \cite{fulton,goodman}).

Define a partition of $n$ as a vector
$\lambda=(\lambda_1,\ldots,\lambda_d)$
with
\[
    \lambda_1\geq\cdots\geq\lambda_d\geq0,
    \qquad
    \sum_{i=1}^d\lambda_i=n.
\]

A Young diagram is defined by an ordered tuple of integers $\lambda=(\lambda_1,\ldots,\lambda_d)$ with $\lambda_1\geq\ldots\geq\lambda_d\geq0$ and can be represented graphically by a diagram with $d$ lines each containing $\lambda_i$ boxes. For example a typical Young diagram looks like:
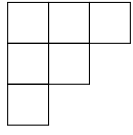
\begin{center}
\begin{ytableau}
 $\text{ }$&& \cr
 & \cr
 \cr
\end{ytableau}
\end{center}
with $\lambda=(3,2,1)$.

Thus, a partition of $n$ can be identified with a Young diagram with
$n$ boxes. Recall that
\[
    \Young
    =
    \{\lambda\vdash n:\ell(\lambda)\leq d\}
\]
is the set of partitions of $n$ with at most $d$ nonzero parts,
equivalently Young diagrams with at most $d$ rows.

Schur-Weyl duality gives a unitary decomposition
\[
    \cH_d^{\ot n}
    \cong
    \bigoplus_{\lambda\in\Young}
    \cU_\lambda\ot\cV_\lambda,
\]
where 
$\cU_\lambda$ carries an irreducible representation of $U(d)$ and $\cV_\lambda$ carries an irreducible representation of $S_n$.

In this decomposition,
\[
\widetilde\pi_n(\sigma)
=
\bigoplus_{\lambda\in\Young}
I_{\cU_\lambda}\otimes\pi_\lambda(\sigma),
\qquad \sigma\in S_n,
\]
while
\[
    U^{\ot n}
    \longmapsto
    \bigoplus_{\lambda\in\Young}
     U_\lambda(U)\otimes I_{\cV_\lambda}. 
\]

The commutant of the collective unitary action is
\[
    \{U^{\ot n}:U\in U(d)\}'
    =
    \bigoplus_{\lambda\in\Young}
    I_{\cU_\lambda} \ot \mathcal L(\cV_\lambda),
\]
whereas the commutant of the permutation action is
\[
    \{\widetilde\pi_n(\sigma):\sigma\in S_n\}'
    =
    \bigoplus_{\lambda\in\Young}
    \mathcal L(\cU_\lambda) \ot I_{\cV_\lambda}.
\]

Because $\rho^{\otimes n}$ commutes with every permutation,
\[
\rho^{\otimes n}
=\bigoplus_{\lambda\in\mathbb Y_{n,d}}
R_\lambda(\rho)\otimes I_{\mathcal V_\lambda},
\]
where $R_\lambda(\rho)$ is a positive operator on
$\mathcal U_\lambda$.

\section{Weak Schur sampling and the Schur-Weyl distribution}

\subsection{Weak Schur sampling}
Let $\Pi_\lambda$ denote the orthogonal projection onto the
isotypic component $\mathcal U_\lambda\otimes\mathcal V_\lambda$.

\begin{definition}[Weak Schur sampling]
Weak Schur sampling is the projective measurement
\[
    \WSS_n:=\{\Pi_\lambda:\lambda\in\Young\}.
\]
Its outcome is the Young diagram $\lambda$.
\end{definition}

The induced classical experiment is
\[
    \cF_n
    :=
    \left\{
        P_\rho^{\WSS_n}:
        \rho\in\mathcal S(\mathbb C^d)
    \right\},
\]
where
\[
    P_\rho^{\WSS_n}(\lambda)
    =\Tr\bigl(\rho^{\ot n}\Pi_\lambda\bigr).
\]

\begin{prop}[Commutativity of the Schur-Weyl projectors]
Let
\[
(\mathbb C^d)^{\otimes n}
\cong
\bigoplus_{\mu\in\mathbb Y_{n,d}}
\mathcal U_\mu\otimes\mathcal V_\mu
\]
be the Schur-Weyl decomposition, and let \(\Pi_\lambda\) denote the
orthogonal projection onto the \(\lambda\)-isotypic component
\[
\mathcal H_\lambda
:=
\mathcal U_\lambda\otimes\mathcal V_\lambda.
\]
Then, for every \(U\in U(d)\),
\[
[U^{\otimes n},\Pi_\lambda]=0.
\]
\end{prop}

\begin{proof}
Under the Schur-Weyl decomposition, the collective unitary
representation has the block-diagonal form
\[
U^{\otimes n}
=
\bigoplus_{\mu\in\mathbb Y_{n,d}}
U_\mu(U)\otimes I_{\mathcal V_\mu},
\]
where \(U_\mu(U)\) denotes the representation of \(U(d)\) acting on
\(\mathcal U_\mu\).

The projector \(\Pi_\lambda\) acts as the identity on
\(\mathcal U_\lambda\otimes\mathcal V_\lambda\) and as zero on every
other isotypic component. Thus, for a vector
\[
v=\bigoplus_{\mu\in\mathbb Y_{n,d}}v_\mu,
\qquad
v_\mu\in\mathcal U_\mu\otimes\mathcal V_\mu,
\]
we have
\[
\Pi_\lambda v=v_\lambda.
\]

Applying \(U^{\otimes n}\) after \(\Pi_\lambda\) gives
\[
U^{\otimes n}\Pi_\lambda v
=
\bigl(U_\lambda(U)\otimes I_{\mathcal V_\lambda}\bigr)v_\lambda.
\]
On the other hand,
\[
U^{\otimes n}v
=
\bigoplus_{\mu\in\mathbb Y_{n,d}}
\bigl(U_\mu(U)\otimes I_{\mathcal V_\mu}\bigr)v_\mu.
\]
Projecting this vector onto the \(\lambda\)-isotypic component yields
\[
\Pi_\lambda U^{\otimes n}v
=
\bigl(U_\lambda(U)\otimes I_{\mathcal V_\lambda}\bigr)v_\lambda.
\]
Therefore,
\[
U^{\otimes n}\Pi_\lambda v
=
\Pi_\lambda U^{\otimes n}v
\]
for every \(v\in(\mathbb C^d)^{\otimes n}\). And hence
$
[U^{\otimes n},\Pi_\lambda]=0.
$
\end{proof}

\begin{corollary}[Unitary invariance of the weak Schur distribution]
Let \(\rho,\sigma\in\mathcal S(\mathbb C^d)\) be unitarily
conjugate:
\[
\sigma=U\rho U^\dagger
\]
for some \(U\in U(d)\). Then
\[
\operatorname{Tr}
\left(
\sigma^{\otimes n}\Pi_\lambda
\right)
=
\operatorname{Tr}
\left(
\rho^{\otimes n}\Pi_\lambda
\right)
\]
for every \(\lambda\in\mathbb Y_{n,d}\).
\end{corollary}

\begin{proof}
Since
\[
\sigma^{\otimes n}
=
U^{\otimes n}\rho^{\otimes n}(U^\dagger)^{\otimes n},
\]
cyclicity of the trace gives
\begin{align*}
\operatorname{Tr}
\left(
\sigma^{\otimes n}\Pi_\lambda
\right)
&=
\operatorname{Tr}
\left[
U^{\otimes n}\rho^{\otimes n}
(U^\dagger)^{\otimes n}\Pi_\lambda
\right]\\
&=
\operatorname{Tr}
\left[
\rho^{\otimes n}
(U^\dagger)^{\otimes n}
\Pi_\lambda U^{\otimes n}
\right].
\end{align*}
By the preceding proposition,
$(U^\dagger)^{\otimes n}\Pi_\lambda U^{\otimes n}
=\Pi_\lambda$.
It follows that
\[
\operatorname{Tr}
\left(
\sigma^{\otimes n}\Pi_\lambda
\right)
=
\operatorname{Tr}
\left(
\rho^{\otimes n}\Pi_\lambda
\right).
\]
Thus the weak Schur-sampling distribution is constant on unitary
orbits and therefore depends on \(\rho\) only through its spectrum.
\end{proof}
\subsection{Schur-Weyl distribution}
Since conjugate states have the same Schur distribution, this family depends on $\rho$
only through $\theta=\spec(\rho)$.  We may therefore write
\[
    P_\theta^{(n)}(\lambda)
    :=P_\rho^{\WSS_n}(\lambda)
    \quad\text{whenever }\spec(\rho)=\theta.
\]

We describe the Schur-Weyl distribution using standard and semistandard
Young tableaux.

For a cell \(u=(i,j)\) of \(\lambda\), let \(\lambda'_j\) be the
number of rows containing at least \(j\) cells. Its hook length is
\[
h_\lambda(i,j)=\lambda_i-j+\lambda'_j-i+1.
\]
The number of standard Young tableaux of shape \(\lambda\) is
\[
f^\lambda
=
\frac{n!}{\displaystyle\prod_{u\in\lambda}h_\lambda(u)}.
\]

Let \(\operatorname{SSYT}_d(\lambda)\) denote the semistandard Young
tableaux of shape \(\lambda\) with entries in \(\{1,\ldots,d\}\).
Entries increase weakly along rows and strictly down columns. If
\(m_i(T)\) counts the entries equal to \(i\), define the Schur polynomial by
\[
s_\lambda(\theta)
=
\sum_{T\in\operatorname{SSYT}_d(\lambda)}
\prod_{i=1}^{d}\theta_i^{m_i(T)}.
\]

The Schur-Weyl distribution is
\[ P_\theta^{(n)}(\lambda)
=
f^\lambda s_\lambda(\theta),
\qquad \lambda\in\Young.
\]

\section{Proof of the main theorems}

\begin{proof}[Proof of Theorem \ref{thm:converse-invariant-factorization}]
By Proposition~\ref{prop:observable-side-converse}, for every
invariant effect $E\in\cA_G$, $0\leq E\leq I$, there exists
a function $f_E:\mathsf Z\to[0,1]$ satisfying
\eqref{eq:observable-factorization}, namely
\[
    \Tr(\bar\rho_\eta E)
    =
    \sum_{z\in\mathsf Z}p_\eta(z)f_E(z),
    \qquad \eta\in\Theta/G.
\]
Under Assumption~\ref{ass:block-coordinate-identifiability},
the coefficient vector
$\bigl(f_E(z)\bigr)_{z\in\mathsf Z}$ is unique.
For each $z\in\mathsf Z$, define
\[
    \varphi_z(E):=f_E(z).
\]
The following lemma, proved in the appendix, extends these
coefficients from invariant effects to the whole invariant
algebra.

\begin{lemma}
\label{phi_z_ext}
For each $z\in\mathsf Z$, the map $\varphi_z$ is affine
on the set of invariant effects and extends uniquely to
a positive linear functional on $\cA_G$.
\end{lemma}

We henceforth denote this positive linear extension again
by $\varphi_z$. Since $\Pi_w$ is a central projection in
$\cA_G$, it is an invariant effect. Taking $E=\Pi_w$ in
\eqref{eq:observable-factorization} and using
$p_\eta(w)=\Tr(\bar\rho_\eta\Pi_w)$ gives
\[
    p_\eta(w)
    =
    \sum_{z\in\mathsf Z}
    p_\eta(z)\varphi_z(\Pi_w).
\]
On the other hand,
\[
    p_\eta(w)
    =
    \sum_{z\in\mathsf Z}
    p_\eta(z)\delta_{zw}.
\]
By Assumption~\ref{ass:block-coordinate-identifiability},
\begin{equation}
    \varphi_z(\Pi_w)
    =
    \delta_{zw},
    \qquad z,w\in\mathsf Z.
    \label{eq:phi-central-projections}
\end{equation}
In particular, since $\sum_w\Pi_w=I$, linearity gives
$\varphi_z(I)=1$.

We now show that $\varphi_z$ is supported on the $z$th
block. Since
\[
    I-\Pi_z
    =
    \sum_{w\neq z}\Pi_w,
\]
equation~\eqref{eq:phi-central-projections} and linearity
imply
\[
    \varphi_z(I-\Pi_z)=0.
\]
Because $\varphi_z$ is positive, it satisfies the
Cauchy-Schwarz inequality
\[
    |\varphi_z(B^\dagger C)|^2
    \leq
    \varphi_z(B^\dagger B)\,
    \varphi_z(C^\dagger C),
    \qquad B,C\in\cA_G.
\]
Taking $B=I-\Pi_z$ and $C=A$, and using
$(I-\Pi_z)^2=I-\Pi_z$, gives
\[
    \left|\varphi_z\bigl((I-\Pi_z)A\bigr)\right|^2
    \leq
    \varphi_z(I-\Pi_z)\varphi_z(A^\dagger A)
    =
    0.
\]
Thus
\[
    \varphi_z\bigl((I-\Pi_z)A\bigr)=0,
    \qquad A\in\cA_G.
\]
Similarly, taking $B=A^\dagger$ and $C=I-\Pi_z$ gives
\[
    \left|\varphi_z\bigl(A(I-\Pi_z)\bigr)\right|^2
    \leq
    \varphi_z(AA^\dagger)\varphi_z(I-\Pi_z)
    =
    0.
\]
Consequently,
\[
    \varphi_z\bigl(A(I-\Pi_z)\bigr)=0,
    \qquad A\in\cA_G.
\]
Since
\[
    A-\Pi_zA\Pi_z
    =
    (I-\Pi_z)A+\Pi_zA(I-\Pi_z)
\]
and $\Pi_zA\in\cA_G$, these identities imply
\begin{equation}
    \varphi_z(A)
    =
    \varphi_z(\Pi_zA\Pi_z),
    \qquad A\in\cA_G.
    \label{eq:phi-block-support}
\end{equation}
Thus $\varphi_z$ is completely determined by its
restriction to
\[
    \cA_z:=\Pi_z\cA_G\Pi_z.
\]

Since $\cH$ is finite dimensional, $\cA_z$ is a
finite-dimensional $C^*$-algebra. The trace pairing on
$\cA_z$ is nondegenerate.
Hence there exists
a unique $\tau_z\in\cA_z$ such that
\[
    \varphi_z(A)
    =
    \Tr(\tau_zA),
    \qquad A\in\cA_z.
\]
Here $\Tr$ denotes the ordinary operator trace on
$\cH_z$, equivalently the trace on $\cH$ after
extending the operators by zero.

We verify that $\tau_z\geq0$. A positive linear
functional is real-valued on self-adjoint elements,
since each such element is a difference of positive
elements. It therefore satisfies
\[
    \varphi_z(A^\dagger)
    =
    \overline{\varphi_z(A)},
    \qquad A\in\cA_z.
\]
Consequently,
\[
    \Tr(\tau_z^\dagger A)
    =
    \overline{\Tr(\tau_zA^\dagger)}
    =
    \overline{\varphi_z(A^\dagger)}
    =
    \varphi_z(A),
    \qquad A\in\cA_z.
\]
Uniqueness of the trace representation gives
$\tau_z^\dagger=\tau_z$. If $\tau_z$ had a negative
eigenvalue, its spectral projection $P_-$ onto the
negative eigenspaces would belong to $\cA_z$ by functional calculus and would satisfy
\[
    \varphi_z(P_-)
    =
    \Tr(\tau_zP_-)
    <
    0,
\]
contradicting positivity of $\varphi_z$. Thus
$\tau_z\geq0$.

Moreover, $\Pi_z$ is the identity element of $\cA_z$,
and $\tau_z=\Pi_z\tau_z\Pi_z$. By
\eqref{eq:phi-central-projections},
\[
    \Tr(\tau_z)
    =
    \Tr(\tau_z\Pi_z)
    =
    \varphi_z(\Pi_z)
    =
    1.
\]
Therefore $\tau_z$ is a density operator supported on
$\cH_z$, independent of $\eta$.

Before applying the observable-side factorization to
an arbitrary $A\in\cA_G$, we note that invariant
effects linearly span $\cA_G$. Indeed, for a
self-adjoint $A\in\cA_G$, choose $r>0$ with
$r\geq\|A\|$. Then
\[
    E:=\frac{A+rI}{2r}
\]
is an invariant effect and
\[
    A=r(2E-I).
\]
Every element of $\cA_G$ is a complex linear
combination of self-adjoint elements. Hence
\eqref{eq:observable-factorization}, together with
linearity of the extended functionals, implies
\[
    \Tr(\bar\rho_\eta A)
    =
    \sum_{z\in\mathsf Z}
    p_\eta(z)\varphi_z(A),
    \qquad A\in\cA_G.
\]

Now let $A\in\cA_G$. Using this identity,
\eqref{eq:phi-block-support}, and the trace
representation above, we obtain
\begin{align*}
    \Tr(\bar\rho_\eta A)
    &=
    \sum_{z\in\mathsf Z}
    p_\eta(z)\varphi_z(A)\\
    &=
    \sum_{z\in\mathsf Z}
    p_\eta(z)
    \Tr\!\left(\tau_z\Pi_zA\Pi_z\right)\\
    &=
    \Tr\!\left[
        \left(
            \sum_{z\in\mathsf Z}p_\eta(z)\tau_z
        \right)A
    \right].
\end{align*}
Hence defining $D_\eta
    :=
    \bar\rho_\eta
    -
    \sum_{z\in\mathsf Z}p_\eta(z)\tau_z$, we have
    
\[
    \Tr(D_\eta A)=0,
    \qquad A\in\cA_G.
\]

Since $\bar\rho_\eta$ is $G$-invariant,
$\bar\rho_\eta\in\cA_G$. Also,
$\tau_z\in\cA_z\subseteq\cA_G$ for every $z$.
Thus both terms defining $D_\eta$ belong to $\cA_G$,
and consequently $D_\eta\in\cA_G$.
Since $D_\eta$ is self-adjoint, we may take
$A=D_\eta$ to obtain
\[
    0
    =
    \Tr(D_\eta^2)
    =
    \|D_\eta\|_{\mathrm{HS}}^2.
\]
Therefore $D_\eta=0$. Finally, since the $\tau_z$
are supported on the mutually orthogonal subspaces
$\cH_z$, their sum has the direct-sum form
\[
    \bar\rho_\eta
    =
    \bigoplus_{z\in\mathsf Z}p_\eta(z)\tau_z.
\]
This proves \eqref{eq:converse-factorization}.
\end{proof}

\begin{proof}[Proof of Theorem \ref{thm:symmetrization}]
Represent the original decision procedure by an action-valued POVM
\(D\) on \((\mathsf A,\mathcal A)\). Define its symmetrization by
\[
\overline D(C)
:=
\int_G U_g D(g^{-1}C)U_g^\dagger\,\mu_G(dg),
\qquad C\in\mathcal A,
\]
where \(\mu_G\) denotes the normalized Haar measure on \(G\). The
integral is understood in the weak operator sense. Positivity,
normalization, and countable additivity follow from the corresponding
properties of \(D\); hence \(\overline D\) is an action-valued POVM.

We first verify that \(\overline D\) is equivariant. For \(h\in G\),
the left invariance of Haar measure gives
\[
\begin{aligned}
U_h\overline D(C)U_h^\dagger
&=
\int_G U_{hg}D(g^{-1}C)U_{hg}^\dagger\,\mu_G(dg)\\
&=
\int_G U_kD(k^{-1}hC)U_k^\dagger\,\mu_G(dk)\\
&=
\overline D(hC),
\end{aligned}
\]
where we used the change of variables \(k=hg\). Thus
\(\overline D\) is an equivariant procedure.

We next compute its risk. Since the loss is nonnegative, Tonelli's
theorem permits interchange of the relevant integrals. Using the
equivariance relation
\[
U_g^\dagger\rho_\theta U_g=\rho_{g^{-1}\theta},
\]
we obtain
\[
\begin{aligned}
R(\theta;\overline D)
&=
\int_{\mathsf A}
L(\theta,a)\,
\operatorname{Tr}\!\left[
    \rho_\theta\overline D(da)
\right]
\\
&=
\int_G\int_{\mathsf A}
L(\theta,a)\,
\operatorname{Tr}\!\left[
    U_g^\dagger\rho_\theta U_g\,D(g^{-1}da)
\right]
\,\mu_G(dg)
\\
&=
\int_G\int_{\mathsf A}
L(\theta,gb)\,
\operatorname{Tr}\!\left[
    \rho_{g^{-1}\theta}D(db)
\right]
\,\mu_G(dg),
\end{aligned}
\]
where \(b=g^{-1}a\). By invariance of the loss,
\[
L(\theta,gb)
=
L(g^{-1}\theta,b).
\]
Consequently,
\[
R(\theta;\overline D)
=
\int_G R(g^{-1}\theta;D)\,\mu_G(dg).
\]

Integrating with respect to the prior \(\Pi\), we find
\[
\begin{aligned}
R_\Pi(\overline D)
&=
\int_\Theta R(\theta;\overline D)\,\Pi(d\theta)\\
&=
\int_G\int_\Theta
R(g^{-1}\theta;D)\,\Pi(d\theta)\,\mu_G(dg).
\end{aligned}
\]
Because \(\Pi\) is \(G\)-invariant, the pushforward of \(\Pi\) under
\(\theta\mapsto g^{-1}\theta\) is again \(\Pi\). Therefore, for every
\(g\in G\),
\[
\int_\Theta R(g^{-1}\theta;D)\,\Pi(d\theta)
=
\int_\Theta R(\vartheta;D)\,\Pi(d\vartheta)
=
R_\Pi(D).
\]
It follows that
\[
R_\Pi(\overline D)
=
\int_G R_\Pi(D)\,\mu_G(dg)
=
R_\Pi(D).
\]
Hence every decision procedure admits an equivariant symmetrization
with the same Bayes risk.
\end{proof}

\begin{proof}[Proof of Theorem~\ref{thm:WSS_Bayes_asymp}]

\textbf{Lower bound}\\

Let
\[
    N_n=(N_{n,1},\ldots,N_{n,d})\sim\Mult_d(n,\btheta)
\]
be the multinomial count vector based on \(n\) independent observations
with cell probabilities \(\btheta\). Its Bayes risk is
\[
    r_{n,\psi}^{\Mult}(\Pi)
    :=\inf_\delta
      \int_{\Theta_{\downarrow}}
      \E_{\btheta}^{\Mult}
      \|\delta(N_n)-\psi(\btheta)\|^2
      \,\Pi(d\btheta).
\]

Next we show that for every \(n\), there exists a parameter-independent Markov kernel
\(K_n\) from multinomial count vectors to Young diagrams such that
\begin{equation}
    P_{\btheta}^{(n)}
    =K_n\Mult_d(n,\btheta),
    \qquad \btheta\in\Theta_{\downarrow}.
\label{eq:mult-wss-randomization}
\end{equation}
Let \(X_1,\ldots,X_n\) be independent random variables taking values
in \(\{1,\ldots,d\}\), with
\(\Pp_{\btheta}(X_j=i)=\theta_i\). Their count vector is \(N_n\).
Conditional on \(N_n=k\), all words with content \(k\) have the same
probability; consequently, the conditional law of the word given its
content is independent of \(\btheta\).

Apply the Robinson-Schensted-Knuth correspondence (cf. \cite{odonnell2016efficient}) to the word and
retain its shape. Define
\[
    K_n(k,\lambda)
    :=\Pp\!\left(
      \operatorname{shape}(\operatorname{RSK}(X_1,\ldots,X_n))
      =\lambda\mid N_n=k
    \right).
\]

The kernel is parameter-independent. The RSK shape distribution of an
i.i.d.\ word with letter probabilities \(\btheta\) satisfies
\[
    \Pp_{\btheta}\!\left(
      \operatorname{shape}(\operatorname{RSK}(X_1,\ldots,X_n))
      =\lambda
    \right)
    =\dim(\cV_\lambda)s_\lambda(\btheta);
\]
see \cite[Section~2, Eq.~(8)]{odonnell2016efficient}.
By \eqref{eq:wss-law-bayes}, this is exactly the weak Schur sampling
law, which proves \eqref{eq:mult-wss-randomization}. Composing any WSS decision rule
with \(K_n\) produces a multinomial decision rule having exactly the
same action distribution and risk.
Consequently,
\begin{equation}
    r_{n,\psi}^{\Mult}(\Pi)
    \leq r_{n,\psi}^{\WSS}(\Pi),
\label{eq:mult-wss-bayes-order}
\end{equation}
where \(r_{n,\psi}^{\Mult}(\Pi)\) is the Bayes risk in the
multinomial experiment.

For the local multinomial calculation, use the intrinsic parameter
\[
    \bvartheta=(\theta_1,\ldots,\theta_{d-1})^\top,
    \qquad
    \theta_d=1-\sum_{i=1}^{d-1}\vartheta_i,
\]
and define
\[
    A:=
    \begin{pmatrix}
      I_{d-1}\\
      -\bone_{d-1}^\top
    \end{pmatrix}.
\]
The inverse multinomial Fisher information in these coordinates is
\[
    V_{\btheta}
    :=\diag(\theta_1,\ldots,\theta_{d-1})
      -\btheta_{[-d]}\btheta_{[-d]}^\top,
    \qquad
    \btheta_{[-d]}=(\theta_1,\ldots,\theta_{d-1})^\top.
\]
A direct calculation gives
\begin{equation}
    AV_{\btheta}A^\top=\Sigma(\btheta).
\label{eq:intrinsic-ambient-covariance}
\end{equation}

The next lemma uses these intrinsic coordinates to establish the
following pointwise posterior-variance limit.

\begin{lemma}[Pointwise multinomial posterior variance]
\label{lem:mult-pointwise}
Fix \(\btheta_0\in\Theta_{\mathrm{reg}}\) such that
\(\pi(\btheta_0)>0\) and \(\pi\) is continuous at \(\btheta_0\).
Then
\[
    n\,\E_{\btheta_0}^{\Mult}
    \left[\Tr\Var\{\psi(\btheta)\mid N_n\}\right]
    \rightarrow\cI_\psi(\btheta_0).
\]
\end{lemma}

The lemma is the moment form of the Bernstein-von Mises theorem for
the intrinsic parameter \(\bvartheta\), followed by the delta method
and \eqref{eq:intrinsic-ambient-covariance}. The proof  follows from Lemma \ref{lem:multinomial-posterior-covariance} and \ref{lem:multinomial-functional-posterior-variance} proved in Appendix \ref{sec:tech_lem}.

Next, note that under squared loss,
\[
    r_{n,\psi}^{\Mult}(\Pi)
    =\int_{\Theta_{\downarrow}}
      \E_{\btheta}^{\Mult}
      \left[\Tr\Var\{\psi(\btheta')\mid N_n\}\right]
      \Pi(d\btheta),
\]
where \(\btheta'\) denotes the posterior variable and \(\btheta\) the
true parameter under the outer expectation. The singular set
\[
    \Theta_{\downarrow}\setminus\Theta_{\mathrm{reg}}
    =\{\theta_d=0\}
      \cup\bigcup_{i=1}^{d-1}\{\theta_i=\theta_{i+1}\}
\]
has \(\Pi\)-measure zero. Lemma~\ref{lem:mult-pointwise} therefore
applies for \(\Pi\)-almost every \(\btheta\). Since the integrands are
nonnegative, Fatou's lemma yields

\begin{align}
\liminf_{n\to\infty}n\,r_{n,\psi}^{\Mult}(\Pi)
    &\geq
    \int_{\Theta_{\downarrow}}
      \liminf_{n\to\infty}
      n\,\E_{\btheta}^{\Mult}
      \left[\Tr\Var\{\psi(\btheta')\mid N_n\}\right]
      \Pi(d\btheta)\nonumber\\
    &=\int_{\Theta_{\downarrow}}
      \cI_\psi(\btheta)\,\Pi(d\btheta). \label{eq:mult_int_lower}
\end{align}

By \eqref{eq:mult-wss-bayes-order} and
\eqref{eq:mult_int_lower},
\begin{equation}
    \liminf_{n\to\infty}
    n\,r_{n,\psi}^{\WSS}(\Pi)
    \geq
    \int_{\Theta_{\downarrow}}
      \cI_\psi(\btheta)\,\Pi(d\btheta).
\label{eq:wss-bayes-lower-bound}
\end{equation}

\begin{remark}
    We note that a direct lower bound to the Bayes risk is possible under stronger assumptions on the prior. The proof uses van Trees inequality \cite{Gill&Levit}. However, to keep our prior choice general we adopt the current method.
\end{remark}

\textbf{Upper bound}\\

The upper bounds in both the Bayes and minimax problems are based on a
single uniform risk result. We first state the distributional estimate
needed to control its second moments.

Let
\[
    \mathbb L_{n,d}
    :=
    \left\{
        k\in\mathbb N_0^d:
        \sum_{i=1}^d k_i=n
    \right\}.
\]

Regard \(P_{\btheta}^{(n)}\) as a probability measure on
\(\mathbb L_{n,d}\) by assigning mass zero outside
\(\mathbb Y_{n,d}\), and let \(M_{\btheta}^{(n)}\) denote the law on
\(\mathbb L_{n,d}\) of \(N_n\sim\Mult_d(n,\btheta)\). For
\(\alpha\in(1/2,1)\), define
\[
    \mathcal T_{n,\alpha}(\btheta)
    :=\left\{k\in\mathbb L_{n,d}:
      \max_{1\leq i\leq d}|k_i-n\theta_i|\leq n^\alpha
    \right\}.
\]

The next lemma compares the weak Schur and multinomial distributions
uniformly over compact subsets of $\Theta_{\mathrm{reg}}$ and
controls their probabilities outside the typical set above.
\begin{lemma}[Uniform WSS-multinomial comparison]
\label{lem:uniform-wss-mult-comparison}
Let $\mathsf K\subset\Theta_{\mathrm{reg}}$ be compact and fix
$\alpha\in(1/2,1)$. There exist constants
$C_{\mathsf K},c_{\mathsf K}>0$ and
$n_{\mathsf K}\in\mathbb N$ such that, for every
$n\geq n_{\mathsf K}$,
\begin{equation}
\label{eq:uniform-l1-comparison}
\sup_{\btheta\in\mathsf K}
\left\|
    P_{\btheta}^{(n)}-M_{\btheta}^{(n)}
\right\|_1
\leq
C_{\mathsf K}
\left(
    n^{-1/2}+n^{\alpha-1}
\right),
\end{equation}
and
\begin{equation}
\label{eq:uniform-typical-tail}
\sup_{\btheta\in\mathsf K}
\left\{
    P_{\btheta}^{(n)}
    \bigl(\mathcal T_{n,\alpha}(\btheta)^c\bigr)
    +
    M_{\btheta}^{(n)}
    \bigl(\mathcal T_{n,\alpha}(\btheta)^c\bigr)
\right\}
\leq
C_{\mathsf K}n^{d/2}
\exp\!\left\{
    -c_{\mathsf K}n^{2\alpha-1}
\right\}.
\end{equation}
\end{lemma}

The next lemma uses Lemma \ref{lem:uniform-wss-mult-comparison} to show that the risk of the Young diagram estimator can be approximated by the multinomial risk uniformly.

\begin{lemma}[Uniform risk of the empirical Young diagram]
\label{lem:uniform-wss-functional-risk}
Let \(\mathsf K\subset\Theta_{\mathrm{reg}}\) be compact and suppose
Assumption~\ref{ass:psi} holds. Then
\[
    \sup_{\btheta\in\mathsf K}
    \left|
        n\,\E_{\btheta}^{\WSS}
        \left\|
            \psi\!\left(\frac{\Lambda_n}{n}\right)
            -\psi(\btheta)
        \right\|^2
        -\cI_\psi(\btheta)
    \right|
    \rightarrow0.
\]
\end{lemma}

The preceding lemma is uniform only on regular compact sets. The proof is deferred to Appendix~\ref{sec:tech_lem}. To
integrate its pointwise consequence against priors that may approach
the singular strata, we use the following global domination, which holds for every \(n\) and every \(\btheta\in\Theta_{\downarrow}\),
\begin{equation}
    \E_{\btheta}^{\WSS}
      \|\widehat{\btheta}_n-\btheta\|^2
    \leq\frac dn.
\label{eq:ow-spectrum-bound}
\end{equation}
The assertion is Theorem~1.1 of
\cite{odonnell2016efficient}.
Consequently, from the global
Lipschitz bound in \eqref{eq:psi-lipschitz-constant}
\begin{equation}
    n\,\E_{\btheta}^{\WSS}
    \|\psi(\widehat{\btheta}_n)-\psi(\btheta)\|^2
    \leq dL_\psi^2
\label{eq:ow-functional-bound}
\end{equation}
for every \(n\) and \(\btheta\in\Theta_{\downarrow}\).

For the upper bound, consider the following plug-in estimator based on WSS
\[
    \widehat\psi_n:=\psi\!\left(\frac{\Lambda_n}{n}\right).
\]
By optimality of the Bayes rule,
\[
    r_{n,\psi}^{\WSS}(\Pi)
    \leq
    \int_{\Theta_{\downarrow}}
      \E_{\btheta}^{\WSS}
      \|\widehat\psi_n-\psi(\btheta)\|^2
      \,\Pi(d\btheta).
\]
Define
\[
    g_n(\btheta)
    :=n\,\E_{\btheta}^{\WSS}
      \|\widehat\psi_n-\psi(\btheta)\|^2.
\]
For every fixed \(\btheta\in\Theta_{\mathrm{reg}}\),
Lemma~\ref{lem:uniform-wss-functional-risk}, applied to
\(\mathsf K=\{\btheta\}\), gives
\[
    g_n(\btheta)\rightarrow\cI_\psi(\btheta).
\]
Assumption~\ref{ass:prior} implies
\(\Pi(\Theta_{\mathrm{reg}})=1\), while \eqref{eq:ow-functional-bound} gives
\[
    0\leq g_n(\btheta)\leq dL_\psi^2
\]
for every \(n\) and \(\btheta\). Dominated convergence therefore
yields
\[
    \limsup_{n\to\infty}n\,r_{n,\psi}^{\WSS}(\Pi)
    \leq
    \int_{\Theta_{\downarrow}}
      \cI_\psi(\btheta)\,\Pi(d\btheta).
\]
Together with \eqref{eq:wss-bayes-lower-bound}, this proves
\eqref{eq:wss-bayes-asymptotics}.
\end{proof}

\begin{proof}[Proof of Theorem \ref{thm:hunt-stein-thermal}]
Write
$V_w:=D(w)^{\otimes n},
    w\in\bbC.$
Recall that the covariance relation for the \(n\)-copy displaced thermal model is
\begin{equation}
    V_w\rho_{z,N}^{\otimes n}V_w^\dagger
    =
    \rho_{z+w,N}^{\otimes n},
    \qquad
    z,w\in\bbC,\quad N>0.
    \label{eq:n-copy-displacement-covariance}
\end{equation}
For a decision POVM \(D\), define the displacement action by
\begin{equation}
    (w\cdot D)(B)
    :=
    V_w^\dagger D(B)V_w,
    \qquad
    B\in\mathcal A.
    \label{eq:displacement-action-decision-povm}
\end{equation}
Accordingly,
$D\in\cM_{\mathrm{inv}}^{(n)}$
if and only if
\begin{equation}
    V_w^\dagger D(B)V_w=D(B)
    \qquad
    \text{for every }w\in\bbC
    \text{ and }B\in\mathcal A.
    \label{eq:decision-povm-invariance}
\end{equation}
First we consider asymptotically invariant probability measures in $\bbC$ similar to the setup considered in \cite{KumagaiHayashi2013}.

Identify \(\bbC\) with \(\bbR^2\). For \(r>0\), set
$F_r:=[-r,r]^2$, 
and let \(\mu_r\) be normalized Lebesgue measure on \(F_r\):
\begin{equation}
    \mu_r(B)
    :=
    \frac{\operatorname{Leb}(B\cap F_r)}
         {\operatorname{Leb}(F_r)}
    =
    \frac{\operatorname{Leb}(B\cap F_r)}
         {4r^2}.
    \label{eq:normalized-folner-measure}
\end{equation}

For every fixed \(v\in\bbC\) observe that
\begin{equation}
    \left\|
        \mu_r(\,\cdot-v)-\mu_r
    \right\|_{\mathrm{TV}}
    \rightarrow0
    \qquad
    \text{as }r\rightarrow\infty.
    \label{eq:folner-property}
\end{equation}
Indeed, both measures have constant density \(1/(4r^2)\) on translates of
the square \(F_r\), and hence
\begin{equation}
    \left\|
        \mu_r(\,\cdot-v)-\mu_r
    \right\|_{\mathrm{TV}}
    \leq
    \frac{
        \operatorname{Leb}
        \bigl(
            F_r\triangle(F_r-v)
        \bigr)
    }{4r^2}.
    \label{eq:folner-symmetric-difference}
\end{equation}
For fixed \(v\), the symmetric difference in the numerator has area of
order \(r\), whereas the denominator has order \(r^2\). This proves
\eqref{eq:folner-property}.

In particular, if \(r_j\to\infty\), set
\(\nu_j:=\mu_{r_j}\). For fixed \(v\in\bbC\), define the translated
measure
\[
    \nu_j^{\,v}(B):=\nu_j(B-v),
    \qquad B\in\mathcal B(\bbC).
\]
By \eqref{eq:folner-property},
\[
    \|\nu_j^{\,v}-\nu_j\|_{\mathrm{TV}}
    \rightarrow0.
\]

Consequently, for every Borel measurable
\(q:\bbC\to[0,1]\),
\begin{align}
&\left|
    \int_{\bbC}q(w+v)\,\nu_j(dw)
    -
    \int_{\bbC}q(w)\,\nu_j(dw)
\right|
\nonumber\\
=&
\left|
    \int_{\bbC}q(w)\,\nu_j^{\,v}(dw)
    -
    \int_{\bbC}q(w)\,\nu_j(dw)
\right|
\nonumber\\
\leq &
\|\nu_j^{\,v}-\nu_j\|_{\mathrm{TV}}
\rightarrow0.
\label{eq:bounded-function-asymptotic-invariance}
\end{align}

By an affine rescaling, the same conclusion holds for every bounded
real-valued Borel function \(q\). Hence it holds for every bounded
complex-valued Borel function as well. This property will later be used to show invariance of a limiting POVM.

Next, let
 $\mathcal K_n:=\cH^{\otimes n}.$
For a POVM \(D\) on \(\mathsf A\), define
\begin{equation}
    \Phi_D(f)
    :=
    \int_{\mathsf A}f(a)\,D(da),
    \qquad
    f\in C(\mathsf A).
    \label{eq:povm-positive-unital-map}
\end{equation}
Then
$\Phi_D:C(\mathsf A)\rightarrow\mathcal L(\mathcal K_n)$
is a positive unital linear map. In particular the following holds:
\begin{equation}
    f\geq0
    \quad\implies\quad
    \Phi_D(f)\geq0,
    \qquad
    \Phi_D(1)=I,  \qquad \|\Phi_D(f)\|
    \leq
    \|f\|_\infty.
    \label{eq:positive-unital-properties}
\end{equation}

We equip the space of POVMs with the topology of pointwise ultraweak
convergence of their associated positive unital maps. Thus
$ D_j\rightarrow D$
means that
\begin{equation}
    \Tr\!\left[
        X\Phi_{D_j}(f)
    \right]
    \rightarrow
    \Tr\!\left[
        X\Phi_D(f)
    \right]
    \label{eq:povm-topology}
\end{equation}
for every \(f\in C(\mathsf A)\) and every
\(X\in\mathcal T_1(\mathcal K_n)\).

Now let $D$ be an arbitrary decision POVM on $\mathsf A$.
Before defining its displacement averages, we note that the required
operator-valued integrals are well defined in the ultraweak sense.
Indeed, the Weyl representation $w\mapsto D(w)$ is strongly continuous
(see, e.g., \cite[Chapter~1, \S1.3]{Folland1989})
and hence so is
\[
    w\mapsto V_w=D(w)^{\otimes n}.
\]
For every $X\in\mathcal T_1(\mathcal K_n)$,
\[
    \|V_wXV_w^\dagger-V_{w_0}XV_{w_0}^\dagger\|_1
    \rightarrow0
    \qquad (w\to w_0).
\]
Indeed, this follows first for rank-one operators from
$\||u\rangle\langle v|\|_1=\|u\|\,\|v\|$, then for finite-rank
operators by linearity, and finally for arbitrary trace-class operators
by trace-norm approximation and invariance of the trace norm under
unitary conjugation. Therefore, for every
$A\in\mathcal L(\mathcal K_n)$,

\[
\left|
    \Tr[XV_w^\dagger A V_w]
    -
    \Tr[XV_{w_0}^\dagger A V_{w_0}]
\right| \leq
\|V_wXV_w^\dagger-V_{w_0}XV_{w_0}^\dagger\|_1\,\|A\|
\rightarrow0.
\]
Thus $w\mapsto V_w^\dagger A V_w$ is ultraweakly continuous, and hence
ultraweakly Borel measurable.

Define
\begin{equation}
    D_j(B)
    :=
    \int_{\bbC}
        V_w^\dagger D(B)V_w\,\nu_j(dw),
    \qquad
    B\in\mathcal A,
    \label{eq:asymptotically-averaged-povm}
\end{equation}
where the integral is understood in the ultraweak sense.

The following lemma establishes the existence of a convergent subsequence of $D_{j}$; this will be crucial for the invariant-risk reductions that follow.

\begin{lemma}\label{seq_conv_POVM}
There exists a subsequence of $D_j$ that converges to an invariant POVM $\bar{D}$.
\end{lemma}

\medskip
\noindent
\textbf{Reduction for the minimax risk}\\
\medskip
\noindent

From \eqref{eq:n-copy-displacement-covariance},
\begin{equation}
    V_w
    \rho_{z,N}^{\otimes n}
    V_w^\dagger
    =
    \rho_{z+w,N}^{\otimes n}.
    \label{eq:risk-covariance}
\end{equation}
Using the definition of \(D_j\), Tonelli's theorem, and cyclicity of the
trace, we obtain
\begin{align}
    R_{z,N}(D_j)
    &=
    \int_{\mathsf A}
        L(N,a)\,
        \Tr\!\left[
            \rho_{z,N}^{\otimes n}
            D_j(da)
        \right]
    \nonumber\\
    &=
    \int_{\mathsf A}
        L(N,a)
        \int_{\bbC}
            \Tr\!\left[
                \rho_{z,N}^{\otimes n}
                V_w^\dagger D(da)V_w
            \right]
            \nu_j(dw)
    \nonumber\\
    &=
    \int_{\bbC}
        \int_{\mathsf A}
            L(N,a)\,
            \Tr\!\left[
                V_w
                \rho_{z,N}^{\otimes n}
                V_w^\dagger
                D(da)
            \right]
        \nu_j(dw)
    \nonumber\\
    &=
    \int_{\bbC}
        \int_{\mathsf A}
            L(N,a)\,
            \Tr\!\left[
                \rho_{z+w,N}^{\otimes n}
                D(da)
            \right]
        \nu_j(dw)
    \nonumber\\
    &=
    \int_{\bbC}
        R_{z+w,N}(D)\,
        \nu_j(dw).
    \label{eq:averaged-risk-identity}
\end{align}

Therefore, for every \(z\in\bbC\) and \(N\in K\),
\begin{equation}
    R_{z,N}(D_j)
    \leq
    \sup_{u\in\bbC}R_{u,N}(D)
    \leq
    \sup_{\substack{u\in\bbC\\M\in K}}
    R_{u,M}(D).
    \label{eq:averaged-risk-bound}
\end{equation}
Taking the supremum over \(z\in\bbC\) and \(N\in K\) gives
\begin{equation}
    \sup_{\substack{z\in\bbC\\N\in K}}
    R_{z,N}(D_j)
    \leq
    \sup_{\substack{z\in\bbC\\N\in K}}
    R_{z,N}(D).
    \label{eq:averaged-maximal-risk}
\end{equation}
Fix \(z\in\bbC\) and \(N\in K\). Since \(L\) is continuous on
\(K\times\mathsf A\), the function
$f_N(a):=L(N,a)$
belongs to \(C(\mathsf A)\). By
\eqref{eq:averaged-povm-convergence}, from the proof of Lemma \ref{seq_conv_POVM}, applied with
$X=\rho_{z,N}^{\otimes n}$ and 
    $f=f_N$
we obtain
\begin{equation}
    R_{z,N}(D_{j_\ell})
    =
    \Tr\!\left[
        \rho_{z,N}^{\otimes n}
        \Phi_{D_{j_\ell}}(f_N)
    \right]
    \rightarrow
    \Tr\!\left[
        \rho_{z,N}^{\otimes n}
        \Phi_{\overline D}(f_N)
    \right]
    =
    R_{z,N}(\overline D).
    \label{eq:direct-risk-convergence}
\end{equation}
Passing to the limit in \eqref{eq:averaged-risk-bound} therefore yields
\begin{equation}
    R_{z,N}(\overline D)
    \leq
    \sup_{\substack{u\in\bbC\\M\in K}}
    R_{u,M}(D)
    \label{eq:pointwise-limit-risk-bound}
\end{equation}
for every \(z\in\bbC\) and \(N\in K\). Taking the supremum over
\((z,N)\in\bbC\times K\), we obtain
\begin{equation}
    \sup_{\substack{z\in\bbC\\N\in K}}
    R_{z,N}(\overline D)
    \leq
    \sup_{\substack{z\in\bbC\\N\in K}}
    R_{z,N}(D).
    \label{eq:invariant-risk-improvement}
\end{equation}

Thus every action-valued POVM \(D\) admits a displacement-invariant
replacement \(\overline D\) whose maximal risk is no larger.

It follows that
\begin{equation}
    \inf_{D\in\cM_{\mathrm{inv}}^{(n)}}
    \sup_{\substack{z\in\bbC\\N\in K}}
    R_{z,N}(D)
    \leq
    \inf_D
    \sup_{\substack{z\in\bbC\\N\in K}}
    R_{z,N}(D).
    \label{eq:first-minimax-direction}
\end{equation}
The reverse inequality is immediate because
$\cM_{\mathrm{inv}}^{(n)}$
is a subclass of the class of all action-valued POVMs:
\begin{equation}
    \inf_D
    \sup_{\substack{z\in\bbC\\N\in K}}
    R_{z,N}(D)
    \leq
    \inf_{D\in\cM_{\mathrm{inv}}^{(n)}}
    \sup_{\substack{z\in\bbC\\N\in K}}
    R_{z,N}(D).
    \label{eq:second-minimax-direction}
\end{equation}
Combining \eqref{eq:first-minimax-direction} and
\eqref{eq:second-minimax-direction} proves
\begin{equation}
    \inf_D
    \sup_{\substack{z\in\bbC\\N\in K}}
    R_{z,N}(D)
    =
    \inf_{D\in\cM_{\mathrm{inv}}^{(n)}}
    \sup_{\substack{z\in\bbC\\N\in K}}
    R_{z,N}(D).
    \label{eq:minimax-first-equality-proved}
\end{equation}

Finally let \(D\in\cM_{\mathrm{inv}}^{(n)}\). Then
$
    V_z^\dagger D(B)V_z=D(B)
$ and 
using
$\rho_{z,N}^{\otimes n}
    =
    V_z\rho_{0,N}^{\otimes n}V_z^\dagger,
$ we obtain
\begin{align}
    R_{z,N}(D)
    &=
    \int_{\mathsf A}
        L(N,a)\,
        \Tr\!\left[
            V_z\rho_{0,N}^{\otimes n}V_z^\dagger
            D(da)
        \right]
    \nonumber\\
    &=
    \int_{\mathsf A}
        L(N,a)\,
        \Tr\!\left[
            \rho_{0,N}^{\otimes n}
            V_z^\dagger D(da)V_z
        \right]
    \nonumber\\
    &=
    \int_{\mathsf A}
        L(N,a)\,
        \Tr\!\left[
            \rho_{0,N}^{\otimes n}D(da)
        \right]
    \nonumber\\
    &=
    R_{0,N}(D).
    \label{eq:invariant-risk-independent-displacement}
\end{align}
Hence
\begin{equation}
    \sup_{\substack{z\in\bbC\\N\in K}}
    R_{z,N}(D)
    =
    \sup_{N\in K}R_{0,N}(D)
    \label{eq:invariant-maximal-risk}
\end{equation}
for every invariant \(D\). Combining
\eqref{eq:minimax-first-equality-proved} and
\eqref{eq:invariant-maximal-risk} proves
\eqref{eq:minimax-symmetry}.

\medskip
\noindent
\textbf{Bayes-minimax reduction.}
\medskip
\noindent

Let \(\Pi\) be a probability measure supported on \(K\), and define
\begin{equation}
    B_z(D)
    :=
    \int_K
        R_{z,N}(D)\,
        \Pi(dN).
    \label{eq:integrated-risk-at-z}
\end{equation}
Thus
\begin{equation}
    r_n^{\mathrm{BM}}(\Pi)
    =
    \inf_D\sup_{z\in\bbC}B_z(D).
    \label{eq:bayes-minimax-integrated-risk}
\end{equation}

By \eqref{eq:averaged-risk-identity} and Tonelli's theorem,
\begin{align}
    B_z(D_j)
    &=
    \int_K
        R_{z,N}(D_j)\,
        \Pi(dN)
    \nonumber\\
    &=
    \int_K
        \int_{\bbC}
            R_{z+w,N}(D)\,
            \nu_j(dw)\,
        \Pi(dN)
    \nonumber\\
    &=
    \int_{\bbC}
        \left\{
            \int_K
                R_{z+w,N}(D)\,
                \Pi(dN)
        \right\}
        \nu_j(dw)
    \nonumber\\
    &=
    \int_{\bbC}
        B_{z+w}(D)\,
        \nu_j(dw).
    \label{eq:averaged-integrated-risk}
\end{align}
Consequently,
\begin{equation}
    B_z(D_j)
    \leq
    \sup_{u\in\bbC}B_u(D)
    \label{eq:averaged-integrated-risk-bound}
\end{equation}
for every \(z\in\bbC\).

Since \(K\times\mathsf A\) is compact and \(L\) is continuous, there is a
finite constant
\begin{equation}
    C_L
    :=
    \sup_{\substack{N\in K\\a\in\mathsf A}}
    L(N,a)
    <
    \infty.
    \label{eq:uniform-loss-bound}
\end{equation}
Therefore,
\begin{equation}
    0
    \leq
    R_{z,N}(D_{j_\ell})
    \leq
    C_L
    \label{eq:uniform-risk-bound}
\end{equation}
for every \(z\in\bbC\), \(N\in K\), and \(\ell\).

For every fixed \(z\) and \(N\),
\eqref{eq:direct-risk-convergence} gives
\[
    R_{z,N}(D_{j_\ell})
    \rightarrow
    R_{z,N}(\overline D).
\]
The uniform bound \eqref{eq:uniform-risk-bound} permits the use of dominated
convergence under the \(\Pi\)-integral. Hence
\begin{align}
    B_z(D_{j_\ell})
    &=
    \int_K
        R_{z,N}(D_{j_\ell})\,
        \Pi(dN)
    \nonumber\\
    &\rightarrow
    \int_K
        R_{z,N}(\overline D)\,
        \Pi(dN)
    \nonumber\\
    &=
    B_z(\overline D).
    \label{eq:integrated-risk-convergence}
\end{align}

Passing to the limit in
\eqref{eq:averaged-integrated-risk-bound} gives
\begin{equation}
    B_z(\overline D)
    \leq
    \sup_{u\in\bbC}B_u(D)
    \label{eq:limit-integrated-risk-bound}
\end{equation}
for every \(z\in\bbC\). Taking the supremum over \(z\),
\begin{equation}
    \sup_{z\in\bbC}B_z(\overline D)
    \leq
    \sup_{z\in\bbC}B_z(D).
    \label{eq:bayes-minimax-invariant-improvement}
\end{equation}
Thus every decision POVM has a displacement-invariant replacement with no
larger Bayes-minimax risk. It follows, exactly as in the minimax case, that
\begin{equation}
    \inf_D\sup_{z\in\bbC}B_z(D)
    =
    \inf_{D\in\cM_{\mathrm{inv}}^{(n)}}
    \sup_{z\in\bbC}B_z(D).
    \label{eq:bayes-minimax-restriction}
\end{equation}

Finally, if \(D\in\cM_{\mathrm{inv}}^{(n)}\), then
\eqref{eq:invariant-risk-independent-displacement} gives
\[
    R_{z,N}(D)=R_{0,N}(D)
\]
for every \(z\in\bbC\) and \(N\in K\). Therefore,
\begin{align}
    B_z(D)
    &=
    \int_K
        R_{z,N}(D)\,
        \Pi(dN)
    \nonumber\\
    &=
    \int_K
        R_{0,N}(D)\,
        \Pi(dN),
    \label{eq:invariant-integrated-risk}
\end{align}
which is independent of \(z\). Hence
\begin{equation}
    \sup_{z\in\bbC}B_z(D)
    =
    \int_K
        R_{0,N}(D)\,
        \Pi(dN).
    \label{eq:invariant-bayes-minimax-risk}
\end{equation}
Combining \eqref{eq:bayes-minimax-restriction} and
\eqref{eq:invariant-bayes-minimax-risk} proves
\eqref{eq:bayes-minimax-invariant}.
\end{proof}

\section{Proof of technical lemmas}{\label{sec:tech_lem}}

\begin{proof}[Proof of Lemma \ref{phi_z_ext}]

Let
\[
\mathcal{E}_G
:=
\{E\in\mathcal{A}_G:0\leq E\leq I\}
\]
denote the convex set of invariant effects. Uniqueness of the coefficients implies that $\varphi_z$ is affine on the
invariant effect space. Indeed, for invariant effects $E_1,E_2$ and
$t\in[0,1]$, under the identifiability condition, we have,
\[
    \varphi_z(tE_1+(1-t)E_2)
    =
    t\varphi_z(E_1)+(1-t)\varphi_z(E_2).
\]
Moreover,
\[
    \varphi_z(0)=0,
    \qquad
    \varphi_z(I)=1,
    \qquad
    0\leq \varphi_z(E)\leq1.
\]

 Since $\varphi_z$ is affine and
$\varphi_z(0)=0$, for every $E\in\mathcal{E}_G$ and $t\in[0,1]$ we have
\[
\varphi_z(tE)
=
\varphi_z\bigl(tE+(1-t)0\bigr)
=
t\varphi_z(E).
\]
Furthermore, if $E,F,E+F\in\mathcal{E}_G$, then
\[
\frac{1}{2}\varphi_z(E+F)
=
\varphi_z\left(\frac{E+F}{2}\right)
=
\frac{1}{2}\varphi_z(E)+\frac{1}{2}\varphi_z(F),
\]
and therefore
\[
\varphi_z(E+F)=\varphi_z(E)+\varphi_z(F).
\]

We first extend $\varphi_z$ to the positive cone
$(\mathcal{A}_G)_+$. For $A\in(\mathcal{A}_G)_+$, choose $r>0$ such that
$A\leq rI$, and define
\[
\widetilde{\varphi}_z(A)
:=
r\,\varphi_z\left(\frac{A}{r}\right).
\]
This definition is independent of the choice of $r$. Indeed, if $s\geq r$,
then
\[
\varphi_z\left(\frac{A}{s}\right)
=
\varphi_z\left(\frac{r}{s}\frac{A}{r}\right)
=
\frac{r}{s}\varphi_z\left(\frac{A}{r}\right),
\]
so
\[
s\,\varphi_z\left(\frac{A}{s}\right)
=
r\,\varphi_z\left(\frac{A}{r}\right).
\]

The map $\widetilde{\varphi}_z$ is additive on $(\mathcal{A}_G)_+$.
Indeed, for $A,B\in(\mathcal{A}_G)_+$, choose $r>0$ such that
$A+B\leq rI$. Then
\[
\begin{aligned}
\widetilde{\varphi}_z(A+B)
&=
r\,\varphi_z\left(\frac{A+B}{r}\right)\\
&=
r\left[
\varphi_z\left(\frac{A}{r}\right)
+
\varphi_z\left(\frac{B}{r}\right)
\right]\\
&=
\widetilde{\varphi}_z(A)+\widetilde{\varphi}_z(B).
\end{aligned}
\]
It is also positively homogeneous.

Every self-adjoint $X\in\mathcal{A}_G$ can be written as
\[
X=A-B
\]
for some $A,B\in(\mathcal{A}_G)_+$. Define
\[
\widetilde{\varphi}_z(X)
:=
\widetilde{\varphi}_z(A)-\widetilde{\varphi}_z(B).
\]
This is well defined: if $X=A-B=C-D$, then $A+D=B+C$, and additivity on
the positive cone gives
\[
\widetilde{\varphi}_z(A)+\widetilde{\varphi}_z(D)
=
\widetilde{\varphi}_z(B)+\widetilde{\varphi}_z(C).
\]
Thus
\[
\widetilde{\varphi}_z(A)-\widetilde{\varphi}_z(B)
=
\widetilde{\varphi}_z(C)-\widetilde{\varphi}_z(D).
\]

Finally, for $X,Y\in(\mathcal{A}_G)_{\mathrm{sa}}$, define
\[
\widetilde{\varphi}_z(X+iY)
:=
\widetilde{\varphi}_z(X)
+i\widetilde{\varphi}_z(Y).
\]
This yields a complex-linear functional on $\mathcal{A}_G$. Since
\[
\widetilde{\varphi}_z(A)\geq 0
\qquad
\text{for every } A\in(\mathcal{A}_G)_+,
\]
the extension is positive. Moreover,
\[
\widetilde{\varphi}_z(I)
=
\varphi_z(I)
=
1,
\]
so it is unital. By construction, it agrees with $\varphi_z$ on
$\mathcal{E}_G$.
\end{proof}

\begin{lemma}[Posterior covariance in the multinomial experiment]
\label{lem:multinomial-posterior-covariance}
Let $d\geq2$ be fixed, and define
\[
    \Theta_{\downarrow}
    :=
    \left\{
        \theta=(\theta_1,\ldots,\theta_d)^\top\in[0,1]^d:
        \theta_1\geq\cdots\geq\theta_d,\quad
        \sum_{i=1}^d\theta_i=1
    \right\}.
\]
Let
\[
    \Theta_{\mathrm{reg}}
    :=
    \left\{
        \theta\in\Theta_{\downarrow}:
        \theta_1>\cdots>\theta_d>0
    \right\}.
\]
Fix $\theta_0\in\Theta_{\mathrm{reg}}$. Write
\[
    \vartheta=(\theta_1,\ldots,\theta_{d-1})^\top,
    \qquad
    \theta_d=1-\sum_{i=1}^{d-1}\vartheta_i,
\]
and let $\vartheta_0$ denote the intrinsic coordinates of $\theta_0$.
Suppose that
\[
    N_n=(N_{n,1},\ldots,N_{n,d})
    \sim\operatorname{Mult}_d(n,\theta_0).
\]
Let $\Pi$ be a prior on $\Theta_{\downarrow}$ that is absolutely
continuous with respect to $(d-1)$-dimensional Lebesgue measure on the
affine hyperplane containing the simplex. Assume that its density has a
version that is continuous and strictly positive at $\theta_0$.
Define
\[
    V_{\vartheta_0}
    :=
    \operatorname{diag}(\theta_{0,1},\ldots,\theta_{0,d-1})
    -\vartheta_0\vartheta_0^\top.
\]
Then
\[
    \mathbb E_{\theta_0}
    \left[
        \left\|
            n\operatorname{Var}(\vartheta\mid N_n)
            -V_{\vartheta_0}
        \right\|_{\mathrm{op}}
    \right]
    \rightarrow0.
\]
\end{lemma}

\begin{proof}
Set $k:=d-1$ and define
\[
    \Omega_{\downarrow}
    :=
    \left\{
        \vartheta\in\mathbb R^k:
        \vartheta_1>\cdots>\vartheta_k>
        1-\sum_{i=1}^k\vartheta_i>0
    \right\}.
\]
The map
\[
    \theta(\vartheta)
    :=
    \left(
        \vartheta_1,\ldots,\vartheta_k,
        1-\sum_{i=1}^k\vartheta_i
    \right)^\top
\]
is an affine parametrization of $\Theta_{\mathrm{reg}}$.
Since $\Pi$ is absolutely continuous with respect to the Lebesgue measure, the complement $\Theta_{\downarrow}\setminus\Theta_{\mathrm{reg}}$
has $\Pi$-measure zero. Every posterior is absolutely
continuous with respect to $\Pi$; hence posterior integration may also be
restricted to $\Theta_{\mathrm{reg}}$.
Let $\widetilde\pi$ denote the induced prior density on
$\Omega_{\downarrow}$.

For $\vartheta\in\Omega_{\downarrow}$, the multinomial likelihood, up to
a factor independent of $\vartheta$, is
\[
    L_n(\vartheta)
    =
    \prod_{i=1}^k\vartheta_i^{N_{n,i}}
    \left(1-\sum_{i=1}^k\vartheta_i\right)^{N_{n,d}},
\]
with log-likelihood
\[
    \ell_n(\vartheta)
    =
    \sum_{i=1}^kN_{n,i}\log\vartheta_i
    +N_{n,d}\log\left(1-\sum_{i=1}^k\vartheta_i\right).
\]
Write
\[
    \widehat\theta_n:=N_n/n,
    \qquad
    \widehat\vartheta_n
    :=(N_{n,1}/n,\ldots,N_{n,k}/n)^\top.
\]

Since $\theta_0\in\Theta_{\mathrm{reg}}$,
\[
    \delta_0
    :=
    \min\left\{
        \theta_{0,d},
        \theta_{0,1}-\theta_{0,2},\ldots,
        \theta_{0,d-1}-\theta_{0,d}
    \right\}>0.
\]
Choose $r>0$ sufficiently small that
$\overline{B(\vartheta_0,3r)}\subset\Omega_{\downarrow}$ and
\[
    0<c_\pi\leq\widetilde\pi(\vartheta)\leq C_\pi<\infty,
    \qquad
    \vartheta\in\overline{B(\vartheta_0,3r)}.
\]
Set $U:=B(\vartheta_0,2r)$ and choose $\varepsilon_0>0$ such that
$4\varepsilon_0<\delta_0$ and $\sqrt{k}\,\varepsilon_0<r$.
Define
\[
    G_n
    :=
    \left\{
        \max_{1\leq i\leq d}
        |\widehat\theta_{n,i}-\theta_{0,i}|
        \leq\varepsilon_0
    \right\}.
\]
On $G_n$,
\[
    \widehat\theta_{n,d}
    \geq\theta_{0,d}-\varepsilon_0>3\varepsilon_0,
    \qquad
    \widehat\theta_{n,i}-\widehat\theta_{n,i+1}
    \geq\theta_{0,i}-\theta_{0,i+1}-2\varepsilon_0>0
\]
for $1\leq i<d$. Thus $\widehat\theta_n\in\Theta_{\mathrm{reg}}$ and
$\|\widehat\vartheta_n-\vartheta_0\|<r$ on $G_n$. In particular,
$\widehat\vartheta_n$ lies in a fixed compact subset of
$\Omega_{\downarrow}$. Since the empirical proportions maximize the
multinomial likelihood over the full simplex, they also maximize it over
$\Omega_{\downarrow}$ on this event. Hoeffding's inequality gives
\begin{equation}
    \mathbb P_{\theta_0}(G_n^c)
    \leq2d\exp(-2n\varepsilon_0^2).
    \label{eq:good-event-probability}
\end{equation}

All likelihood-ratio and information-matrix calculations at
$\widehat\vartheta_n$ below are performed on $G_n$. Auxiliary quantities
defined only on this event may be extended arbitrarily to $G_n^c$ when
stating convergence in probability, since
$\mathbb P_{\theta_0}(G_n^c)\to0$. The posterior distribution itself is
always the actual posterior, including on $G_n^c$.

For $\vartheta\in\Omega_{\downarrow}$, on $G_n$,
\[
\begin{aligned}
    \ell_n(\vartheta)-\ell_n(\widehat\vartheta_n)
    &=
    n\sum_{i=1}^d\widehat\theta_{n,i}
    \log\frac{\theta_i(\vartheta)}{\widehat\theta_{n,i}}\\
    &=-nD\!\left(
        \widehat\theta_n\,\middle\|\,\theta(\vartheta)
    \right),
\end{aligned}
\]
where $D(p\|q):=\sum_{i=1}^dp_i\log(p_i/q_i)$ uses natural logarithms.
Consequently,
\begin{equation}
    \frac{L_n(\vartheta)}{L_n(\widehat\vartheta_n)}
    =
    \exp\left\{
        -nD\!\left(
            \widehat\theta_n\,\middle\|\,\theta(\vartheta)
        \right)
    \right\}.
    \label{eq:likelihood-kl}
\end{equation}
This identity supplies both the local quadratic approximation and the
posterior tail bounds.

Introduce the local coordinate
$h:=\sqrt n(\vartheta-\widehat\vartheta_n)$ and its domain
\[
    \mathcal H_n
    :=
    \left\{
        h\in\mathbb R^k:
        \widehat\vartheta_n+h/\sqrt n\in\Omega_{\downarrow}
    \right\}.
\]
For every fixed $M<\infty$ and all sufficiently large $n$, on $G_n$,
\[
    \widehat\vartheta_n+h/\sqrt n\in U
    \quad\text{whenever }\|h\|\leq M.
\]
In particular, $\{h:\|h\|\leq M\}\subset\mathcal H_n$. Thus the
ordering constraints do not truncate the local limit around $\theta_0$.

Define the population Fisher information in intrinsic coordinates by
\[
    I(\vartheta)
    :=
    \operatorname{diag}\left(
        \frac1{\vartheta_1},\ldots,\frac1{\vartheta_k}
    \right)
    +\frac1{\theta_d(\vartheta)}\mathbf1_k\mathbf1_k^\top,
    \qquad
    \theta_d(\vartheta):=1-\sum_{i=1}^k\vartheta_i.
\]
The observed information at the empirical proportions satisfies
\begin{equation}
\begin{aligned}
    -\frac1n\nabla^2\ell_n(\widehat\vartheta_n)
    &=I(\widehat\vartheta_n)\\
    &=
    \operatorname{diag}\left(
        \frac1{\widehat\theta_{n,1}},\ldots,
        \frac1{\widehat\theta_{n,k}}
    \right)
    +\frac1{\widehat\theta_{n,d}}\mathbf1_k\mathbf1_k^\top.
\end{aligned}
    \label{eq:empirical-information}
\end{equation}
On the fixed neighborhood $B(\vartheta_0,3r)$, the third derivatives of
$n^{-1}\ell_n$ are uniformly bounded, independently of the sample.
Since $\nabla\ell_n(\widehat\vartheta_n)=0$ on $G_n$, Taylor's theorem
gives, for each fixed $M<\infty$ and all sufficiently large $n$,
\begin{equation}
    \sup_{\|h\|\leq M}
    \left|
        \ell_n(\widehat\vartheta_n+h/\sqrt n)
        -\ell_n(\widehat\vartheta_n)
        +\frac12h^\top I(\widehat\vartheta_n)h
    \right|
    \leq\frac{CM^3}{\sqrt n}
    \rightarrow0
    \label{eq:local-likelihood-expansion}
\end{equation}
on $G_n$. Moreover, $\widehat\vartheta_n\to\vartheta_0$ in
$\mathbb P_{\theta_0}$-probability, so
$I(\widehat\vartheta_n)\to I(\vartheta_0)$ in probability, where
\begin{equation}
    I(\vartheta_0)
    =
    \operatorname{diag}\left(
        \frac1{\theta_{0,1}},\ldots,\frac1{\theta_{0,k}}
    \right)
    +\frac1{\theta_{0,d}}\mathbf1_k\mathbf1_k^\top.
    \label{eq:fisher-information}
\end{equation}
Its inverse is
\begin{equation}
    I(\vartheta_0)^{-1}
    =V_{\vartheta_0}
    =\operatorname{diag}(\theta_{0,1},\ldots,\theta_{0,k})
     -\vartheta_0\vartheta_0^\top.
    \label{eq:inverse-information}
\end{equation}

For every sample, the posterior density on $\Omega_{\downarrow}$ is
\[
    \widetilde\pi_n(\vartheta\mid N_n)
    =
    \frac{L_n(\vartheta)\widetilde\pi(\vartheta)}{
        \displaystyle\int_{\Omega_{\downarrow}}
        L_n(u)\widetilde\pi(u)\,du}.
\]
Let $q_n(\cdot\mid N_n)$ denote the posterior density of
$H_n:=\sqrt n(\vartheta-\widehat\vartheta_n)$. On $G_n$, define
\[
    g_n(h)
    :=
    \exp\left\{
        \ell_n(\widehat\vartheta_n+h/\sqrt n)
        -\ell_n(\widehat\vartheta_n)
    \right\}
    \widetilde\pi(\widehat\vartheta_n+h/\sqrt n)
    1_{\mathcal H_n}(h),
\]
where the factors are evaluated only for $h\in\mathcal H_n$, and
$g_n(h)$ is defined to be zero otherwise. Then
\[
    q_n(h\mid N_n)=\frac{g_n(h)}{Z_n},
    \qquad
    Z_n:=\int_{\mathbb R^k}g_n(h)\,dh.
\]
The change of variables also gives, on $G_n$,
\begin{equation}
    \int_{\Omega_{\downarrow}}
    \frac{L_n(u)}{L_n(\widehat\vartheta_n)}
    \widetilde\pi(u)\,du
    =n^{-k/2}Z_n.
    \label{eq:intrinsic-posterior-normalizer}
\end{equation}
Write
\[
    \phi_{\vartheta_0}(h)
    :=
    \frac{\det\{I(\vartheta_0)\}^{1/2}}{(2\pi)^{k/2}}
    \exp\left\{-\frac12h^\top I(\vartheta_0)h\right\}
\]
for the density of $N_k(0,V_{\vartheta_0})$.

We first establish convergence on bounded sets. For every fixed
$M<\infty$, the likelihood expansion and convergence of the information
matrix yield
\begin{equation}
    \sup_{\|h\|\leq M}
    \left|
        e^{\ell_n(\widehat\vartheta_n+h/\sqrt n)
             -\ell_n(\widehat\vartheta_n)}
        -e^{-h^\top I(\vartheta_0)h/2}
    \right|
    \rightarrow0
    \label{local_expansion_exp}
\end{equation}
in $\mathbb P_{\theta_0}$-probability. Continuity of
$\widetilde\pi$ at $\vartheta_0$ also gives
\[
    \sup_{\|h\|\leq M}
    \left|
        \widetilde\pi(\widehat\vartheta_n+h/\sqrt n)
        -\widetilde\pi(\vartheta_0)
    \right|
    \rightarrow0
\]
in probability. Hence $g_n$ converges uniformly on every fixed bounded
set, in probability, to
\begin{equation}
    g(h):=\widetilde\pi(\vartheta_0)
    \exp\left\{-\frac12h^\top I(\vartheta_0)h\right\}.
    \label{eq:local-density-limit}
\end{equation}

We next control the posterior tails. Pinsker's inequality gives, for
$\vartheta\in\Omega_{\downarrow}$ and on $G_n$,
\[
    D\!\left(\widehat\theta_n\,\middle\|\,\theta(\vartheta)\right)
    \geq\frac12\|\widehat\theta_n-\theta(\vartheta)\|_1^2
    \geq\frac12\|\widehat\vartheta_n-\vartheta\|^2.
\]
Therefore
\begin{equation}
    \frac{L_n(\vartheta)}{L_n(\widehat\vartheta_n)}
    \leq
    \exp\left\{-\frac n2\|\vartheta-\widehat\vartheta_n\|^2\right\},
    \qquad \vartheta\in\Omega_{\downarrow},
    \label{eq:local-gaussian-bound}
\end{equation}
or, in local coordinates,
\begin{equation}
    \frac{L_n(\widehat\vartheta_n+h/\sqrt n)}{
          L_n(\widehat\vartheta_n)}
    \leq e^{-\|h\|^2/2},
    \qquad h\in\mathcal H_n.
    \label{eq:local-h-bound}
\end{equation}
Since $U=B(\vartheta_0,2r)$ and
$\|\widehat\vartheta_n-\vartheta_0\|<r$ on $G_n$,
\begin{equation}
    \inf_{\vartheta\in\Omega_{\downarrow}\setminus U}
    D\!\left(\widehat\theta_n\,\middle\|\,\theta(\vartheta)\right)
    \geq r^2/2=:c_U>0.
    \label{outside_U_bound}
\end{equation}
Consequently,
\begin{equation}
    \sup_{\vartheta\in\Omega_{\downarrow}\setminus U}
    \frac{L_n(\vartheta)}{L_n(\widehat\vartheta_n)}
    \leq e^{-nc_U}.
    \label{eq:outside-neighborhood-bound}
\end{equation}

A lower bound for $Z_n$ follows by restricting its integral to
$\|h\|\leq1$. For all sufficiently large $n$, on $G_n$, this ball is
mapped into $U$. The expansion \eqref{eq:local-likelihood-expansion},
the uniform upper bound for $I(\widehat\vartheta_n)$ on $G_n$, and the
local lower bound $c_\pi$ for the prior density therefore give a
deterministic constant $c_Z>0$ such that
\begin{equation}
    Z_n\geq c_Z
    \label{eq:normalizing-lower-bound}
\end{equation}
on $G_n$ for all sufficiently large $n$.

Split the posterior tail according to whether
$\widehat\vartheta_n+h/\sqrt n$ belongs to $U$. On the first part,
use $\widetilde\pi\leq C_\pi$ and \eqref{eq:local-h-bound}. On the
second part, use \eqref{eq:outside-neighborhood-bound}, boundedness of
the simplex, and
\[
    \int_{\mathcal H_n}
    \widetilde\pi(\widehat\vartheta_n+h/\sqrt n)\,dh
    =n^{k/2}.
\]
In particular, $1+\|h\|^2\leq Cn$ on $\mathcal H_n$. Together with
\eqref{eq:normalizing-lower-bound}, these estimates give, on $G_n$,
\begin{equation}
\begin{aligned}
    &\int_{\|h\|>M}(1+\|h\|^2)q_n(h\mid N_n)\,dh\\
    &\qquad\leq
    C\int_{\|h\|>M}(1+\|h\|^2)e^{-\|h\|^2/2}\,dh
    +Cn^{1+k/2}e^{-nc_U}
\end{aligned}
    \label{eq:posterior-tail-bound}
\end{equation}
for every $M>0$ and all sufficiently large $n$.
 The right-hand side
has the form $\varepsilon_M+r_n$, where $\varepsilon_M\downarrow0$
as $M\to\infty$ and $r_n\to0$ as $n\to\infty$. Thus the weighted
posterior tails vanish by first letting $n\to\infty$ and then
$M\to\infty$.

We now identify the normalization. Set
\[
    Z:=\int_{\mathbb R^k}g(h)\,dh
      =\widetilde\pi(\vartheta_0)(2\pi)^{k/2}
        \det\{I(\vartheta_0)\}^{-1/2}>0,
\]
so that $\phi_{\vartheta_0}=g/Z$. For fixed $M<\infty$, local uniform
convergence gives
\[
    A_{n,M}:=\int_{\|h\|\leq M}g_n(h)\,dh
    \rightarrow
    A_M:=\int_{\|h\|\leq M}g(h)\,dh
\]
in probability. By \eqref{eq:posterior-tail-bound}, on $G_n$,
\[
    \frac{A_{n,M}}{Z_n}
    =\int_{\|h\|\leq M}q_n(h\mid N_n)\,dh
    \geq1-\varepsilon_M-r_n.
\]
For $M$ sufficiently large that $\varepsilon_M<1/2$ and then $n$
sufficiently large, this yields
\[
    A_{n,M}\leq Z_n\leq\frac{A_{n,M}}{1-\varepsilon_M-r_n}.
\]
Since $A_M\uparrow Z$ and $\varepsilon_M\downarrow0$,
first letting $n\to\infty$ and then  $M\to\infty$, shows that
$Z_n\to Z$ in $\mathbb P_{\theta_0}$-probability.

It follows that, for every fixed $M$,
\[
    \int_{\|h\|\leq M}(1+\|h\|^2)
    |q_n(h\mid N_n)-\phi_{\vartheta_0}(h)|\,dh
    \rightarrow0
\]
in probability. On the complement,
\[
\begin{aligned}
    &\int_{\|h\|>M}(1+\|h\|^2)
      |q_n(h\mid N_n)-\phi_{\vartheta_0}(h)|\,dh\\
    &\qquad\leq
      \int_{\|h\|>M}(1+\|h\|^2)q_n(h\mid N_n)\,dh
      +\int_{\|h\|>M}(1+\|h\|^2)\phi_{\vartheta_0}(h)\,dh.
\end{aligned}
\]
The first term is at most $\varepsilon_M+r_n$ on $G_n$, and the second
tends to zero as $M\to\infty$. Since
$\mathbb P_{\theta_0}(G_n^c)\to0$, we conclude for the actual posterior
that
\begin{equation}
    \int_{\mathbb R^k}(1+\|h\|^2)
    |q_n(h\mid N_n)-\phi_{\vartheta_0}(h)|\,dh
    \rightarrow0
    \label{posterior-moment-conv}
\end{equation}
in $\mathbb P_{\theta_0}$-probability.

The weighted convergence \eqref{posterior-moment-conv} directly implies
convergence of the posterior first and second moments:
\[
    \mathbb E[H_n\mid N_n]\rightarrow0,
    \qquad
    \mathbb E[H_nH_n^\top\mid N_n]\rightarrow V_{\vartheta_0}
\]
in probability. Hence
\begin{equation}
    \operatorname{Var}(H_n\mid N_n)
    \rightarrow V_{\vartheta_0}
    \label{eq:posterior-covariance-probability}
\end{equation}
in probability.

To obtain convergence in expectation, the tail bound gives deterministic
constants $C<\infty$ and $n_0$ such that, for every $n\geq n_0$, on $G_n$,
\[
    \|\operatorname{Var}(H_n\mid N_n)\|_{\mathrm{op}}
    \leq\mathbb E[\|H_n\|^2\mid N_n]\leq C.
\]
On $G_n^c$, the posterior parameter $\vartheta$ lies in the bounded
ordered simplex and $\widehat\vartheta_n$ in the bounded full simplex.
Thus $\|H_n\|^2\leq Cn$ for every posterior draw, and
\[
    \mathbb E_{\theta_0}\left[
        \|\operatorname{Var}(H_n\mid N_n)\|_{\mathrm{op}}
        1_{G_n^c}
    \right]
    \leq Cn\,\mathbb P_{\theta_0}(G_n^c)
    \rightarrow0.
\]
Boundedness on $G_n$, convergence in probability, and the last estimate
therefore imply
\[
    \mathbb E_{\theta_0}\left[
        \|\operatorname{Var}(H_n\mid N_n)-V_{\vartheta_0}\|_{\mathrm{op}}
    \right]\rightarrow0.
\]
Finally, $\widehat\vartheta_n$ is fixed conditional on $N_n$, so
$\operatorname{Var}(H_n\mid N_n)=n\operatorname{Var}(\vartheta\mid N_n)$.
This proves the lemma.
\end{proof}

\begin{lemma}[Posterior variance of a smooth functional]
\label{lem:multinomial-functional-posterior-variance}
Under the assumptions of
Lemma~\ref{lem:multinomial-posterior-covariance}, let
$\psi:\Theta_{\downarrow}\to\mathbb R^q$ satisfy
Assumption~\ref{ass:psi}. Then
\[
    n\mathbb E_{\theta_0}\left[
        \operatorname{Tr}\operatorname{Var}\{\psi(\theta)\mid N_n\}
    \right]
    \rightarrow
    \operatorname{Tr}\left[
        D\psi(\theta_0)\Sigma(\theta_0)D\psi(\theta_0)^\top
    \right],
\]
where
\[
    \Sigma(\theta_0)
    :=\operatorname{diag}(\theta_0)-\theta_0\theta_0^\top.
\]
\end{lemma}

\begin{proof}
Use the intrinsic parametrization and posterior notation of the preceding
proof. Define
\[
    g(\vartheta):=\psi\{\theta(\vartheta)\},
    \qquad \vartheta\in\Omega_{\downarrow},
    \qquad
    A:=\begin{pmatrix}I_{d-1}\\-\mathbf1_{d-1}^\top\end{pmatrix}.
\]
By the chain rule,
\begin{equation}
    J:=Dg(\vartheta_0)=D\psi(\theta_0)A.
    \label{eq:functional-chain-rule}
\end{equation}
Assumption~\ref{ass:psi} gives a continuously differentiable extension of
$g$ to a neighborhood of the compact convex set
$\overline{\Omega_{\downarrow}}$. In particular, $g$ is Lipschitz on
this set.

For $u,v\in\Omega_{\downarrow}$, write
\[
    g(u)-g(v)=J(u-v)+r(u,v).
\]
There is a constant $C_r<\infty$ such that
$\|r(u,v)\|\leq C_r\|u-v\|$ throughout the domain. For sufficiently
small $\varepsilon>0$, put
\[
    \omega(\varepsilon)
    :=\sup_{\|u-\vartheta_0\|\leq\varepsilon}
      \|Dg(u)-J\|_{\mathrm{op}}.
\]
Then $\omega(\varepsilon)\to0$ as $\varepsilon\downarrow0$, and the
mean-value formula along the line segment from $v$ to $u$ gives
\[
    \|r(u,v)\|\leq\omega(\varepsilon)\|u-v\|
    \quad\text{for }u,v\in\overline{B(\vartheta_0,\varepsilon)}.
\]

Let $\vartheta'$ and $\vartheta''$ be conditionally independent draws
from the posterior given $N_n$. We have
\begin{equation}
    n\operatorname{Tr}\operatorname{Var}\{g(\vartheta)\mid N_n\}
    =\frac n2\mathbb E\left[
        \|g(\vartheta')-g(\vartheta'')\|^2\,\middle|\,N_n
    \right].
    \label{eq:variance-two-copy}
\end{equation}
Define
\[
    H_n':=\sqrt n(\vartheta'-\widehat\vartheta_n),
    \qquad
    H_n'':=\sqrt n(\vartheta''-\widehat\vartheta_n),
    \qquad
    \widetilde r_n:=\sqrt n\,r(\vartheta',\vartheta'').
\]
Then
\[
    \sqrt n\{g(\vartheta')-g(\vartheta'')\}
    =J(H_n'-H_n'')+\widetilde r_n.
\]
We shall show that
\begin{equation}
    \mathbb E_{\theta_0}R_n\rightarrow0,
    \quad \text {where }
    R_n:=\mathbb E[\|\widetilde r_n\|^2\mid N_n].
    \label{eq:multinomial-functional-remainder-L1}
\end{equation}

Fix $\varepsilon>0$ small enough that
$\overline{B(\vartheta_0,\varepsilon)}\subset\Omega_{\downarrow}$,
and set
\[
    E_{n,\varepsilon}
    :=G_n\cap
    \{\|\widehat\vartheta_n-\vartheta_0\|\leq\varepsilon/2\}.
\]
On this event, $\|\vartheta-\vartheta_0\|>\varepsilon$ implies
$\|\vartheta-\widehat\vartheta_n\|>\varepsilon/2$. Thus
\eqref{eq:local-gaussian-bound} gives
$L_n(\vartheta)/L_n(\widehat\vartheta_n)\leq e^{-n\varepsilon^2/8}$
on this posterior tail. Equations
\eqref{eq:intrinsic-posterior-normalizer} and
\eqref{eq:normalizing-lower-bound} also give, on $G_n$,
\[
    \int_{\Omega_{\downarrow}}
    \frac{L_n(u)}{L_n(\widehat\vartheta_n)}\widetilde\pi(u)\,du
    \geq c_Zn^{-k/2}
\]
for all sufficiently large $n$. Consequently, on $E_{n,\varepsilon}$,
\[
    \Pi\!\left(
        \|\vartheta-\vartheta_0\|>\varepsilon\,\middle|\,N_n
    \right)
    \leq c_Z^{-1}n^{k/2}e^{-n\varepsilon^2/8}.
\]
Hoeffding's inequality gives
$\mathbb P_{\theta_0}(E_{n,\varepsilon}^c)
 \leq C_\varepsilon e^{-c_\varepsilon n}$.
Thus we obtain
\begin{equation}
    \mathbb E_{\theta_0}\Pi\!\left(
        \|\vartheta-\vartheta_0\|>\varepsilon\,\middle|\,N_n
    \right)
    \leq C_\varepsilon n^{k/2}e^{-c_\varepsilon n},
    \label{eq:multinomial-posterior-exterior-probability}
\end{equation}
after adjusting the positive constants, which may depend on $\varepsilon$
but not on $n$.

When both posterior draws lie in
$\overline{B(\vartheta_0,\varepsilon)}$, use the local bound involving
$\omega(\varepsilon)$. On the complement, use
$\|r(u,v)\|\leq C_r\|u-v\|$ and boundedness of the parameter domain.
The union bound and the identity
\[
    n\mathbb E[\|\vartheta'-\vartheta''\|^2\mid N_n]
    =2\operatorname{Tr}\operatorname{Var}(H_n\mid N_n)
\]
then yield
\[
    R_n
    \leq
    2\omega(\varepsilon)^2
      \operatorname{Tr}\operatorname{Var}(H_n\mid N_n)
    +Cn\,\Pi\!\left(
        \|\vartheta-\vartheta_0\|>\varepsilon\,\middle|\,N_n
    \right).
\]
By Lemma~\ref{lem:multinomial-posterior-covariance},
$\mathbb E_{\theta_0}\operatorname{Tr}\operatorname{Var}(H_n\mid N_n)$
is bounded for all sufficiently large $n$. Taking expectations and using
\eqref{eq:multinomial-posterior-exterior-probability} gives
\[
    \mathbb E_{\theta_0}R_n
    \leq C\omega(\varepsilon)^2
       +C_\varepsilon n^{1+k/2}e^{-c_\varepsilon n}.
\]
Letting first $n\to\infty$ and then $\varepsilon\downarrow0$ proves
\eqref{eq:multinomial-functional-remainder-L1}. Since $R_n\geq0$, it
also follows that $R_n\to0$ in $\mathbb P_{\theta_0}$-probability.

Finally, let
\[
    T_n:=\operatorname{Tr}\left[
        J\operatorname{Var}(H_n\mid N_n)J^\top
    \right].
\]
Expanding the squared norm in \eqref{eq:variance-two-copy} and applying
conditional Cauchy-Schwarz gives
\[
    \left|
        n\operatorname{Tr}\operatorname{Var}\{g(\vartheta)\mid N_n\}
        -T_n
    \right|
    \leq\sqrt{2T_nR_n}+\frac12R_n.
\]
The preceding covariance lemma implies that
$\sup_n\mathbb E_{\theta_0}T_n<\infty$. Therefore
\[
\begin{aligned}
    &\mathbb E_{\theta_0}\left|
        n\operatorname{Tr}\operatorname{Var}\{g(\vartheta)\mid N_n\}
        -T_n
    \right|\\
    &\qquad\leq
    \left\{2\mathbb E_{\theta_0}T_n\,
              \mathbb E_{\theta_0}R_n\right\}^{1/2}
    +\frac12\mathbb E_{\theta_0}R_n
    \rightarrow0.
\end{aligned}
\]
Again by Lemma~\ref{lem:multinomial-posterior-covariance},
$T_n\to\operatorname{Tr}(JV_{\vartheta_0}J^\top)$ in
$L^1(\mathbb P_{\theta_0})$. Hence
\[
    n\mathbb E_{\theta_0}\left[
        \operatorname{Tr}\operatorname{Var}\{g(\vartheta)\mid N_n\}
    \right]
    \rightarrow\operatorname{Tr}(JV_{\vartheta_0}J^\top).
\]
Finally, a direct calculation gives
\[
    AV_{\vartheta_0}A^\top
    =\operatorname{diag}(\theta_0)-\theta_0\theta_0^\top
    =\Sigma(\theta_0).
\]
Combining this identity with \eqref{eq:functional-chain-rule}, we obtain
\[
    \operatorname{Tr}(JV_{\vartheta_0}J^\top)
    =\operatorname{Tr}\left[
        D\psi(\theta_0)\Sigma(\theta_0)D\psi(\theta_0)^\top
    \right],
\]
which proves the asserted limit.
\end{proof}

\begin{proof}[Proof of Lemma \ref{lem:uniform-wss-mult-comparison}]
For $\btheta\in\Theta_{\mathrm{reg}}$, define
\begin{equation}
\label{eq:delta-theta-definition}
\delta(\btheta)
:=
\min\left\{
    \theta_d,\,
    \theta_1-\theta_2,\,
    \ldots,\,
    \theta_{d-1}-\theta_d
\right\}.
\end{equation}
Since $\mathsf K$ is a compact subset of
$\Theta_{\mathrm{reg}}$, the function
$\btheta\mapsto\delta(\btheta)$ is continuous and strictly positive
on $\mathsf K$. Therefore
\begin{equation}
\label{eq:delta-K-definition}
\delta_{\mathsf K}
:=
\inf_{\btheta\in\mathsf K}\delta(\btheta)
=
\min_{\btheta\in\mathsf K}\delta(\btheta)
>0.
\end{equation}
Equivalently,
\[
\delta_{\mathsf K}
=
\min\left\{
    \inf_{\btheta\in\mathsf K}\theta_d,\,
    \min_{1\leq i<d}
    \inf_{\btheta\in\mathsf K}
        (\theta_i-\theta_{i+1})
\right\}.
\]

We apply the comparison developed in
\cite[Section~7.5]{Kahn&Guta}. In the notation used there, take the
central spectrum to be
$\boldsymbol\mu=\btheta,$ and 
take the local diagonal parameter to be
$\mathbf u=0,
$
and set the local-neighborhood exponent equal to
$\gamma=0.$
The condition $\alpha>\gamma+1/2$ then reduces to
$\alpha>1/2$, which holds by assumption. With this specialization, the block-weight distribution
$p_{\mathbf 0,n}$ of \cite{Kahn&Guta} is precisely the weak Schur
sampling law in our notation:
\[
    p_{\mathbf 0,n}(\lambda)
    =
    P_{\btheta}^{(n)}(\lambda),
    \qquad
    \lambda\in\mathbb Y_{n,d}.
\]
Likewise, their multinomial distribution $M_{n,\mathbf 0}$ is exactly
our $M_{\btheta}^{(n)}$, namely the law of
$\Mult_d(n,\btheta)$ on $\mathbb L_{n,d}$.
Their typical Young-diagram set $\Lambda_{n,\alpha}$, after setting
$\boldsymbol\mu=\btheta$ and $\mathbf u=0$, is the restriction to
$\mathbb Y_{n,d}$ of the set
$\mathcal T_{n,\alpha}(\btheta)$.

Equation~(7.36) of \cite{Kahn&Guta} states that
\[
\sup_{\|\mathbf u\|\leq n^\gamma}
\left\|
    p_{\mathbf u,n}-M_{n,\mathbf u}
\right\|_1
\leq
\frac{C}{\delta}
\left(
    n^{-1/2+\gamma}+n^{\alpha-1}
\right),
\]
where $C$ depends only on the dimension and $\delta$ is a positive
lower bound on the smallest component and the consecutive gaps of
the central spectrum.

Since $\mathbf u=0$ belongs to
$\{\|\mathbf u\|\leq1\}$, specializing this estimate to
$\gamma=0$, $\boldsymbol\mu=\btheta$, and $\mathbf u=0$ gives
\begin{equation}
\label{eq:pointwise-wss-mult-bound}
\left\|
    P_{\btheta}^{(n)}-M_{\btheta}^{(n)}
\right\|_1
\leq
\frac{C}{\delta(\btheta)}
\left(
    n^{-1/2}+n^{\alpha-1}
\right)
\end{equation}
for all sufficiently large $n$.

The proof of equation~(7.36) uses the condition
$n\delta(\btheta)>2dn^\alpha$
in the comparison on the typical set. This condition is equivalent
to
\begin{equation}
\label{eq:typical-set-sample-size}
n>
\left(
    \frac{2d}{\delta(\btheta)}
\right)^{1/(1-\alpha)}.
\end{equation}
The atypical-set estimate obtained from Lemma~6.2 and
equation~(7.34) of \cite{Kahn&Guta} is valid under the sufficient
sample-size condition
\begin{equation}
\label{eq:KG-full-sample-size}
n>
\left(
    \frac{2d}{\delta(\btheta)}
\right)^{1/(1-\alpha)}
+
(2d)^{1/(\alpha-1/2)}.
\end{equation}

Choose $n_{\mathsf K}$ so that
\begin{equation}
\label{eq:n-K-choice}
n_{\mathsf K}
>
\left(
    \frac{2d}{\delta_{\mathsf K}}
\right)^{1/(1-\alpha)}
+
(2d)^{1/(\alpha-1/2)}.
\end{equation}
Because
$\delta(\btheta)\geq\delta_{\mathsf K},
\btheta\in\mathsf K,
$
conditions \eqref{eq:typical-set-sample-size} and \eqref{eq:KG-full-sample-size} hold simultaneously for
every $\btheta\in\mathsf K$ and every
$n\geq n_{\mathsf K}$.

It follows from \eqref{eq:pointwise-wss-mult-bound} that
\begin{align*}
\sup_{\btheta\in\mathsf K}
\left\|
    P_{\btheta}^{(n)}-M_{\btheta}^{(n)}
\right\|_1
&\leq
\sup_{\btheta\in\mathsf K}
\frac{C}{\delta(\btheta)}
\left(
    n^{-1/2}+n^{\alpha-1}
\right)\\
&\leq
\frac{C}{\delta_{\mathsf K}}
\left(
    n^{-1/2}+n^{\alpha-1}
\right).
\end{align*}
Thus \eqref{eq:uniform-l1-comparison} holds after enlarging
$C_{\mathsf K}$, if necessary.

We now prove the tail estimate. In the notation of
\cite{Kahn&Guta}, Lemma~6.2 and equation~(7.34), specialized again
to $\boldsymbol\mu=\btheta$, $\mathbf u=0$, and $\gamma=0$, give
constants $C_1,C_2>0$, depending only on $d$, such that
\begin{equation}
\label{eq:KG-combined-tail}
P_{\btheta}^{(n)}
\bigl(\mathcal T_{n,\alpha}(\btheta)^c\bigr)
+
M_{\btheta}^{(n)}
\bigl(\mathcal T_{n,\alpha}(\btheta)^c\bigr)
\leq
C_1n^{d/2}
\exp\!\left\{
    -C_2n^{2\alpha-1}
\right\}
\end{equation}
whenever \eqref{eq:KG-full-sample-size} holds. Since
$n\geq n_{\mathsf K}$ makes that condition uniform over
$\btheta\in\mathsf K$, we obtain
\[
\sup_{\btheta\in\mathsf K}
\left\{
    P_{\btheta}^{(n)}
    \bigl(\mathcal T_{n,\alpha}(\btheta)^c\bigr)
    +
    M_{\btheta}^{(n)}
    \bigl(\mathcal T_{n,\alpha}(\btheta)^c\bigr)
\right\}
\leq
C_{\mathsf K}n^{d/2}
\exp\!\left\{
    -c_{\mathsf K}n^{2\alpha-1}
\right\},
\]
after taking, for example,
\[
    C_{\mathsf K}\geq C_1,
    \qquad
    0<c_{\mathsf K}\leq C_2.
\]
This proves \eqref{eq:uniform-typical-tail}.
\end{proof}

\begin{proof}[Proof of Lemma \ref{lem:uniform-wss-functional-risk}]
We first establish the required uniform second-moment expansion. Fix
\(\alpha\in(1/2,2/3)\), and for \(1\leq i,j\leq d\) define
\[
    F_{n,ij}^{\btheta}(k)
    :=n\left(\frac{k_i}{n}-\theta_i\right)
       \left(\frac{k_j}{n}-\theta_j\right).
\]
On \(\mathcal T_{n,\alpha}(\btheta)\),
\( |F_{n,ij}^{\btheta}(k)|\leq n^{2\alpha-1}\). Hence
Lemma~\ref{lem:uniform-wss-mult-comparison} gives
\[
\begin{aligned}
&\sup_{\btheta\in\mathsf K}
\left|
 \E_{P_{\btheta}^{(n)}}
 \bigl[F_{n,ij}^{\btheta}\mathbbm1_{\mathcal T_{n,\alpha}(\btheta)}\bigr]
 -\E_{M_{\btheta}^{(n)}}
 \bigl[F_{n,ij}^{\btheta}\mathbbm1_{\mathcal T_{n,\alpha}(\btheta)}\bigr]
\right|\\
&\qquad\leq
 C_{\mathsf K}n^{2\alpha-1}
 \bigl(n^{-1/2}+n^{\alpha-1}\bigr)\\
&\qquad=
 O\!\left(n^{2\alpha-3/2}+n^{3\alpha-2}\right)
 =o(1).
\end{aligned}
\]
For every \(k\in\mathbb L_{n,d}\),
\(|F_{n,ij}^{\btheta}(k)|\leq n\). Therefore
\eqref{eq:uniform-typical-tail} implies, for each of the two laws,
\[
    \sup_{\btheta\in\mathsf K}
    \E\!\left[
      |F_{n,ij}^{\btheta}|
      \mathbbm1_{\mathcal T_{n,\alpha}(\btheta)^c}
    \right]
    \leq
    C_{\mathsf K}n^{1+d/2}
      e^{-c_{\mathsf K}n^{2\alpha-1}}
    =o(1).
\]
Since
\[
    \E_{M_{\btheta}^{(n)}}F_{n,ij}^{\btheta}
    =\Sigma(\btheta)_{ij}
\]
exactly, and \(d\) is fixed, we obtain
\begin{equation}
    \sup_{\btheta\in\mathsf K}
    \left\|
      n\,\E_{\btheta}^{\WSS}
      \left[
        \left(\frac{\Lambda_n}{n}-\btheta\right)
        \left(\frac{\Lambda_n}{n}-\btheta\right)^\top
      \right]
      -\Sigma(\btheta)
    \right\|_{\mathrm{op}}
    \rightarrow0.
\label{eq:uniform-wss-covariance}
\end{equation}

We now pass from second moments to the risk of \(\psi\). Put
\[
    H_{n,\btheta}:=\frac{\Lambda_n}{n}-\btheta,
    \qquad
    D_{\btheta}:=D\psi(\btheta).
\]
Since \(D\psi\) is uniformly continuous on the compact simplex, its
modulus of continuity
\[
    \omega(t)
    :=\sup_{\substack{x,y\in\Theta_{\downarrow}\\
                      \|x-y\|\leq t}}
      \|D\psi(x)-D\psi(y)\|_{\mathrm{op}}
\]
satisfies \(\omega(t)\downarrow0\) as \(t\downarrow0\). Taylor's formula yields

\begin{align}
    \psi(\btheta+h)-\psi(\btheta)
      &=D_{\btheta}h+r_{\btheta}(h),\nonumber\\
    \|r_{\btheta}(h)\|
      &\leq\omega(\|h\|)\|h\|\label{eq:uniform-taylor-remainder}
\end{align}
whenever \(\btheta,\btheta+h\in\Theta_{\downarrow}\). We also have
\(\|r_{\btheta}(h)\|\leq2L_\psi\|h\|\).

Let
\[
A_{n,\btheta}
:=
n\,\E_{\btheta}^{\WSS}
[H_{n,\btheta}H_{n,\btheta}^\top]
-\Sigma(\btheta).
\]
Then
\[
n\,\E_{\btheta}^{\WSS}
\|D_{\btheta}H_{n,\btheta}\|^2
-\cI_\psi(\btheta)
=
\Tr\!\left(
D_{\btheta}A_{n,\btheta}D_{\btheta}^\top
\right).
\]
Since
$\|D_{\btheta}\|\le L_\psi$, 
by \eqref{eq:uniform-wss-covariance}, we obtain
\begin{align}
&\sup_{\btheta\in\mathsf K}
\left|
 n\,\E_{\btheta}^{\WSS}
   \|D_{\btheta}H_{n,\btheta}\|^2
 -\cI_\psi(\btheta)
\right|\nonumber\\
&\qquad\leq
 qL_\psi^2
 \sup_{\btheta\in\mathsf K}
 \left\|
 n\,\E_{\btheta}^{\WSS}
 [H_{n,\btheta}H_{n,\btheta}^\top]
 -\Sigma(\btheta)
 \right\|_{\mathrm{op}}
 \rightarrow 0.\label{eq:uniform-linear-risk}
\end{align}
Taking traces in \eqref{eq:uniform-wss-covariance} also gives
\begin{equation}
    \sup_n\sup_{\btheta\in\mathsf K}
      n\,\E_{\btheta}^{\WSS}\|H_{n,\btheta}\|^2<\infty.
\label{eq:uniform-second-moment-bound}
\end{equation}
Fix \(\delta>0\). From \eqref{eq:uniform-taylor-remainder},
\begin{equation}
n\,\E_{\btheta}^{\WSS}
  \|r_{\btheta}(H_{n,\btheta})\|^2
\leq
 \omega(\delta)^2n\,\E_{\btheta}^{\WSS}
   \|H_{n,\btheta}\|^2
 +4L_\psi^2n\,\E_{\btheta}^{\WSS}
 \left[
   \|H_{n,\btheta}\|^2
   \mathbbm1_{\{\|H_{n,\btheta}\|>\delta\}}
 \right]. \label{eq:uniform-remainder-split}
\end{equation}

Choose any \(\alpha\in(1/2,1)\). On
\(\mathcal T_{n,\alpha}(\btheta)\),
\[
    \|H_{n,\btheta}\|\leq\sqrt d\,n^{\alpha-1}.
\]
Thus, for all sufficiently large \(n\), uniformly over
\(\btheta\in\mathsf K\),
\[
    \{\|H_{n,\btheta}\|>\delta\}
    \subseteq\mathcal T_{n,\alpha}(\btheta)^c.
\]
Since the diameter of the simplex is at most \(\sqrt2\),
Lemma~\ref{lem:uniform-wss-mult-comparison} gives

\begin{align}
&\sup_{\btheta\in\mathsf K}
 n\,\E_{\btheta}^{\WSS}
 \left[
   \|H_{n,\btheta}\|^2
   \mathbbm1_{\{\|H_{n,\btheta}\|>\delta\}}
 \right]\nonumber\\
&\qquad\leq
 2C_{\mathsf K}n^{1+d/2}
 e^{-c_{\mathsf K}n^{2\alpha-1}}
 \rightarrow0. \label{eq:uniform-large-remainder}
\end{align}

Combining \eqref{eq:uniform-second-moment-bound},
\eqref{eq:uniform-remainder-split}, and
\eqref{eq:uniform-large-remainder}, and then letting
\(\delta\downarrow0\), yields
\begin{equation}
    \sup_{\btheta\in\mathsf K}
    n\,\E_{\btheta}^{\WSS}
      \|r_{\btheta}(H_{n,\btheta})\|^2
    \rightarrow0.
\label{eq:uniform-remainder-vanishes}
\end{equation}
Finally, Cauchy-Schwarz, \eqref{eq:uniform-linear-risk}, and
\eqref{eq:uniform-remainder-vanishes} imply
\[
    \sup_{\btheta\in\mathsf K}
    n\left|
      \E_{\btheta}^{\WSS}
      \langle D_{\btheta}H_{n,\btheta},
              r_{\btheta}(H_{n,\btheta})\rangle
    \right|
    \rightarrow0.
\]
Expanding the squared norm completes the proof.
\end{proof}

\begin{proof}[Proof of Lemma~\ref{lem:concentration}]
We use the passive two-mode interactions employed in the construction of
the concentrating operator by Kumagai and
Hayashi~\cite[Section~2]{KumagaiHayashi2013}.

For $j=1,\ldots,n$, let
\[
    a_j
    :=
    I^{\otimes(j-1)}
    \otimes a
    \otimes I^{\otimes(n-j)}
\]
denote the annihilation operator acting on the $j$th mode, and let
$a_j^\dagger$ denote its adjoint.  Define the total photon-number operator
\[
    \widehat N_{\mathrm{tot}}
    :=
    \sum_{\ell=1}^n a_\ell^\dagger a_\ell .
\]
For $r\geq0$, let
\[
    \mathcal K_r
    :=
    \operatorname{span}
    \left\{
        |k_1,\ldots,k_n\rangle:
        \sum_{\ell=1}^n k_\ell=r
    \right\},
\]
so that
\[
    \cH^{\otimes n}
    =
    \bigoplus_{r=0}^\infty\mathcal K_r .
\]
The finite-particle subspace is
\[
    \mathscr D_{\mathrm{fin}}
    :=
    \left\{
        \sum_{r=0}^{R}\psi_r:
        R<\infty,\quad
        \psi_r\in\mathcal K_r
    \right\},
\]
and, for $R\geq0$, we set
\[
    \cH_{\leq R}
    :=
    \bigoplus_{r=0}^{R}\mathcal K_r .
\]
Each $\mathcal K_r$ and $\cH_{\leq R}$ is finite dimensional.

For $j=1,\ldots,n-1$, consider on
$\mathscr D_{\mathrm{fin}}$ the two-mode Hamiltonian
\begin{equation}
    H_{j,j+1}
    :=
    i\left(
        a_{j+1}^\dagger a_j
        -
        a_j^\dagger a_{j+1}
    \right).
    \label{eq:beam-splitter-Hamiltonian}
\end{equation}
The operator $H_{j,j+1}$ preserves each eigenspace
$\mathcal K_r$ of $\widehat N_{\mathrm{tot}}$, and its restriction
$H_{j,j+1}^{[r]}$ to $\mathcal K_r$ is a finite-dimensional
self-adjoint operator. Define the self-adjoint direct sum
\[
H_{j,j+1}
:=
\bigoplus_{r=0}^{\infty}H_{j,j+1}^{[r]}.
\]
This direct sum gives the self-adjoint realization of
\eqref{eq:beam-splitter-Hamiltonian}, and
$\mathscr D_{\mathrm{fin}}$ is a core for it.
Define 
\begin{equation}
    V_j(t)
    :=
    e^{itH_{j,j+1}},
    \qquad t\in\mathbb R .
    \label{eq:beam-splitter-unitary}
\end{equation}
Since $H_{j,j+1}$ is block diagonal with respect to the decomposition
$\bigoplus_r\mathcal K_r$, each $\mathcal K_r$, each
$\cH_{\leq R}$, and $\mathscr D_{\mathrm{fin}}$ are invariant under
$V_j(t)$.

We first determine the action of $V_j(t)$ on the two mode operators.
The canonical commutation relations give, on
$\mathscr D_{\mathrm{fin}}$,
\begin{equation}
    [H_{j,j+1},a_j]
    =
    ia_{j+1},
    \qquad
    [H_{j,j+1},a_{j+1}]
    =
    -ia_j.
    \label{eq:beam-splitter-commutators}
\end{equation}

Next we study conjugation of the operators $a_j$ and $a_{j+1}$ by $V_j(t)$. To justify the conjugation formulas without differentiating unbounded
operator-valued maps, fix $R\geq0$ and work on the finite-dimensional
space $\cH_{\leq R}$.  Write
\[
    H^{(R)}
    :=
    H_{j,j+1}|_{\cH_{\leq R}},
    \qquad
    a_j^{(R)}
    :=
    a_j|_{\cH_{\leq R}},
    \qquad
    a_{j+1}^{(R)}
    :=
    a_{j+1}|_{\cH_{\leq R}},
\]
and
\[
    V^{(R)}(t)
    :=
    V_j(t)|_{\cH_{\leq R}}
    =
    e^{itH^{(R)}} .
\]
Since annihilation lowers total photon number, $a_j$ and $a_{j+1}$ map
$\cH_{\leq R}$ into itself.  Thus all the operators above are bounded
operators on the finite-dimensional space $\cH_{\leq R}$, and
\eqref{eq:beam-splitter-commutators} restricts to
\[
    [H^{(R)},a_j^{(R)}]
    =
    ia_{j+1}^{(R)},
    \qquad
    [H^{(R)},a_{j+1}^{(R)}]
    =
    -ia_j^{(R)}.
\]

Define
\[
    A_R(t)
    :=
    V^{(R)}(t)a_j^{(R)}V^{(R)}(t)^\dagger,
    \qquad
    B_R(t)
    :=
    V^{(R)}(t)a_{j+1}^{(R)}V^{(R)}(t)^\dagger .
\]
These maps are differentiable in operator norm.  Since
$H^{(R)}$ commutes with $V^{(R)}(t)$,
\begin{align*}
    A_R'(t)
    &=
    iV^{(R)}(t)
    [H^{(R)},a_j^{(R)}]
    V^{(R)}(t)^\dagger
    =
    -B_R(t),\\
    B_R'(t)
    &=
    iV^{(R)}(t)
    [H^{(R)},a_{j+1}^{(R)}]
    V^{(R)}(t)^\dagger
    =
    A_R(t).
\end{align*}

Fix $u,v\in\cH_{\leq R}$ and set
\[
    f(t)
    :=
    \langle u,A_R(t)v\rangle,
    \qquad
    g(t)
    :=
    \langle u,B_R(t)v\rangle .
\]
Then
\[
    f'(t)=-g(t),
    \qquad
    g'(t)=f(t),
\]
with
\[
    f(0)=\langle u,a_j^{(R)}v\rangle,
    \qquad
    g(0)=\langle u,a_{j+1}^{(R)}v\rangle.
\]
By uniqueness of the solution of this scalar linear system,
\begin{align*}
    f(t)
    &=
    \langle u,a_j^{(R)}v\rangle\cos t
    -
    \langle u,a_{j+1}^{(R)}v\rangle\sin t,\\
    g(t)
    &=
    \langle u,a_j^{(R)}v\rangle\sin t
    +
    \langle u,a_{j+1}^{(R)}v\rangle\cos t .
\end{align*}
Since $u$ and $v$ are arbitrary,
\[
    V^{(R)}(t)a_j^{(R)}V^{(R)}(t)^\dagger
    =
    a_j^{(R)}\cos t-a_{j+1}^{(R)}\sin t
\]
and
\[
    V^{(R)}(t)a_{j+1}^{(R)}V^{(R)}(t)^\dagger
    =
    a_j^{(R)}\sin t+a_{j+1}^{(R)}\cos t .
\]

Letting $R$ vary, we obtain on $\mathscr D_{\mathrm{fin}}$
\begin{align}
    V_j(t)a_jV_j(t)^\dagger
    &=
    a_j\cos t-a_{j+1}\sin t,
    \label{eq:mode-rotation-first}\\
    V_j(t)a_{j+1}V_j(t)^\dagger
    &=
    a_j\sin t+a_{j+1}\cos t.
    \label{eq:mode-rotation-second}
\end{align}
Since $\mathscr D_{\mathrm{fin}}$ is invariant under
$V_j(t)$ and $V_j(t)^\dagger$, and the creation and annihilation
operators map $\mathscr D_{\mathrm{fin}}$ into itself, the corresponding
creation-operator identities can be verified directly on this common
domain.

Indeed, let $\phi,\psi\in\mathscr D_{\mathrm{fin}}$. Then
\begin{align*}
\left\langle
    \phi,
    V_j(t)a_j^\dagger V_j(t)^\dagger\psi
\right\rangle
&=
\left\langle
    V_j(t)a_jV_j(t)^\dagger\phi,
    \psi
\right\rangle\\
&=
\left\langle
    (a_j\cos t-a_{j+1}\sin t)\phi,
    \psi
\right\rangle\\
&=
\left\langle
    \phi,
    (a_j^\dagger\cos t-a_{j+1}^\dagger\sin t)\psi
\right\rangle,
\end{align*}
where we used \eqref{eq:mode-rotation-first}. 
Since $\mathscr D_{\mathrm{fin}}$ is dense, this proves the first
creation-operator identity on $\mathscr D_{\mathrm{fin}}$.
Applying the same argument to \eqref{eq:mode-rotation-second} gives
\begin{align}
    V_j(t)a_j^\dagger V_j(t)^\dagger
    &=
    a_j^\dagger\cos t-a_{j+1}^\dagger\sin t,
    \label{eq:creation-rotation-first}\\
    V_j(t)a_{j+1}^\dagger V_j(t)^\dagger
    &=
    a_j^\dagger\sin t+a_{j+1}^\dagger\cos t.
    \label{eq:creation-rotation-second}
\end{align}
Both identities are understood on $\mathscr D_{\mathrm{fin}}$.

We next pass to the displacement operators.  For
$\alpha,\beta\in\bbC$, define on $\mathscr D_{\mathrm{fin}}$
\[
    X_{\alpha,\beta}
    :=
    \alpha a_j^\dagger-\overline{\alpha}a_j
    +
    \beta a_{j+1}^\dagger-\overline{\beta}a_{j+1}.
\]
Finite-particle vectors in $\mathscr D_{\mathrm{fin}}$ are analytic for
these field operators. Since $iX_{\alpha,\beta}$ is symmetric on
$\mathscr D_{\mathrm{fin}}$, Nelson's analytic-vector theorem
\cite[Theorem~X.39]{ReedSimonII} implies that
$X_{\alpha,\beta}$ is essentially skew-adjoint on
$\mathscr D_{\mathrm{fin}}$. Moreover,
\[
    e^{\overline{X_{\alpha,\beta}}}
    =
    D_j(\alpha)D_{j+1}(\beta),
\]
where $D_j(\alpha)$ denotes the displacement operator acting on the
$j$th mode.

By
\eqref{eq:mode-rotation-first}-\eqref{eq:creation-rotation-second},
\[
    V_j(t)X_{\alpha,\beta}V_j(t)^\dagger
    =
    X_{\alpha',\beta'}
\]
on $\mathscr D_{\mathrm{fin}}$, where
\[
    \alpha'
    =
    \alpha\cos t+\beta\sin t,
    \qquad
    \beta'
    =
    -\alpha\sin t+\beta\cos t.
\]
Since $\mathscr D_{\mathrm{fin}}$ is a common core for the corresponding
essentially skew-adjoint operators, their closures satisfy
\[
    V_j(t)\overline{X_{\alpha,\beta}}V_j(t)^\dagger
    =
    \overline{X_{\alpha',\beta'}}.
\]
Unitary covariance of the exponential therefore yields
\begin{equation}
\begin{aligned}
&
V_j(t)
D_j(\alpha)D_{j+1}(\beta)
V_j(t)^\dagger
\\
&\qquad=
D_j(\alpha\cos t+\beta\sin t)
D_{j+1}(-\alpha\sin t+\beta\cos t).
\end{aligned}
\label{eq:two-mode-displacement-rotation}
\end{equation}

For $j=1,\ldots,n-1$, choose $t_j\in[0,\pi/2)$ so that
$\tan t_j=\sqrt{n-j}.
    $
Then the following holds 
\begin{align}
    \begin{pmatrix}
        \cos t_j& \sin t_j\\
        -\sin t_j& \cos t_j
    \end{pmatrix}\begin{pmatrix}
        w\\
        w\sqrt{n-j} 
    \end{pmatrix}= \begin{pmatrix}
        w \sqrt{n-j+1} \\
        0 
    \end{pmatrix} \label{exp_rot}
\end{align}
i.e. the rotation sends
$\left(
        w,\sqrt{n-j}\,w
    \right)
$
to
$\left(
        \sqrt{n-j+1}\,w,0
    \right).
$

Define
\begin{equation}
    U_n
    :=
    V_1(t_1)V_2(t_2)\cdots V_{n-1}(t_{n-1}).
    \label{eq:concentration-unitary}
\end{equation}
Since the rightmost factor acts first under conjugation by $U_n$, a
backward induction gives
\begin{align}
&
V_j(t_j)\cdots V_{n-1}(t_{n-1})
D(w)^{\otimes n}
V_{n-1}(t_{n-1})^\dagger\cdots V_j(t_j)^\dagger
\nonumber\\
&\qquad=
D(w)^{\otimes(j-1)}
\otimes
D\!\left(
    \sqrt{n-j+1}\,w
\right)
\otimes
I^{\otimes(n-j)}.
\label{eq:partial-concentration}
\end{align}
Indeed for $j=n-1$, using \eqref{eq:two-mode-displacement-rotation} and \eqref{exp_rot} this follows from the transformation
$(w,w)\mapsto(\sqrt2\,w,0)$.  If it holds at $j+1$, then modes $j$
and $j+1$ carry displacements
$(w,\sqrt{n-j}\,w)$; the choice of $t_j$ transforms these into
$(\sqrt{n-j+1}\,w,0)$, proving the induction step.

Taking $j=1$ in \eqref{eq:partial-concentration} gives
\begin{equation}
    U_nD(w)^{\otimes n}U_n^\dagger
    =
    D(\sqrt n\,w)\otimes I^{\otimes(n-1)},
    \label{eq:appendix-concentrated-action}
\end{equation}
which is \eqref{eq:concentrated-action}.

It remains to prove the state identity.  Each $H_{j,j+1}$ preserves total
photon number, equivalently,
\[
    [H_{j,j+1},\widehat N_{\mathrm{tot}}]=0
\]
on $\mathscr D_{\mathrm{fin}}$, where 
\[\widehat N_{\mathrm{tot}} =\sum_{j=1}^n a_j^\dagger a_j.\]
This can also be read directly from
\eqref{eq:beam-splitter-Hamiltonian}: each term removes one excitation
from one of the two modes and creates one in the other.
It follows that
$V_j(t)$ commutes with every bounded Borel function of
$\widehat N_{\mathrm{tot}}$.

Set
$
    q_N:=\frac{N}{N+1}.
$
Since
\[
    \phi_N
    =
    \frac{1}{N+1}
    q_N^{\,a^\dagger a},
\]
we have
\begin{equation}
    \phi_N^{\otimes n}
    =
    \frac{1}{(N+1)^n}
    q_N^{\,\widehat N_{\mathrm{tot}}}.
    \label{eq:thermal-product-number-form}
\end{equation}
Consequently,
\[
    V_j(t)\phi_N^{\otimes n}V_j(t)^\dagger
    =
    \phi_N^{\otimes n}
\]
for every $j$ and $t$.  Applying this successively to the factors defining
$U_n$ gives
\begin{equation}
    U_n\phi_N^{\otimes n}U_n^\dagger
    =
    \phi_N^{\otimes n}.
    \label{eq:thermal-product-invariance}
\end{equation}

Finally,
\[
    \rho_{z,N}^{\otimes n}
    =
    D(z)^{\otimes n}
    \phi_N^{\otimes n}
    \bigl(D(z)^{\otimes n}\bigr)^\dagger.
\]
Using \eqref{eq:concentrated-action} and
\eqref{eq:thermal-product-invariance},
\begin{align*}
    U_n\rho_{z,N}^{\otimes n}U_n^\dagger
    &=
    \bigl(
        D(\sqrt n\,z)\otimes I^{\otimes(n-1)}
    \bigr)
    \phi_N^{\otimes n}
    \bigl(
        D(\sqrt n\,z)^\dagger\otimes I^{\otimes(n-1)}
    \bigr)
\\
    &=
    \bigl[
        D(\sqrt n\,z)
        \phi_N
        D(\sqrt n\,z)^\dagger
    \bigr]
    \otimes
    \phi_N^{\otimes(n-1)}
\\
    &=
    \rho_{\sqrt n\,z,N}
    \otimes
    \phi_N^{\otimes(n-1)}.
\end{align*}
This proves \eqref{eq:concentrated-state}.
\end{proof}

\begin{proof}[Proof of Lemma \ref{seq_conv_POVM}]
Since \(\mathsf A\) is
compact metric, \(C(\mathsf A)\) is separable in the supremum norm. Since
\(\mathcal K_n\) is separable, the trace-class space
\(\mathcal T_1(\mathcal K_n)\) is separable in the trace norm. Choose
countable dense subsets
\begin{equation}
    \{f_r:r\geq1\}
    \subset
    \left\{
        f\in C(\mathsf A):
        \|f\|_\infty\leq1
    \right\}
    \label{eq:dense-continuous-functions}
\end{equation}
and
\begin{equation}
    \{X_s:s\geq1\}
    \subset
    \left\{
        X\in\mathcal T_1(\mathcal K_n):
        \|X\|_1\leq1
    \right\}.
    \label{eq:dense-trace-class-operators}
\end{equation}

Let \((D_j)\) be any sequence of POVMs. For every \(j,r,s\),
\begin{align}
    \left|
        \Tr\!\left[
            X_s\Phi_{D_j}(f_r)
        \right]
    \right|
    &\leq
    \|X_s\|_1
    \|\Phi_{D_j}(f_r)\|
    \nonumber\\
    &\leq
    \|X_s\|_1
    \|f_r\|_\infty
    \leq1.
    \label{eq:bounded-matrix-elements}
\end{align}
A diagonal-subsequence argument therefore gives a subsequence, again denoted
by \((D_j)\), for which
$\Tr\!\left[
        X_s\Phi_{D_j}(f_r)
    \right]
    $
converges for every pair \(r,s\).

For arbitrary \(X\in\mathcal T_1(\mathcal K_n)\) and
\(f\in C(\mathsf A)\), we have the uniform bound
\begin{equation}
    \left|
        \Tr\!\left[
            X\Phi_{D_j}(f)
        \right]
    \right|
    \leq
    \|X\|_1\|f\|_\infty.
    \label{eq:uniform-pairing-bound}
\end{equation}
Using the density of the families in
\eqref{eq:dense-continuous-functions} and
\eqref{eq:dense-trace-class-operators}, together with
\eqref{eq:uniform-pairing-bound},$\Tr\!\left[
        X\Phi_{D_j}(f)
    \right]
    $ converges for every 
$
    X\in\mathcal T_1(\mathcal K_n),
    f\in C(\mathsf A).
$

For each \(f\in C(\mathsf A)\), the resulting limit is a bounded linear
functional of \(X\in\mathcal T_1(\mathcal K_n)\). Since
$\mathcal L(\mathcal K_n)
    =
    \mathcal T_1(\mathcal K_n)^*,
$
there is a unique operator \(\Phi(f)\in\mathcal L(\mathcal K_n)\) such that
\begin{equation}
    \Tr[X\Phi(f)]
    =
    \lim_{j\to\infty}
    \Tr\!\left[
        X\Phi_{D_j}(f)
    \right]
    \label{eq:limiting-positive-map}
\end{equation}
for every \(X\in\mathcal T_1(\mathcal K_n)\).

Linearity of \(\Phi\) follows directly by passing to the limit in
\eqref{eq:limiting-positive-map}. If \(f\geq0\), then
$\Phi_{D_j}(f)\geq0$ for every \(j\). Hence, for every \(\psi\in\mathcal K_n\),
\begin{align}
    \langle\psi,\Phi(f)\psi\rangle
    &=
    \lim_{j\to\infty}
    \langle\psi,\Phi_{D_j}(f)\psi\rangle
    \geq0.
    \label{eq:positivity-of-limit-map}
\end{align}
Thus \(\Phi(f)\geq0\). Moreover,
\begin{equation}
    \Phi(1)
    =
    \lim_{j\to\infty}\Phi_{D_j}(1)
    =
    I
    \label{eq:unitality-of-limit-map}
\end{equation}
in the ultraweak sense. Therefore \(\Phi\) is positive and unital.

By the operator-valued Riesz-Markov-Kakutani representation theorem (cf. Theorem 4.4 of \cite{Quantum_Measurements_2016}), there is a unique POVM
\(\overline D\) on \(\mathsf A\) such that
\begin{equation}
    \Phi(f)
    =
    \Phi_{\overline D}(f)
    =
    \int_{\mathsf A}f(a)\,\overline D(da).
    \label{eq:limiting-povm}
\end{equation}
We have therefore proved that every sequence of POVMs contains a subsequence
that converges in the topology \eqref{eq:povm-topology}.

Recall that for \(D\) an arbitrary decision POVM on \(\mathsf A\), $D_j$ is defined in \eqref{eq:asymptotically-averaged-povm} as
\begin{equation}
    D_j(B)
    :=
    \int_{\bbC}
        V_w^\dagger D(B)V_w\,
        \nu_j(dw),
    \qquad
    B\in\mathcal A,
\end{equation}
where the integral is understood ultraweakly. Equivalently,
\begin{equation}
    \Phi_{D_j}(f)
    =
    \int_{\bbC}
        V_w^\dagger\Phi_D(f)V_w\,
        \nu_j(dw),
    \qquad
    f\in C(\mathsf A).
    \label{eq:asymptotically-averaged-positive-map}
\end{equation}

We verify that \(D_j\) is a POVM. Positivity follows because
$V_w^\dagger D(B)V_w\geq0$
for every \(w\). Moreover,
\begin{align}
    D_j(\mathsf A)
    &=
    \int_{\bbC}
        V_w^\dagger D(\mathsf A)V_w\,
        \nu_j(dw)
    \nonumber\\
    &=
    \int_{\bbC}
        V_w^\dagger IV_w\,
        \nu_j(dw)
    =
    I.
    \label{eq:averaged-povm-normalization}
\end{align}

Let \(B_1,B_2,\ldots\) be pairwise disjoint Borel subsets of
\(\mathsf A\), and let \(X\geq0\) be trace class. Then
\begin{align}
&
    \Tr\!\left[
        X D_j\!\left(\bigcup_{\ell=1}^\infty B_\ell\right)
    \right]
    \nonumber\\
&\quad=
    \int_{\bbC}
        \Tr\!\left[
            X
            V_w^\dagger
            D\!\left(\bigcup_{\ell=1}^\infty B_\ell\right)
            V_w
        \right]
        \nu_j(dw)
    \nonumber\\
&\quad=
    \int_{\bbC}
        \sum_{\ell=1}^\infty
        \Tr\!\left[
            X V_w^\dagger D(B_\ell)V_w
        \right]
        \nu_j(dw)
    \nonumber\\
&\quad=
    \sum_{\ell=1}^\infty
        \int_{\bbC}
            \Tr\!\left[
                X V_w^\dagger D(B_\ell)V_w
            \right]
            \nu_j(dw)
    \nonumber\\
&\quad=
    \sum_{\ell=1}^\infty
        \Tr[X D_j(B_\ell)].
    \label{eq:averaged-povm-countable-additivity}
\end{align}
The interchange of the sum and integral follows from monotone convergence.
Since positive trace-class operators separate bounded operators,
\eqref{eq:averaged-povm-countable-additivity} proves countable additivity in
the weak operator topology. Thus \(D_j\) is a POVM.

By the compactness established above, there exists a subsequence
\((j_\ell)\) and a POVM \(\overline D\) such that
\begin{equation}
    \Tr\!\left[
        X\Phi_{D_{j_\ell}}(f)
    \right]
    \rightarrow
    \Tr\!\left[
        X\Phi_{\overline D}(f)
    \right]
    \label{eq:averaged-povm-convergence}
\end{equation}
for every \(X\in\mathcal T_1(\mathcal K_n)\) and every
\(f\in C(\mathsf A)\).

Next we prove invariance of $\bar D$. Fix \(v\in\bbC\), a density operator \(\sigma\) on \(\mathcal K_n\), and a
real-valued function \(f\in C(\mathsf A)\) satisfying
$0\leq f\leq1$. Define the scalar function
\begin{equation}
    q_{\sigma,f}(w)
    :=
    \Tr\!\left[
        \sigma
        V_w^\dagger\Phi_D(f)V_w
    \right],
    \qquad
    w\in\bbC.
    \label{eq:scalar-povm-orbit-function}
\end{equation}
Since
\[
    0\leq\Phi_D(f)\leq I,
\]
we have
\begin{equation}
    0\leq q_{\sigma,f}(w)\leq1.
    \label{eq:scalar-orbit-function-range}
\end{equation}
Strong continuity of \(w\mapsto V_w\) implies that
\(q_{\sigma,f}\) is Borel measurable.

By \eqref{eq:asymptotically-averaged-positive-map},
\begin{equation}
    \Tr\!\left[
        \sigma\Phi_{D_j}(f)
    \right]
    =
    \int_{\bbC}
        q_{\sigma,f}(w)\,
        \nu_j(dw).
    \label{eq:unshifted-scalar-average}
\end{equation}
On the other hand,
\begin{align}
\Tr\!\left[
        \sigma
        V_v^\dagger\Phi_{D_j}(f)V_v
    \right]
&=
    \int_{\bbC}
        \Tr\!\left[
            \sigma
            V_v^\dagger V_w^\dagger
            \Phi_D(f)
            V_wV_v
        \right]
        \nu_j(dw)\nonumber\\
&=\int_{\bbC}
        \Tr\!\left[
            \sigma
            V_{w+v}^\dagger\Phi_D(f)V_{w+v}
        \right]
        \nu_j(dw)\nonumber\\
&=
    \int_{\bbC}
        q_{\sigma,f}(w+v)\,
        \nu_j(dw).
    \label{eq:translated-scalar-average}
\end{align}

Applying the bounded-function asymptotic-invariance result
\eqref{eq:bounded-function-asymptotic-invariance} to
\(q_{\sigma,f}\), we obtain
\begin{equation}
    \Tr\!\left[
        \sigma
        V_v^\dagger\Phi_{D_j}(f)V_v
    \right]
    -
    \Tr\!\left[
        \sigma\Phi_{D_j}(f)
    \right]
    \rightarrow0.
    \label{eq:approximate-invariance-expectation}
\end{equation}

We now pass to the convergent subsequence \((j_\ell)\). From
\eqref{eq:averaged-povm-convergence},
\begin{equation}
    \Tr\!\left[
        \sigma\Phi_{D_{j_\ell}}(f)
    \right]
    \rightarrow
    \Tr\!\left[
        \sigma\Phi_{\overline D}(f)
    \right].
    \label{eq:unshifted-povm-limit}
\end{equation}
Furthermore, using the cyclicity of trace
\begin{equation}
\Tr\!\left[
        \sigma
        V_v^\dagger
        \Phi_{D_{j_\ell}}(f)
        V_v
    \right]
\rightarrow
    \Tr\!\left[
        \sigma
        V_v^\dagger
        \Phi_{\overline D}(f)
        V_v
    \right],
    \label{eq:translated-povm-limit}
\end{equation}
because \(V_v\sigma V_v^\dagger\) is trace class.

Combining
\eqref{eq:approximate-invariance-expectation},
\eqref{eq:unshifted-povm-limit}, and
\eqref{eq:translated-povm-limit}, we obtain
\begin{equation}
    \Tr\!\left[
        \sigma
        \left\{
            V_v^\dagger
            \Phi_{\overline D}(f)
            V_v
            -
            \Phi_{\overline D}(f)
        \right\}
    \right]
    =
    0
    \label{eq:invariance-against-density-operators}
\end{equation}
for every density operator \(\sigma\) on \(\mathcal K_n\).

Since density operators separate bounded self-adjoint operators,
the preceding identity implies
\begin{equation}
    V_v^\dagger
    \Phi_{\overline D}(f)
    V_v
    =
    \Phi_{\overline D}(f).
    \label{eq:operator-invariance-positive-functions}
\end{equation}
An affine rescaling extends
\eqref{eq:operator-invariance-positive-functions} from functions satisfying
\(0\leq f\leq1\) to all real-valued functions in \(C(\mathsf A)\). By
decomposing a complex-valued function into its real and imaginary parts, the
same identity holds for every \(f\in C(\mathsf A)\):
\begin{equation}
    V_v^\dagger
    \Phi_{\overline D}(f)
    V_v
    =
    \Phi_{\overline D}(f).
    \label{eq:operator-invariance-all-functions}
\end{equation}

Let \(v\cdot\overline D\) be the conjugated POVM defined by
$
    (v\cdot\overline D)(B)
    :=
    V_v^\dagger\overline D(B)V_v.
$
Its associated positive unital map satisfies
\begin{equation}
    \Phi_{v\cdot\overline D}(f)
    =
    V_v^\dagger
    \Phi_{\overline D}(f)
    V_v
    =
    \Phi_{\overline D}(f)
    \label{eq:equality-positive-unital-maps}
\end{equation}
for every \(f\in C(\mathsf A)\). Uniqueness in the operator-valued Riesz-Markov-Kakutani representation theorem (cf. Theorem 4.4 of \cite{Quantum_Measurements_2016}) now implies
\begin{equation}
    v\cdot\overline D=\overline D.
    \label{eq:limiting-povm-invariant-one-v}
\end{equation}
Since \(v\in\bbC\) was arbitrary,
$\overline D\in\cM_{\mathrm{inv}}^{(n)}.$
\end{proof}

\begin{proof}[Proof of Lemma \ref{lem:thermal-uniform-plug-in-risk}]
The geometric family has moments of every order.
Since \(K\) is a compact subset of \((0,\infty)\), its centered moments of any fixed order are
uniformly bounded over \(N\in K\). In particular,
\begin{equation}
    \sup_{N\in K}
    \E_N|X_1-N|^4
    <
    \infty.
    \label{eq:uniform-geometric-fourth-moment}
\end{equation}

Write
\[
    Y_j:=X_j-N.
\]
Then \(\E_NY_j=0\), and independence gives
\begin{align}
    \E_N
    \left(
        \sum_{j=1}^mY_j
    \right)^4
    &=
    m\E_NY_1^4
    +
    3m(m-1)
    \{\E_NY_1^2\}^2.
    \label{eq:fourth-moment-independent-sum}
\end{align}
It follows that
\begin{equation}
    \sup_{N\in K}
    \E_N|\widehat N_m-N|^4
    =
    O(m^{-2}).
    \label{eq:sample-mean-fourth-moment}
\end{equation}

By \eqref{eq:natural-thermal-estimator-variance},
\begin{equation}
    m\E_N(\widehat N_m-N)^2
    =
    N(N+1)
    \label{eq:thermal-exact-second-moment}
\end{equation}
for every \(m\) and \(N>0\).

Taylor's formula gives
\begin{equation}
    \psi(\widehat N_m)-\psi(N)
    =
    \psi'(N)(\widehat N_m-N)
    +
    r_{m,N},
    \label{eq:thermal-functional-taylor}
\end{equation}
where Assumption~\ref{ass:thermal-smooth-functional} implies
\begin{equation}
    |r_{m,N}|
    \leq
    \frac12
    \|\psi''\|_\infty
    |\widehat N_m-N|^2.
    \label{eq:thermal-taylor-remainder-bound}
\end{equation}
Therefore, by \eqref{eq:sample-mean-fourth-moment},
\begin{align}
    \sup_{N\in K}
    m\E_Nr_{m,N}^2
    &\leq
    \frac14
    \|\psi''\|_\infty^2
    \sup_{N\in K}
    m\E_N|\widehat N_m-N|^4
    \rightarrow0.
    \label{eq:thermal-taylor-remainder-L2}
\end{align}

By Cauchy-Schwarz,
\begin{align}
&
    m
    \left|
        \E_N
        \left[
            \psi'(N)
            (\widehat N_m-N)
            r_{m,N}
        \right]
    \right|
    \nonumber\\
&\quad\leq
    |\psi'(N)|
    \left\{
        m\E_N(\widehat N_m-N)^2
    \right\}^{1/2}
    \left\{
        m\E_Nr_{m,N}^2
    \right\}^{1/2}.
    \label{eq:thermal-cross-term-bound}
\end{align}
Since \(\psi'\) is bounded on the compact set \(K\), the first two
factors on the right are uniformly bounded over \(N\in K\), while the
last factor converges uniformly to zero by
\eqref{eq:thermal-taylor-remainder-L2}. Thus
\begin{equation}
    \sup_{N\in K}
    m
    \left|
        \E_N
        \left[
            \psi'(N)
            (\widehat N_m-N)
            r_{m,N}
        \right]
    \right|
    \rightarrow0.
    \label{eq:thermal-cross-term-vanishes}
\end{equation}

Expanding the square in \eqref{eq:thermal-functional-taylor}, and using
\eqref{eq:thermal-exact-second-moment},
\eqref{eq:thermal-taylor-remainder-L2}, and
\eqref{eq:thermal-cross-term-vanishes}, gives
\begin{align}
&
    m\E_N
    \left[
        \{\psi(\widehat N_m)-\psi(N)\}^2
    \right]
    \nonumber\\
&\quad=
    \{\psi'(N)\}^2
    m\E_N(\widehat N_m-N)^2
    +
    o(1)
    \nonumber\\
&\quad=
    \{\psi'(N)\}^2N(N+1)+o(1)
    \label{eq:thermal-plug-in-leading-term}
\end{align}
uniformly over \(N\in K\). This proves
\eqref{eq:uniform-thermal-plug-in-risk}.
\end{proof}

\begin{proof}[Proof of Lemma \ref{lem:nb-posterior-variance}]
The proof follows the successive steps of Lemmas~\ref{lem:multinomial-posterior-covariance} and~\ref{lem:multinomial-functional-posterior-variance}.
The principal identification is that, after the transformation
\[
    \vartheta=\frac{N}{N+1},
\]
the negative-binomial likelihood admits an exact Bernoulli
Kullback-Leibler representation. This supplies the posterior-tail
bound directly.

Define
\[
    T(N):=\frac{N}{N+1},
    \qquad
    F(\vartheta):=\frac{\vartheta}{1-\vartheta},
\]
so that \(F=T^{-1}\). Write
\[
    Q:=T(J)=[q_-,q_+],
    \qquad
    q_-:=\frac{a}{a+1},
    \qquad
    q_+:=\frac{b}{b+1}.
\]
Thus
\[
    0<q_-<q_+<1.
\]
Fix \(N_0\in(a,b)\), and put
\[
    \vartheta_0:=T(N_0)\in(q_-,q_+).
\]

In the parameter \(\vartheta\), the sampling distribution is
\[
    p_\vartheta^{(m)}(s)
    =
    \binom{s+m-1}{s}
    \vartheta^s(1-\vartheta)^m.
\]
Consequently, up to a factor independent of \(\vartheta\),
\[
    L_m(\vartheta)
    =
    \vartheta^S(1-\vartheta)^m,
\]
and
\begin{equation}
    \ell_m(\vartheta)
    :=
    \log L_m(\vartheta)
    =
    S\log\vartheta+m\log(1-\vartheta).
    \label{eq:nb-proof-loglik}
\end{equation}

The transformed prior density on \(Q\) is
\begin{equation}
    \widetilde\pi(\vartheta)
    =
    \pi\bigl(F(\vartheta)\bigr)F'(\vartheta)
    =
    \pi\left(\frac{\vartheta}{1-\vartheta}\right)
    \frac{1}{(1-\vartheta)^2}.
    \label{eq:nb-proof-transformed-prior}
\end{equation}
It is continuous on \(Q\), and there exist constants
\(0<\widetilde c_\pi\leq\widetilde C_\pi<\infty\) such that
\begin{equation}
    \widetilde c_\pi
    \leq
    \widetilde\pi(\vartheta)
    \leq
    \widetilde C_\pi,
    \qquad \vartheta\in Q.
    \label{eq:nb-proof-prior-bounds}
\end{equation}

Define $\widehat\vartheta_m:=S/(S+m)$.
This is the unrestricted maximum likelihood estimator over $[0,1)$.
On the event $G_m$ defined below, it lies in the interior of $Q$ and
therefore also maximizes the likelihood over $Q$. Note that
\[
    \widehat\vartheta_m
    =
    \frac{S}{S+m}
    =
    T(\widehat N_m),
    \qquad
    \widehat N_m:=\frac{S}{m}.
\]

Choose \(\varepsilon_0>0\) sufficiently small that
\[
    [\vartheta_0-4\varepsilon_0,\,
      \vartheta_0+4\varepsilon_0]
    \subset(q_-,q_+),
\]
and define
\begin{equation}
    G_m
    :=
    \left\{
        |\widehat\vartheta_m-\vartheta_0|
        \leq\varepsilon_0
    \right\}.
    \label{eq:nb-proof-good-event}
\end{equation}
On \(G_m\), the estimator \(\widehat\vartheta_m\) lies in a fixed
compact subinterval of \((q_-,q_+)\); in particular, \(S>0\).

Note that,
\[
    T'(x)=\frac{1}{(1+x)^2}\leq1,
    \qquad x\geq0,
\]
so that
\[
    |\widehat\vartheta_m-\vartheta_0|
    \leq
    |\widehat N_m-N_0|.
\]
The Chernoff bound therefore gives constants \(C_0,c_0>0\),
depending on \(N_0\) and \(\varepsilon_0\), such that
\begin{equation}
    \mathbb P_{N_0}(G_m^c)
    \leq
    \mathbb P_{N_0}
    \left(
        |\widehat N_m-N_0|>\varepsilon_0
    \right)
    \leq
    C_0e^{-c_0m}.
    \label{eq:nb-proof-good-event-tail}
\end{equation}

The final inequality follows by observing that
$S=X_1+\cdots+X_m,$
where \(X_1,\ldots,X_m\) are independent geometric random variables
with mean \(N_0\) and their common moment generating function
\[
    \mathbb E_{N_0}e^{tX_1}
    =
    \frac{1}{1+N_0-N_0e^t},
    \qquad
    t<\log\frac{1+N_0}{N_0}
\]
 is finite in a neighborhood of zero. 

On \(G_m\),
\[
    S=(S+m)\widehat\vartheta_m,
    \qquad
    m=(S+m)(1-\widehat\vartheta_m).
\]
Hence, for every \(\vartheta\in Q\),
\begin{align}
    \ell_m(\vartheta)
    -
    \ell_m(\widehat\vartheta_m)
    &=
    S\log\frac{\vartheta}{\widehat\vartheta_m}
    +
    m\log\frac{1-\vartheta}{1-\widehat\vartheta_m}
    \nonumber\\
    &=
    -(S+m)
    D_{\mathrm{Ber}}
    \bigl(\widehat\vartheta_m\Vert\vartheta\bigr),
    \label{eq:nb-proof-kl-identity}
\end{align}
where
\[
    D_{\mathrm{Ber}}(p\Vert q)
    :=
    p\log\frac pq
    +
    (1-p)\log\frac{1-p}{1-q}.
\]
Therefore,
\begin{equation}
    \frac{L_m(\vartheta)}
         {L_m(\widehat\vartheta_m)}
    =
    \exp\left\{
        -(S+m)
        D_{\mathrm{Ber}}
        \bigl(\widehat\vartheta_m\Vert\vartheta\bigr)
    \right\}.
    \label{eq:nb-proof-likelihood-ratio}
\end{equation}
This is the exact counterpart of the multinomial likelihood-ratio
identity used in Lemma~\ref{lem:multinomial-posterior-covariance} .

Introduce the local coordinate
\[
    h=\sqrt m\,(\vartheta-\widehat\vartheta_m).
\]

A direct computation shows that
\begin{equation}
    -\frac1m\ddot\ell_m(\widehat\vartheta_m)
    =
    \frac{S/m}{\widehat\vartheta_m^2}
    +
    \frac{1}{(1-\widehat\vartheta_m)^2}
    =
    \frac{1}
    {\widehat\vartheta_m(1-\widehat\vartheta_m)^2}.
    \label{eq:nb-proof-observed-information}
\end{equation}

Furthermore, the third derivative is
\[
    \ell_m^{(3)}(\vartheta)
    =
    \frac{2S}{\vartheta^3}
    -
    \frac{2m}{(1-\vartheta)^3}.
\]
On \(G_m\), the ratio \(S/m\) is uniformly bounded, and
\(\widehat\vartheta_m\) stays in a fixed compact subset of
\((0,1)\). Thus \(m^{-1}\ell_m^{(3)}\) is uniformly bounded in a
fixed neighborhood of \(\widehat\vartheta_m\).

Define
\[
    I_\vartheta(v)
    :=
    \frac{1}{v(1-v)^2},
    \qquad v\in(0,1).
\]

Taylor's theorem, together with
\(\dot\ell_m(\widehat\vartheta_m)=0\), now gives, for every
\(M<\infty\),
\begin{equation}
\begin{split}
    \sup_{|h|\leq M}
    \bigg|
        \ell_m\left(
            \widehat\vartheta_m+\frac{h}{\sqrt m}
        \right)
        -
        \ell_m(\widehat\vartheta_m)
        +
        \frac12
        I_\vartheta(\widehat\vartheta_m)h^2
    \bigg|
    \leq
    \frac{C_M}{\sqrt m}
\end{split}
    \label{eq:nb-proof-local-taylor}
\end{equation}
on \(G_m\), for all sufficiently large \(m\).

Since \(\widehat\vartheta_m\to\vartheta_0\) in probability,
\[
    I_\vartheta(\widehat\vartheta_m)
    \rightarrow
    I_\vartheta(\vartheta_0)
\]
in probability. Consequently,
\begin{equation}
\begin{split}
    \sup_{|h|\leq M}
    \bigg|
        \ell_m\left(
            \widehat\vartheta_m+\frac{h}{\sqrt m}
        \right)
        -
        \ell_m(\widehat\vartheta_m)
        +
        \frac12 I_\vartheta(\vartheta_0)h^2
    \bigg|
    \rightarrow0
\end{split}
    \label{eq:nb-proof-local-gaussian-expansion}
\end{equation}
in \(\mathbb P_{N_0}\)-probability.

Let \(\vartheta'\) denote the transformed posterior variable, and set
\[
    H_m
    :=
    \sqrt m\,(\vartheta'-\widehat\vartheta_m).
\]
Define the admissible local domain
\[
    \mathcal H_m
    :=
    \left\{
        h\in\mathbb R:
        \widehat\vartheta_m+\frac{h}{\sqrt m}\in Q
    \right\}.
\]
On \(G_m\), the posterior density of \(H_m\), conditional on \(S\),
is
\begin{equation}
    q_m(h\mid S)
    =
    \frac{g_m(h)}{Z_m},
    \qquad
    Z_m:=\int_{\mathbb R}g_m(h)\,dh,
    \label{eq:nb-proof-local-posterior}
\end{equation}
where
\begin{equation}
\begin{split}
    g_m(h)
    &:=
    \exp\left\{
        \ell_m\left(
            \widehat\vartheta_m+\frac{h}{\sqrt m}
        \right)
        -
        \ell_m(\widehat\vartheta_m)
    \right\}
    \\
    &\qquad\times
    \widetilde\pi\left(
        \widehat\vartheta_m+\frac{h}{\sqrt m}
    \right)
    1_{\mathcal H_m}(h).
\end{split}
    \label{eq:nb-proof-unnormalized-posterior}
\end{equation}

Define
\[
    g(h)
    :=
    \widetilde\pi(\vartheta_0)
    \exp\left\{
        -\frac12 I_\vartheta(\vartheta_0)h^2
    \right\}.
\]
Continuity of \(\widetilde\pi\), together with
\eqref{eq:nb-proof-local-gaussian-expansion}, yields
\begin{equation}
    \sup_{|h|\leq M}|g_m(h)-g(h)|
    \rightarrow0
    \label{eq:nb-proof-local-density-convergence}
\end{equation}
in probability, for every \(M<\infty\).

Pinsker's inequality for Bernoulli distributions gives
\[
    D_{\mathrm{Ber}}(p\Vert q)
    \geq
    2(p-q)^2.
\]
Combining this with~\eqref{eq:nb-proof-likelihood-ratio} and
\(S+m\geq m\), we obtain, on \(G_m\),
\begin{align}
    \frac{L_m(\vartheta)}
         {L_m(\widehat\vartheta_m)}
    &\leq
    \exp\left\{
        -2(S+m)
        (\vartheta-\widehat\vartheta_m)^2
    \right\}
    \nonumber\\
    &\leq
    \exp\left\{
        -2m(\vartheta-\widehat\vartheta_m)^2
    \right\}.
    \label{eq:nb-proof-global-likelihood-bound}
\end{align}
Therefore,
\begin{equation}
    \frac{
        L_m(\widehat\vartheta_m+h/\sqrt m)
    }{
        L_m(\widehat\vartheta_m)
    }
    \leq e^{-2h^2},
    \qquad h\in\mathcal H_m.
    \label{eq:nb-proof-local-likelihood-bound}
\end{equation}
The transformed prior is bounded above on \(Q\), so
\begin{equation}
    g_m(h)
    \leq
    \widetilde C_\pi e^{-2h^2}
    \label{eq:nb-proof-unnormalized-envelope}
\end{equation}
on \(G_m\).

Let
\[
    \varphi_{\vartheta_0}(h)
    :=
    \sqrt{\frac{I_\vartheta(\vartheta_0)}{2\pi}}
    \exp\left\{
        -\frac12 I_\vartheta(\vartheta_0)h^2
    \right\}.
\]

We now apply the normalization and posterior-moment argument used in
the proof of Lemma~\ref{lem:multinomial-posterior-covariance} that yielded \eqref{posterior-moment-conv}. In the present one-dimensional
setting, the required ingredients are the local convergence
\eqref{eq:nb-proof-local-density-convergence}, the bounds on the
transformed prior in \eqref{eq:nb-proof-prior-bounds}, and the
Gaussian envelope
\eqref{eq:nb-proof-unnormalized-envelope}. The same argument gives
\begin{equation}
    \int_{\mathbb R}
    (1+h^2)
    \left|
        q_m(h\mid S)-\varphi_{\vartheta_0}(h)
    \right|\,dh
    \rightarrow0
    \label{eq:nb-proof-weighted-bvm}
\end{equation}
in \(\mathbb P_{N_0}\)-probability.

More explicitly, consider the argument used in \eqref{eq:normalizing-lower-bound}. A similar argument gives a
constant \(c_Z>0\) such that, on \(G_m\) and for all sufficiently
large \(m\),
$Z_m\geq c_Z.$
Combining this lower bound with
\eqref{eq:nb-proof-unnormalized-envelope} yields
\begin{equation}
    q_m(h\mid S)
    \leq
    C e^{-2h^2},
    \qquad h\in\mathbb R,
    \label{eq:nb-proof-posterior-envelope}
\end{equation}
on \(G_m\), for a deterministic constant \(C<\infty\).
On $G_m$, the bounds
\begin{equation}
    \mathbb E\!\left[
        H_m^2\mid S
    \right]
    \leq C_2,
    \qquad
    \mathbb E\!\left[
        |H_m|^4\mid S
    \right]
    \leq C_4
    \label{eq:nb-proof-posterior-moment-bounds}
\end{equation}
hold for deterministic constants \(C_2,C_4<\infty\).

Equation~\eqref{eq:nb-proof-weighted-bvm} implies
\[
    \mathbb E[H_m\mid S]\rightarrow0,
\qquad
    \mathbb E[H_m^2\mid S]
    \rightarrow
    I_\vartheta(\vartheta_0)^{-1}
\]
in \(\mathbb P_{N_0}\)-probability. Consequently,
\begin{equation}
    \operatorname{Var}(H_m\mid S)
    \rightarrow
    I_\vartheta(\vartheta_0)^{-1}
    =
    \vartheta_0(1-\vartheta_0)^2
    \label{eq:nb-proof-posterior-variance-probability}
\end{equation}
in \(\mathbb P_{N_0}\)-probability.

It remains to pass from convergence in probability to convergence in
expectation. On \(G_m\),
\[
    \operatorname{Var}(H_m\mid S)
    \leq
    \mathbb E[H_m^2\mid S]
    \leq C_2
\]
by \eqref{eq:nb-proof-posterior-moment-bounds}. On \(G_m^c\), both
\(\vartheta'\) and \(\widehat\vartheta_m\) belong to \([0,1]\), and
therefore
\[
    |H_m|^2
    =
    m|\vartheta'-\widehat\vartheta_m|^2
    \leq m.
\]
It follows from \eqref{eq:nb-proof-good-event-tail} that
\begin{equation}
    \mathbb E_{N_0}
    \left[
        \operatorname{Var}(H_m\mid S)1_{G_m^c}
    \right]
    \leq
    m\,\mathbb P_{N_0}(G_m^c)\leq
    C_0m e^{-c_0m}
    \rightarrow0.
    \label{eq:nb-proof-bad-event-variance}
\end{equation}
Thus
\(\{\operatorname{Var}(H_m\mid S):m\geq1\}\) is uniformly
integrable. Combining this fact with
\eqref{eq:nb-proof-posterior-variance-probability} gives
\begin{equation}
    \mathbb E_{N_0}
    \left|
        \operatorname{Var}(H_m\mid S)
        -
        \vartheta_0(1-\vartheta_0)^2
    \right|
    \rightarrow0.
    \label{eq:nb-proof-posterior-variance-lone}
\end{equation}

Because \(\widehat\vartheta_m\) is fixed conditional on \(S\),
\[
    \operatorname{Var}(H_m\mid S)
    =
    m\operatorname{Var}(\vartheta'\mid S).
\]
Thus \eqref{eq:nb-proof-posterior-variance-lone} is precisely the
negative-binomial counterpart of
Lemma~\ref{lem:multinomial-posterior-covariance}.

We make the remainder argument in Lemma~\ref{lem:multinomial-functional-posterior-variance} explicit.
Note that by \eqref{eq:sample-mean-fourth-moment}
\[
   \mathbb E_{N_0}
    |\widehat\vartheta_m-\vartheta_0|^4\leq \mathbb E_{N_0}
    |\widehat N_m-N_0|^4
    =
    O(m^{-2}).
\]

On \(G_m\), the posterior envelope
\eqref{eq:nb-proof-posterior-envelope} gives, for all sufficiently
large \(m\),
\[
\begin{aligned}
    \mathbb E\!\left[|H_m|^4\mid S\right]
    &=
    \int_{\mathbb R}|h|^4 q_m(h\mid S)\,dh \\
    &\leq
    C\int_{\mathbb R}|h|^4e^{-2h^2}\,dh
    =:C_4<\infty.
\end{aligned}
\]
Thus the posterior fourth moment of \(H_m\) is uniformly bounded on
\(G_m\).
Using
\[
    \vartheta'-\vartheta_0
    =
    \frac{H_m}{\sqrt m}
    +
    (\widehat\vartheta_m-\vartheta_0)
\]
and \(|x+y|^4\leq8(|x|^4+|y|^4)\), we obtain
\begin{equation}
    \mathbb E_{N_0}
    \left[
        \mathbb E[
            |\vartheta'-\vartheta_0|^4
            \mid S
        ]1_{G_m}
    \right]
\leq
    \frac{8C}{m^2}
    +
    8\mathbb E_{N_0}
    |\widehat\vartheta_m-\vartheta_0|^4
    =
    O(m^{-2}). \label{fourth-moment-concentration}
\end{equation}
Since $|\vartheta'-\vartheta_0|^4\leq1$, \eqref{eq:nb-proof-good-event-tail} and \eqref{fourth-moment-concentration} therefore yield
\begin{equation}
    \mathbb E_{N_0}
    \mathbb E[
        |\vartheta'-\vartheta_0|^4
        \mid S
    ]
    =
    O(m^{-2}).
    \label{eq:nb-proof-posterior-fourth-moment}
\end{equation}

Define
\[
    g(\vartheta)
    :=
    \psi\bigl(F(\vartheta)\bigr)
    =
    \psi\left(\frac{\vartheta}{1-\vartheta}\right).
\]
The chain rule gives
\begin{equation}
    g'(\vartheta_0)
    =
    \frac{\psi'(N_0)}
         {(1-\vartheta_0)^2}.
    \label{eq:nb-proof-functional-derivative}
\end{equation}
Because \(Q\subset(0,1)\), Assumption~\ref{ass:thermal-smooth-functional} implies that \(g''\)
is bounded on \(Q\). Indeed,
\[
    g''(\vartheta)
    =
    \frac{\psi''(F(\vartheta))}
         {(1-\vartheta)^4}
    +
    \frac{2\psi'(F(\vartheta))}
         {(1-\vartheta)^3}.
\]
Write
\[
    B_g:=\sup_{\vartheta\in Q}|g''(\vartheta)|<\infty.
\]

Let \(\vartheta'\) and \(\vartheta''\) be conditionally independent
draws from the posterior given \(S\). Then
\begin{equation}
    m\operatorname{Var}\{g(\vartheta')\mid S\}
    =
    \frac m2
    \mathbb E
    \left[
        \{g(\vartheta')-g(\vartheta'')\}^2
        \mid S
    \right].
    \label{eq:nb-proof-two-draw-identity}
\end{equation}
Set
\[
    H_m'
    :=
    \sqrt m(\vartheta'-\widehat\vartheta_m),
    \qquad
    H_m''
    :=
    \sqrt m(\vartheta''-\widehat\vartheta_m).
\]
Taylor's theorem around \(\vartheta_0\) gives
\[
    g(u)
    =
    g(\vartheta_0)
    +
    g'(\vartheta_0)(u-\vartheta_0)
    +
    R(u),
\]
with
\[
    |R(u)|
    \leq
    \frac{B_g}{2}|u-\vartheta_0|^2,
    \qquad u\in Q.
\]
Subtracting the expansions for the two posterior draws,
\begin{equation}
    \sqrt m
    \{g(\vartheta')-g(\vartheta'')\}
    =
    g'(\vartheta_0)(H_m'-H_m'')
    +
    \widetilde r_m,
    \label{eq:nb-proof-two-draw-linearization}
\end{equation}
where
\[
    \widetilde r_m
    :=
    \sqrt m\{R(\vartheta')-R(\vartheta'')\}.
\]

By \(|x-y|^2\leq2x^2+2y^2\),
\[
    |\widetilde r_m|^2
    \leq
    \frac{mB_g^2}{2}
    \left(
        |\vartheta'-\vartheta_0|^4
        +
        |\vartheta''-\vartheta_0|^4
    \right).
\]
The two draws have the same posterior distribution, so
\eqref{eq:nb-proof-posterior-fourth-moment} implies
\begin{equation}
    \mathbb E_{N_0}
    \mathbb E[
        |\widetilde r_m|^2
        \mid S
    ]
    \leq
    mB_g^2
    \mathbb E_{N_0}
    \mathbb E[
        |\vartheta'-\vartheta_0|^4
        \mid S
    ]
    =
    O(m^{-1})
    \rightarrow0.
    \label{eq:nb-proof-functional-remainder}
\end{equation}

Let 
\[
    A_m:=H_m'-H_m'',
    \qquad
    d_0:=g'(\vartheta_0).
\]
Substitution of~\eqref{eq:nb-proof-two-draw-linearization} into
\eqref{eq:nb-proof-two-draw-identity} gives
\begin{equation}
    m\operatorname{Var}\{g(\vartheta')\mid S\}
    =
    \frac{d_0^2}{2}\mathbb E[A_m^2\mid S]
    +
    d_0\mathbb E[A_m\widetilde r_m\mid S]
    +
    \frac12\mathbb E[\widetilde r_m^2\mid S].
    \label{eq:nb-proof-functional-variance-expansion}
\end{equation}
Conditional independence gives
\begin{equation}
    \frac12\mathbb E[A_m^2\mid S]
    =
    \operatorname{Var}(H_m\mid S). \label{Am-Hm}
\end{equation}
Furthermore, conditional Cauchy-Schwarz followed by
Cauchy-Schwarz under \(\mathbb P_{N_0}\) yields
\begin{equation*}
    \mathbb E_{N_0}
    \left|\mathbb E[A_m\widetilde r_m\mid S]\right|\leq
    \left\{
        \mathbb E_{N_0}\mathbb E[A_m^2\mid S]
    \right\}^{1/2}
    \left\{
        \mathbb E_{N_0}
        \mathbb E[\widetilde r_m^2\mid S]
    \right\}^{1/2}.
\end{equation*}
The first factor on the RHS is bounded by
\eqref{eq:nb-proof-posterior-variance-lone} and \eqref{Am-Hm}, while the second tends
to zero by~\eqref{eq:nb-proof-functional-remainder}. Thus
\begin{equation}
\begin{split}
    \mathbb E_{N_0}
    \bigg|
        m\operatorname{Var}\{g(\vartheta')\mid S\}
        -
        \{g'(\vartheta_0)\}^2
        \operatorname{Var}(H_m\mid S)
    \bigg|
    \rightarrow0.
\end{split}
    \label{eq:nb-proof-functional-variance-linear}
\end{equation}
Together with~\eqref{eq:nb-proof-posterior-variance-lone}, this gives
\begin{equation}
\begin{split}
    \mathbb E_{N_0}
    \bigg|
        m\operatorname{Var}\{g(\vartheta')\mid S\}
        -
        \{g'(\vartheta_0)\}^2
        \vartheta_0(1-\vartheta_0)^2
    \bigg|
    \rightarrow0.
\end{split}
    \label{eq:nb-proof-functional-variance-transformed}
\end{equation}

Finally, \(g(\vartheta')=\psi(N')\), and by \eqref{eq:nb-proof-functional-derivative}
\begin{align*}
    \{g'(\vartheta_0)\}^2
    \vartheta_0(1-\vartheta_0)^2
    &=
    \{\psi'(N_0)\}^2
    \frac{\vartheta_0}{(1-\vartheta_0)^2}
    \\
    &=
    N_0(N_0+1)\{\psi'(N_0)\}^2
    \\
    &=
    V_\psi(N_0).
\end{align*}
This proves
\eqref{eq:nb-pointwise-posterior-functional-variance}.

Now define
\[
    B_m
    :=
    m\int_J
    \mathbb E_N
    \left[
        \operatorname{Var}\bigl(\psi(N')\mid S\bigr)
    \right]
    \Pi(dN).
\]
For every \(N\in(a,b)\), the pointwise result gives
\[
    m\mathbb E_N
    \left[
        \operatorname{Var}\bigl(\psi(N')\mid S\bigr)
    \right]
    \rightarrow
    V_\psi(N).
\]
Since \(\Pi\) has a density, its endpoints have zero prior mass.
The integrands are nonnegative, so Fatou's lemma gives
\begin{equation}
    \liminf_{m\to\infty}B_m
    \geq
    \int_J V_\psi(N)\,\Pi(dN).
    \label{eq:nb-proof-integrated-lower-bound}
\end{equation}

For the upper bound, let
\[
    \delta_m^{\mathrm B}(S)
    :=
    \mathbb E[\psi(N')\mid S]
\]
be the posterior mean. Under the joint distribution defined by the
prior and the sampling model,
\[
    \int_J
    \mathbb E_N
    \left[
        \operatorname{Var}\bigl(\psi(N')\mid S\bigr)
    \right]
    \Pi(dN)
\]
is the Bayes risk of \(\delta_m^{\mathrm B}\) under squared loss.

Next consider $\psi(S/m)$ as the estimator for $\psi(N)$. By the minimality of the Bayes risk we have 
\begin{align}
    \int_J
    \mathbb E_N
    \left[
        \operatorname{Var}\bigl(\psi(N')\mid S\bigr)
    \right]
    \Pi(dN)
    \leq
    \int_J
    \mathbb E_N
    \left[
        \{\psi(S/m)-\psi(N)\}^2
    \right]
    \Pi(dN).
    \label{eq:nb-proof-bayes-plug-in-comparison}
\end{align}
The RHS is finite
under Assumption~\ref{ass:thermal-smooth-functional}, since a bounded second derivative implies
at most quadratic growth of \(\psi\), while the negative-binomial
distribution has finite moments of every order.

Lemma~\ref{lem:thermal-uniform-plug-in-risk} gives
\[
    \sup_{N\in J}
    \left|
        m\mathbb E_N
        \left[
            \{\psi(S/m)-\psi(N)\}^2
        \right]
        -
        V_\psi(N)
    \right|
    \rightarrow0.
\]
Multiplying~\eqref{eq:nb-proof-bayes-plug-in-comparison} by \(m\)
and integrating this uniform expansion, we obtain
\begin{equation}
    \limsup_{m\to\infty}B_m
    \leq
    \int_J V_\psi(N)\,\Pi(dN).
    \label{eq:nb-proof-integrated-upper-bound}
\end{equation}
Combining~\eqref{eq:nb-proof-integrated-lower-bound} and
\eqref{eq:nb-proof-integrated-upper-bound} proves
\eqref{eq:nb-integrated-posterior-functional-variance}.

\end{proof}

\end{document}